\documentclass[11pt]{amsart}
\usepackage[latin1]{inputenc}
\usepackage{amsmath,amsthm}
\usepackage{amssymb,amsfonts}
\usepackage{hyperref}
\usepackage{tikz}

\usepackage[shortlabels]{enumitem}
\usepackage{bm}
\usepackage{cite}
\usepackage{graphicx}
\usepackage{verbatim}
\usepackage{setspace}
\usepackage{color}

\usepackage[a4paper]{geometry}

\hypersetup{bookmarks,final,colorlinks=true,
linkcolor=[rgb]{0.1, 0.1, 0.44},
citecolor=[rgb]{0.1, 0.1, 0.44}}

\usepackage{stmaryrd}
\DeclareSymbolFont{bbold}{U}{bbold}{m}{n}
\DeclareSymbolFontAlphabet{\mathbbm}{bbold}

\title[]{Another proof that $\MM^{++}$ implies Woodin's axiom $(*)$}

\author{Matteo Viale}
\thanks{
I thank Boban Veli\v{c}kovi\'c for suggesting the idea to relate Asper\'o and Schindler's forcing to consistency properties and
for giving me key hints for the proof of Lemma \ref{lem:keylemASPSCH(*)0} and Fact \ref{fac:treeTforP0}. 
I also thank Ilijas Farah and Ben de Bondt for the many useful comments.
I take full responsibility for the errors or gaps remaining.
}
\theoremstyle{plain}

	\newtheorem{Notation}{Notation}

	\newtheorem{theorem}{Theorem}[section]
	\newtheorem{proposition}[theorem]{Proposition}
	\newtheorem{lemma}[theorem]{Lemma}
	\newtheorem{corollary}[theorem]{Corollary}
	\newtheorem{fact}[theorem]{Fact}

	\newtheorem{claim}{Claim}
	\newtheorem{subclaim}{Subclaim}
\theoremstyle{definition}

	\newtheorem{definition}[theorem]{Definition}
	\newtheorem{notation}[theorem]{Notation}

\theoremstyle{remark}
	\newtheorem{remark}[theorem]{Remark}
	\newtheorem{strategy}[theorem]{Strategy}
	\newtheorem{convention}[theorem]{Convention}
	\newtheorem{warning}[theorem]{Warning}

\newcommand{\Ord}{\ensuremath{\mathrm{Ord}}}

\newcommand{\ZFC}{\ensuremath{\mathsf{ZFC}}}

\DeclareMathOperator{\dom}{dom}
\DeclareMathOperator{\ran}{ran}

\DeclareMathOperator{\otp}{otp}
\DeclareMathOperator{\cof}{cof}

\DeclareMathOperator{\trcl}{trcl}

\DeclareMathOperator{\Ult}{Ult}

\DeclareMathOperator{\Coll}{Coll}

\DeclareMathOperator{\Sat}{Sat}

\newcommand{\Pmax}{\ensuremath{\mathbb{P}_{\mathrm{max}}}}
\newcommand{\NS}{\ensuremath{\mathbf{NS}}} %non stationary ideal
\newcommand{\stUB}{\ensuremath{(*)\text{-}\mathsf{UB}}}

\newcommand{\bool}[1]{\mathsf{#1}}
\newcommand{\tow}[1]{\mathcal{#1}}
\newcommand{\SL}[1]{\Delta_{#1}}

\newcommand{\pow}[1]{\mathcal{P}\left(#1\right)}

\newcommand{\gp}[1]{\ulcorner #1 \urcorner}
\newcommand{\qp}[1]{\left[ #1 \right]}

\newcommand{\ap}[1]{\langle #1 \rangle}
\newcommand{\bp}[1]{\left\lbrace #1 \right\rbrace}

\newcommand{\Cod}{\ensuremath{\text{{\rm Cod}}}}
\newcommand{\ST}{\ensuremath{\text{{\sf ST}}}}
\newcommand{\UB}{\ensuremath{\text{{\sf UB}}}}

\newcommand{\BMM}{\ensuremath{\text{{\sf BMM}}}} 
\newcommand{\MM}{\ensuremath{\text{{\sf MM}}}}

\newcommand{\FA}{\ensuremath{\text{{\sf FA}}}}

\begin{document}

%%%%%%%%%%%%%%%%%%%%%%
%abstract-case-omega1
%%%%%%%%%%%%%%%%%%%%%%

\begin{abstract}
%We prove that Woodin's axiom $(*)$ can be forced by a semiproper forcing over any model of $\ZFC$
Let $\MM^{++}(\kappa)$ state that the forcing axiom $\MM^{++}$ can be instantiated only for stationary set preserving posets of size at most
$\kappa$. 
We give a detailed account of Asper\`o and Schindler's proof that 
 $\MM^{++}(\kappa)+$\emph{there are class many Woodin cardinals} implies Woodin's axiom
$(*)$ if $\Diamond_\kappa$ holds and $\kappa>\aleph_2$.
Our presentation takes advantage of the notion of consistency property: specifically we rephrase  
Asper\`o and Schindler's forcing as a specific instantiation of the notion 
of ``consistency property'' used by Makkai, Keisler, Mansfield and others in the study of infinitary logics. 
We also reorganize the order of presentation of the various parts of the proof. Taken aside these variations, 
our account is quite close to \cite{ASPSCH(*)}. 
\end{abstract}

\maketitle

I assume throughout this note that the reader is familiar with the key properties of stationary set preserving forcings,
knows the definition of $\MM^{++}$ and what are its most important applications, and is familiar with the theory
of $\Pmax$ (for example as exposed in \cite{HSTLARSON}).

I will try to make this paper as self-contained as possible. 
%but I think it is important rightaway to give an exact flavor of the 
%main result and outline why it does not conflict with Woodin's result that $\MM^{++}(\mathfrak{c})$ is compatible with the negation of $(*)$
%\cite[Thm. 10.90]{woodinBOOK}. 

What I will do is follow very closely Asper\`o and Schindler's proof, while reorganizing their presentation relating explicitly
 their forcing machinery to the notion of consistency property introduced by Makkai and others to produce models of theories in infinitary logics.
 
I give rightaway a fast account of the proof strategy of  Asper\`o and Schindler's result, more detais and precise definitions will follow later on.

 Following the proof pattern of Asper\`o and Schindler's result I set up
 their machinery in a slightly different terminology to design for each $A\in\pow{\omega_1}\setminus L(\mathbb{R})$
 and $D$ dense subset of $\Pmax$ universally Baire in the codes a stationary set preserving  forcing $P_{D,A}$ which instantiates the 
 $\Sigma_1$-sentence\footnote{$\Sigma_1$ in a signature admitting predicate symbols for all the universally Baire sets, for the predicate
 $\NS_{\omega_1}$ and for all properties defined by a $\Delta_0$-formula. We will be more precise later on.}:
 %\footnote{Simpler than the one considered by Asper\'o and Schindler  in \cite{ASPSCH(*)}.}
 
 \begin{quote}
 \emph{There is a $\NS$-correct iteration of some $(N,I,a)\in D$ which maps $a$ to $A$.}
  \end{quote}
%  for $D$ a dense subset of 
% $\Pmax$ universally Baire in the codes,
% $A\in \pow{\omega_1}\setminus L(\mathbb{R})$.

 Asper\'o and Schindler show that everything can be done assuming $V$ models 
 $\NS_{\omega_1}$\emph{ is saturated $+$ there are class many Woodin cardinals} with $P_{D,A}$ of size $\kappa$ 
 for any $\kappa$ such that:
  \begin{itemize}
\item
$\kappa\geq 2^{(2^{\aleph_1})}$;
\item
$\Diamond_\kappa$ holds.
\end{itemize}
Now if $\MM^{++}(\mathfrak{c})$ holds, 
%we get that $\Diamond_{\aleph_2}$ holds 
%(using Shelah's result on the equivalence between $\Diamond_{\lambda^+}$ and $2^\lambda=\lambda^+$),
%and 
$\NS_{\omega_1}$ is saturated, and $2^{\aleph_1}=\aleph_2$.
Moreover $\Diamond_\kappa$ can be established for $\kappa=\aleph_3$ by a stationary set preserving forcing of size $2^{\aleph_2}$ which does not add new subets of $\aleph_2$.
%hence $\mathfrak{c}=\kappa=\aleph_2$ is possible in models of
%$\MM^{++}(\mathfrak{c})$. By instantiating $\MM^{++}(\mathfrak{c})$ for the forcings $P_{D,A}$, 
Putting everything together, we get that $\MM^{++}(2^{\aleph_2})+$\emph{there are class many Woodin cardinals} implies the following 
\begin{quote}
\emph{For any 
$A\in \pow{\omega_1}\setminus L(\mathbb{R})$ such that
$L[A]$ computes correctly $\omega_1$, and any $D$ dense subset of $\Pmax$ universally Baire in the codes
there is a $\NS$-correct iteration of some $(N,I,a)\in D$ which maps $a$ to $A$.}
\end{quote} 
This is an even stronger version of $(*)$.

The following is the key consequence of the existence of class many Woodin cardinals we will need:
%Let us point out rightaway that this is not in conflict with 
%\cite[Thm. 10.90]{woodinBOOK}:
%the models for $\MM^{++}(\mathfrak{c})+\neg(*)$ produced in  \cite[Thm. 10.90]{woodinBOOK}
%do not satisfy the following key property which is at the heart of our proofs: 
%\begin{quote}
%\emph{For all universally Baire sets $B$ the set of $\Pmax$-conditions $(N,I,a)$ which are $B$-iterable is dense in $\Pmax$}.
%\end{quote}
\begin{quote}
\emph{
\[
(H_{\omega_1}^{V},\in,B)\prec (H_{\omega_1}^{V[G]},\in,B^{V[G]})
\]
for all $B$ universally Baire set in $V$ and all generic extensions $V[G]$ of $V$.
}
\end{quote}
This property holds assuming class many Woodin cardinals. %and fails in the models of $\MM^{++}(\mathfrak{c})+\neg(*)$ produced by Woodin. 
%We will also need some further consequences of universally Baire sets which are
%consistencywise not much weaker than class many Woodin (e.g. productiveness for the class of universally Baire sets, equivalently model completeness of the theory of $H_{\omega_1}$ in a signature with predicates for universally Baire sets, and generic invariance of this theory
%across forcing extensions).

The presentation is structured as follows:
\begin{itemize}
\item Section \ref{sec:background} gives the main definitions and a review of the results in the literature 
on which one leverages to establish that $\MM^{++}$ implies $(*)$.
This is just to make this paper completely self-contained. The reader familiar with $\Pmax$, $\MM^{++}$ 
etc. can skip this part;
proper references to the results presented in this section are in any case given at the appropriate stages
of this paper.
%and may refer to some of the results if needdirectly move to section \ref{sec:proof} where the original
%results are presented.  gives a brief outline of the background needed to infer Woodin's axiom $(*)$
%from the rsults of Section\ref{sec:proof}.

\item Section \ref{sec:semcert} states the main result (which is just rephrasing in our set up Schindler's and Asper\`o's main theorem)
while introducing the key concept of semantic 
certificate\footnote{Our notion of semantic certificate is slightly different from the one appearing in \cite{ASPSCH(*)}, but in spirit it has the same effects.}.
Roughly a semantic certificate for $A\in\pow{\omega_1}\setminus L(\mathbb{R})$ and 
$D\in L(\mathbb{R})$ dense subset of 
$\Pmax$ existing in some generic extension $V[G]$ of $V$ (possibly collapsing $\omega_1^V$ to countable) is a tuple 
$(\mathcal{J},r,f,T)$ 
such that $\mathcal{J}$ is an iteration 
of length $\omega_1^V$ of some $P_{\max}$ condition $(N,I,a)\in D^{V[G]}$ mapping $a$ to $A$ 
but
\emph{with no evident obstruction to the fact that $\mathcal{J}$ is $\NS$-correct and belongs to 
a stationary set preserving extension of $V$} (the other elements $r,f,T$ play a crucial role which will be outlined at the proper stage).

The main result
states that there is a stationary set preserving forcing $P_{D,A}$ of size $\kappa$ generically 
adding a semantic certificate $(\mathcal{J}_G,r_G,f_G,T_G)$ for $D,A$ and such that in any generic 
extension $V[G]$ of $V$ by $P_{D,A}$, $\mathcal{J}_G$ is really a $\NS$-correct iteration in $V[G]$.

In \ref{subsec:exsemcert}
we prove the first approximations to this result, i.e. the existence of semantic certificates\footnote{This corresponds to \cite[Lemma 3.2]{ASPSCH(*)}. 
Here we also pay attention to explain how to get the tree $T$ appearing in \cite[eqn. (8) pag. 16]{ASPSCH(*)} and its role in their proof.} in generic extensions collapsing $\omega_2^V$ to countable
(cfr. Lemma \ref{lem:keylemASPSCH(*)0} --- corresponding to \cite[Lemma 3.2]{ASPSCH(*)}). 
%where:
%\begin{itemize}
%\item
% $G_A$ the filter on $\Pmax$ induced by the $\NS$-correct iterations of 
%$\Pmax$-conditions $(N,I,a)$ mapping the selected $a$ to $A$,
%\item
%$D$ a dense subset of $\Pmax$ universally Baire in the codes.
%\end{itemize} 

\item Section \ref{sec:proof} proves the main result:
\begin{itemize}
%\item 
%Roughly for $A\in\pow{\omega_1}\setminus L(\mathbb{R})$ and $D$ dense subset of $\Pmax$
%universally Baire in the codes is an iteration existing in some generic extension of $V$ 
%without having at all introduced the relevant forcing notions of Asper\`o and Schindler
%
%We expand the sketchy argument given in their paper for establishing 
%\cite[Lemma 3.2]{ASPSCH(*)}. 
%We also introduce a property of iterations of $\Pmax$-conditions (that of being standard
%cfr. Def. \ref{def:standiter}) which we will need crucially in our version of the proof of the main theorem. 
%(cfr. the proof of Subclaim \ref{subclm:keyXXX}).
\item
In Section \ref{subsec:consprop} we introduce the notion of forcing given by consistency properties and establish some general facts about 
these types of forcings. Specifically a consistency property given by a class of structures $\mathcal{C}$
for a signature $\mathcal{L}$ is the family of finite sets of atomic (or negated atomic) 
$\mathcal{L}$-sentences which are realized by some structure in 
$\mathcal{C}$. Now this is naturally a forcing notion $P_{\mathcal{C}}$ ordered by reverse inclusion. When things are properly organized 
a generic filter for $P_{\mathcal{C}}$ produces a term model in $\mathcal{C}$.
\item
In Section \ref{subsec:P0} we introduce and analyze a specific instance of these forcings:
the forcing $P^*_0$ (which corresponds to the 
forcing $\mathbb{P}_0$ in the sequence of forcings $\bp{\mathbb{P}_\alpha:\alpha\leq\kappa}$ of \cite{ASPSCH(*)}) 
and show that it gives an elegant 
way to produce generic term-models 
whose transitive collapse are semantic certificates
%Specifically we analyze this forcing by means of infinitary logic and leveraging on the model theoretic notion 
%of \emph{consistency property} 
%to establish a correspondence relating density arguments for $P_0$ to properties of the term model of the theory 
%induced by a $V$-generic filter for $P_0$ 
(cfr. Fact \ref{fac:keyfacSEMCERTP0} -- corresponding to \cite[Lemma 3.3]{ASPSCH(*)}).
Unfortunately this forcing cannot be provably stationary set preserving since it has size continuum and Woodin has shown that $\MM^{++}(\mathfrak{c})$ is consistent with the negation of $(*)$
\cite[Thm. 10.90]{woodinbook}.
\item
In Section \ref{subsec:Pkappa} we introduce and analyze the forcings $\bp{P_\alpha:\alpha\in C\cup\bp{\kappa}}$
indexed by a club subset $C$ of $\kappa$
(corresponding to the sequence of forcings $\bp{\mathbb{P}_\alpha:\alpha\leq\kappa}$ of \cite{ASPSCH(*)}), 
and we prove the main Theorem \ref{thm:mainthmPkappa}
(corresponding to \cite[Lemma 3.8, Corollary 3.9]{ASPSCH(*)}) stating that $P_\kappa$ is stationary set preserving and  generically adds
a $\NS$-correct iteration witnessing that $G_A\cap D$ is non empty\footnote{$G_A$ is the $P_{\max}$-filter given by the conditions $(N,I,a)$ which correctly iterate $a$ to $A$, more details later.}. The forcings $P_\alpha$ are also consistency properties of the form $P_{\mathcal{W}_\alpha}$, but now $\mathcal{W}_\alpha$
is designed so that $P_{\mathcal{W}_\alpha}$ adds generically countable substructures of $H_\kappa^V$ which can be used
to seal $P_\alpha$-names for clubs\footnote{Roughly  these countable structures can be used to witness that
a given stationary set of $V$ on $\omega_1^V$ is met by a given $P_\alpha$-name for a club subset of $\omega_1^V$.}. $\Diamond_\kappa$ is then used to show that this sealing process
is completed at stage $\kappa$ as all $P_\kappa$-names for clubs are guessed at some stage $\alpha<\kappa$ and sealed by the trace on $P_\alpha$ of the generic filter for $P_\kappa$. This roughly outlines why
$P_\kappa$
should be stationary set preserving.
\end{itemize}
\end{itemize}

I decided to be overly cautious and to outline where the arguments require the existence of class many Woodin cardinals, and why they cannot go through with weaker assumptions. This might be annoying for some readers, I apologize for that.

Let us remark once more that:
\begin{enumerate}[(A)]
\item This work is just a re-elaboration of Asper\'o and Schindler's result, the key ideas are all appearing
in \cite{ASPSCH(*)}. Here we reformulate them in a slightly different terminology 
and give an alternative presentation of this proof.
\item This work wouldn't exist without 
Boban Velickovic's suggestion to relate Asper\'o and Schindler's work to the notion of 
\emph{consistency property}
used in infinitary logic (consistency properties are particularly useful if one wants to produce 
models of a certain first order theory omitting certain prescribed types). This is somewhat implicit in Asper\'o and Schindler's 
work and we make it explicit here.
\end{enumerate} 

It is also convenient before proceeding to introduce the following short-hand notation to denote 
$\mathcal{L}$-structures for a first order signature $\mathcal{L}$:
\begin{Notation}
Given a signature $\tau$, we denote an $\mathcal{L}$-structure $\mathcal{M}=(M,R^{\mathcal{M}}:R\in\tau)$ by
$(M,\tau^{\mathcal{M}})$.
\end{Notation}

Later on we will need a similar notation for multi-sorted structures (see Section \ref{subsec:P0}).

\section{Main result and background material on $\Pmax$, universally baire sets and $\MM^{++}$}\label{sec:background}

Let us first set up the proper language and terminology in order to deal with the $\Pmax$-technology and to state our main result.

$\ZFC^-$ denotes $\ZFC$ minus the powerset axiom.

\begin{definition}
A countable transitive set $M$ is  an iterable structure if it 
models $\ZFC^-+$\emph{there exists an uncountable cardinal} and for some transitive
$N\supseteq M$:
\begin{itemize}
\item
$N$ models $\ZFC+\NS_{\omega_1}$\emph{ is precipitous},
\item
$\omega_2^N$ is a countable ordinal in $V$,
\item
$\otp(N\cap\Ord)\geq\omega_1^V$,
\item 
$H_{\omega_2}^N=H_{\omega_2}^M$, 
\item
if $A\in N$ is a maximal
antichain of $(\pow{\omega_1}/_\NS)^M$, then $A\in M$.
\end{itemize}
\end{definition}

For example if $N$ is transitive, models $\ZFC$,
$\omega_2^N$ is a countable ordinal, and $\otp(N\cap\Ord)\geq\omega_1$, then
$M=H_{\omega_2}^N$ is an iterable structure if $N$ models $\NS_{\omega_1}$\emph{is saturated}; while
so is
$M=H_\kappa^N$ for any countable (in $V$) cardinal (of $N$) $\kappa>(2^{\aleph_2})^N$ in $N$ if $N$ just models 
$\NS_{\omega_1}$\emph{is precipitous}
(see \cite[Lemma 1.5]{HSTLARSON} for the ratio of this definition).

%\subsubsection*{Generic iterations of countable structures}

\begin{definition}\cite[Def. 1.2]{HSTLARSON}\label{def:iterstr}
Let $M$ be an iterable structure. 
Let $\gamma$ be an ordinal less than or equal to $\omega_1$. 
An iteration $\mathcal{J}$ of $M$ of length $\gamma$ 
consists of a family of transitive models $\ap{M_\alpha:\,\alpha \leq\gamma}$, 
sets $\ap{G_\alpha:\,\alpha< \gamma}$ 
and a commuting family of elementary embeddings 
\[
\ap{j_{\alpha\beta}: M_\alpha\to M_\beta:\, \alpha\leq\beta\leq\gamma}
\]
such that:
\begin{enumerate}
\item
$M_0 = M$,
\item
each $G_\alpha$ is an $M_\alpha$-generic filter for 
$(\pow{\omega_1}/\NS_{\omega_1})^{M_\alpha}$,
\item
each $j_{\alpha\alpha}$ is the identity mapping,
\item
each $j_{\alpha\alpha+1}$ is the ultrapower embedding induced by $G_\alpha$,
\item
for each limit ordinal $\beta\leq\gamma$,
$M_\beta$ is the direct limit of the system
$\bp{M_\alpha, j_{\alpha\delta} :\, \alpha\leq\delta<\beta}$, and for each $\alpha<\beta$, 
$j_{\alpha\beta}$ is the induced embedding.
%\item
%each $M_\alpha$ is transitive for all $\alpha\leq\gamma$.
\end{enumerate}

An iteration
\[
\mathcal{J}=\ap{j_{\alpha\beta}: M_\alpha\to M_\beta:\, \alpha\leq\beta\leq\omega_1}
\]
existing in $V$
is $\NS$-correct if
\[
M_{\omega_1}\cap \NS_{\omega_1}^V=
\NS_{\omega_1}^{M_{\omega_1}}.
\]
\end{definition}

From now we will just write \emph{correct} in place of $\NS$-correct.

We adopt the convention to denote an iteration $\mathcal{J}$ just
by $\ap{j_{\alpha\beta}:\, \alpha\leq\beta\leq\gamma}$, we also stipulate that
if $X$ denotes the domain of $j_{0\alpha}$, $X_\alpha$ or $j_{0\alpha}(X)$ denotes
the domain of $j_{\alpha\beta}$ for any $\alpha\leq\beta\leq\gamma$.

\begin{fact}\cite[Lemma 1.5, Lemma 1.6]{HSTLARSON}
Assume $(M,\in)$ is iterable. Then for any ordinal
$\gamma\leq\omega_1^V$ any iteration 
\[
\mathcal{J}=\ap{j_{\alpha\beta}: M_\alpha\to M_\beta:\, \alpha\leq\beta<\gamma}
\]
induced by filters $\ap{G_\alpha:\,\alpha< \gamma}$ 
is such that the structure $M_\gamma$ obtained in accordance with\footnote{I.e. $(M_\gamma,E_\gamma)$ is the direct limit of $\mathcal{J}$
if $\gamma$ is limit, or the ultrapower of $(M_\alpha,\in)$ by $G_\alpha$ if $\gamma=\alpha+1$, see \cite[Section 1]{HSTLARSON} for details.}
Def. \ref{def:iterstr} is well-founded hence isomorphic to its transitive collapse.
\end{fact}

A trivial fact is that for an iteration
\[
\mathcal{J}=\bp{j_{\alpha\beta}:\,N_\alpha\to N_\beta:\, \alpha\leq\beta\leq\gamma}
\]
and $\alpha<\beta$, $N_\alpha$ and $N_\beta$ are both transitive but in general it holds that
$N_\alpha\not\subseteq N_\beta$.

Notice that $M$ is iterable if and only if it is $\emptyset$-iterable. 

\begin{definition}\cite[Def. 2.1]{HSTLARSON}
$\Pmax$ is the subset of $H_{\omega_1}$ given by the pairs $(M,a)$ such that
\begin{itemize}
\item $M$ is iterable, countable, and models Martin's axiom. %that there are class many Woodin cardinals.
\item $a\in\pow{\omega_1}^M\setminus L(\mathbb{R})^M$, and there exists $r\in \pow{\omega}\cap M$ such that $\omega_1^M=\omega_1^{L[a,r]}$.
\end{itemize}

$(M,a)\leq (N,b)$ if there exists $\mathcal{J}=\ap{j_{\alpha\beta}:\, \alpha\leq\beta\leq\omega_1^M}$ in 
$M$ iteration of $N$ of length $\omega_1^M$ such that
$j_{0\omega_1^M}(b)=a$ and 
$(M,\in)$ models that $\mathcal{J}$ is correct.
%$j_{0\omega_1^M}(\NS_{\omega_1}^N)=\NS_{\omega_1}^m\cap j_{0\omega_1^M}[N]$.
\end{definition}

Note that $\Pmax$ is a definable class in $(H_{\omega_1},\in)$; in particular it belongs to any transitive model of 
$\ZFC^-$
containing $\pow{\omega}$.

Our definition of $\Pmax$ is slightly different  than the one given in \cite{HSTLARSON}, but it defines an equivalent of
the poset defined in \cite[Def. 2.1]{HSTLARSON} in view of the 
following\footnote{Much weaker large  cardinals 
assumptions are needed, we don't spell the optimal hypothesis.}:

\begin{fact}
Let $\Pmax^0$ be the forcing $\Pmax$ as defined in \cite[Def. 2.1]{HSTLARSON}.
Assume there are class many Woodin cardinals. 

Then for every condition $(M,I,a)$ in $\Pmax^0$ there is
a condition $(N,b)$ in $\Pmax$ and an iteration $\ap{j_{\alpha\beta}:\, \alpha\leq\beta\leq\omega_1^N}\in N$ of $(M,I)$ 
according to \cite[Def. 1.2]{HSTLARSON}
such that $j_{0\omega_1^N}(I)=\NS_{\omega_1}^N\cap j_{0\omega_1^N}[M]$ and $j_0(a)=b$.

Conversely assume $N$ is iterable and $\delta$ is Woodin in $N$. Let $G$ be 
$N$-generic for some $P\in N$
forcing $\NS_{\omega_1}$ is precipitous and Martin's axiom. 
Then $(H_{\kappa}^{N[G]},\NS_{\omega_1}^{N[G]},a)\in\Pmax^0$ for all 
$a\in\pow{\omega_1}^{N[G]}\setminus L(\mathbb{R})^{N[G]}$ such that there exists $r\in \pow{\omega}\cap N[G]$ with 
$\omega_1^M=\omega_1^{L[a,r]}$ with $\kappa=(2^{\aleph_1})^+$ in $N[G]$.
\end{fact}
\begin{proof}
Only the first part is non-trivial.
Let $\gamma>\delta$ be two Woodin cardinals.
Let $X\prec V_\gamma$ be countable with $\delta,(M,I,a)\in X$.
Let $N_0$ be the transitive collapse of $X$, and $N$ be a generic extension of $N_0$ by a forcing collapsing $\delta$
to become $\omega_2$ and forcing $\NS_{\omega_1}$ is precipitous and Martin's axiom. 
%Then $\NS_{\omega_1}^N$ is saturated in $N$ by $\MM^{++}$ in $N$.
Since $\gamma$ is Woodin, there are class many measurables in $V_\gamma$, hence $N_0$ is iterable 
and so is $N$
(by \cite[Thm. 4.10]{HSTLARSON}). 

By \cite[Lemma 2.8]{HSTLARSON} there is in $N$  the required iteration $\ap{j_{\alpha\beta}:\, \alpha\leq\beta\leq\omega_1^N}$
of $(M,I)$ and we can set $b=j_{0\omega_1^N}(a)$.
\end{proof}

In particular (at the prize of assuming the right large cardinal assumptions) the forcings 
$\Pmax$ and $\Pmax^0$ are equivalent.

Let from now on $\UB$ denote the class of universally Baire 
sets\footnote{See Section \ref{subsec:univbaire} below for details.}.

\begin{definition}
$(*)$-$\mathsf{UB}$ holds in $V$ if:
\begin{itemize}
\item
There are class many Woodin cardinals. 
\item
For any $A\in \pow{\omega_1}\setminus L(\UB)$ such that $\omega_1^{L[A]}=\omega_1^V$:
\begin{itemize}
\item the set
\begin{align*}
G_A=&\{(M_0,a)\in\Pmax: \,\exists\mathcal{J}\text{ correct iteration of }M_0\text{ such that }\\
&j_{0\omega_1}(a)=A\text{ and }\NS_{\omega_1}^{L(\UB)[G]}\cap M_{\omega_1}=\NS_{\omega_1}^{M_{\omega_1}}\}
\end{align*}
is a filter on $\Pmax$; %.
%is an $L(\mathbb{R})$-generic filter for $\Pmax$ and such that $L(\mathbb{R})[G_A]=L(\mathbb{R})[G]$.
%\item 
%There exists $G\in V$ filter on $\Pmax$ 
\item
$G_A$ meets all dense subsets of $\Pmax$ which are universally Baire 
in the codes\footnote{$D$ is universally Baire in the codes if there is a universally Baire 
set $B$ such that $D=\Cod[B]$ (See Def. \ref{def:cod}). Note that assuming the existence of class many Woodin cardinals any set of reals definable in $L(\mathbb{R})$ is universally Baire.};
\item
$\pow{\omega_1}\subseteq L(\mathbb{R})[A]$.
\end{itemize}
\item
$\NS_{\omega_1}$ is precipitous in $V$.
\end{itemize}
\end{definition}
%
%\begin{definition}
%$(*)$-$\mathsf{UB}$ holds in $V$ if:
%\begin{itemize}
%\item
%Martin's axiom holds in $V$.
%\item 
%There exists $G\in V$ filter on $\Pmax$ which meets all dense subsets of $\Pmax$ which are universally Baire 
%in the codes\footnote{$D$ is universally Baire in the codes if there is a universally Baire 
%set $B$ such that $D=\Cod[B]$ (See Def. \ref{def:cod}). Note that assuming the existence of class many Woodin cardinals any set of reals definable in $L(\mathbb{R})$ is universally Baire.}.
%\item
%$\NS_{\omega_1}$ is precipitous.
%\item
%There are class many Woodin cardinals. 
%\item
%$\pow{\omega_1}^V\subseteq L(\mathbb{R})[G]$.
%\end{itemize}
%\end{definition}
%

Note that the above is expressible as a schema of $\Pi_2$-sentences in the signature $\tau_{\NS_{\omega_1},\UB}$
which admits predicate symbols for all $\Delta_0$-properties, all universally Baire sets, and the non-stationary ideal $\NS_{\omega_1}$,
and constant symbols for $\omega_1,\omega,\emptyset$ (see Proposition \ref{prop:compl(*)}).
%(see Rsections \ref{subsec:signtauST}, \ref{subsec:univbaire},\ref{subsec:nonmodcompletHomega1}, \ref \ref{subsec:MM++BMM++} for details on this signature and ).

The rest of this paper will give a proof of the following remarkable result:
\begin{theorem}[Asper\'o, Schindler \cite{ASPSCH(*)}] \label{thm:mainresult}
Assume there are class many Woodin cardinals and $\MM^{++}(\kappa)$ holds
for $\kappa=2^{\aleph_2}$.
Then so does $(*)$-$\mathsf{UB}$.
\end{theorem}
%
%We will improve it to the following:
%\begin{theorem}\label{thm:mainresult}
%Assume there are class many Woodin cardinals and $\MM^{++}(\omega_2)$ holds.
%Then so does $(*)$-$\mathsf{UB}$.
%\end{theorem}

%Key to our proof is the following result:
%\begin{theorem}
%Assume there are class many Woodin cardinals.
%Then for any $A\in \pow{\omega_1}\setminus L(\mathbb{R})$ and $D$ dense subset of $\Pmax$ universally Baire in the codes\footnote{See Def. \ref{def:cod}}.
%\end{theorem}
Our proof is a variation of Asper\'o and Schindler's argument which reorganizes slightly their presentation while following its streamline.

%\subsection*{Generic absoluteness and universally Baire sets.}
%
%We will need two properties of the class of Universally Baire sets $\UB$:
%\begin{itemize}
%\item
%$\UB$ is closed under projections, complementation and finite unions.
%\item
%The first order theory of the structure
%\[
%(H_{\omega_1},\in,A:A\in\UB)
%\]
%is generically invariant.
%%\item
%%The first order theory of the structure
%%\[
%%(H_{\omega_1},\tau_\ST,A:A\in\UB)
%%\]
%%is model complete (where $\tau_\ST$ is  a signature admitting predicate symbols for all $\Delta_0$-properties).
%\end{itemize}
%
%Both properties  hold assuming there are class many Woodin cardinals.
%
%We will also repeatedly use Shoenfield's absoluteness Lemma. 
%Let us formulate it in the terms which are most convenient for us:

\subsection*{Strategy of the proof}
We will use the following result:
\begin{theorem}\cite[Thm. 5.1, Thm. 6.3, Thm 7.7]{HSTLARSON}
Assume in $V$ there are class many Woodin cardinals.
%The na all dense subsets of $\Pmax$ in $L(\mathbb{R})$ are universally Baire in the codes.
Let $G$ be $L(\mathbb{R})$-generic for $\Pmax$.
Then in $L(\mathbb{R})[G]$ it holds that:
\begin{itemize}
\item
$\psi_{\mathrm{AC}}$ and Martin's axiom hold and therefore, $2^{\omega}=2^{\omega_1}=\omega_2$.
\item
$\NS_{\omega_1}$ is saturated.
\item 
For any $A\in \pow{\omega_1}^{L(\mathbb{R})[G]}\setminus L(\UB)$
\begin{align*}
G_A=&\{(M_0,a)\in\Pmax: \,\exists\mathcal{J}\text{ correct iteration of }M_0\text{ such that }\\
&j_{0\omega_1}(a)=A\text{ and }\NS_{\omega_1}^{L(\UB)[G]}\cap M_{\omega_1}=\NS_{\omega_1}^{M_{\omega_1}}\}
\end{align*}
is an $L(\mathbb{R})$-generic filter for $\Pmax$ and such that $L(\mathbb{R})[G_A]=L(\mathbb{R})[G]$.
\end{itemize}
\end{theorem}

In particular to establish that $\stUB$ holds in $V$ it suffices to prove that:
\begin{enumerate}
\item
There are class many Woodin cardinals;
\item
$\NS_{\omega_1}^V$ is saturated;
\item
For some fixed $A$
in $\pow{\omega_1}^V\setminus L(\UB)$, $G_A$ is a filter for $\Pmax$ meeting all dense subsets of $\Pmax$ universally Baire in the codes, and
$\pow{\omega_1}^V\subseteq \pow{\omega_1}^{L(\mathbb{R})[G_A]}$ ($G_A$ will also be 
$L(\mathbb{R})$-generic since all sets of reals definable in $L(\mathbb{R})$
are universally Baire in the codes assuming class many Woodin cardinals).
\end{enumerate}

Assuming $\MM^{++}(\mathfrak{c})$ and the existence of class many Woodin cardinals, the first two conditions are automatically met.
The proof of \cite[Lemma 7.6]{HSTLARSON} (but not its statement) actually shows that if $V$ models that 
$\NS_{\omega_1}$ is saturated, then $G_A$ (as computed in $V$) is a filter on $\Pmax$ 
for any $A\in\pow{\omega_1}^V\setminus L(\mathbb{R})$ such that
$L[A]$ computes correctly $\omega_1$.

Clearly $L(\mathbb{R})[A]=L(\mathbb{R})[G_A]$.
Also $\MM^{++}(\mathfrak{c})$ entails that $\pow{\omega_1}^V\subseteq L(\mathbb{R})[A]$ 
(since for example it entails $\psi_{\mathrm{AC}}$ \cite[Section 6]{HSTLARSON} or Martin's axiom).
So the unique missing point is to check that $G_A$ is a
filter meeting all dense subsets of $\Pmax$ universally Baire in the codes
assuming $\MM^{++}$.

There is a natural strategy:
\begin{strategy}\label{strategy0}
%\textbf{STRATEGY \ref{strategy0}:}
Given $D\subseteq\Pmax$ dense and 
universally Baire in the codes and $A\in\pow{\omega_1}^V\setminus L(\UB)$, % and $(M,a)\in G_A$, 
find a forcing $P_{D,A}$ 
such that if $G$ is $V$-generic for $P_{D,A}$ in $V[G]$ there is an iteration
\[
\mathcal{K}=
\ap{k_{\alpha\beta}: M_\alpha\to M_\beta:\, \alpha\leq\beta\leq\omega_1^V}
\]
such that in $V[G]$:
\begin{enumerate}
\item
$(M_0,A\cap\omega_1^{M_0})\in D$, 
\item
$k_{0\omega_1^V}(A\cap\omega_1^{M_0})=A$,
\item 
$H_{\omega_2}^V\subseteq M_{\omega_1^V}$ and 
$\NS_{\omega_1}^{M_{\omega_1^V}}\cap V=\NS_{\omega_1}^V$,
\item 
$\NS_{\omega_1}^{M_{\omega_1^V}}=\NS_{\omega_1}^{V[G]}\cap M_{\omega_1^V}$.
\end{enumerate}
\end{strategy}

Assume the above is possible. Then $P_{D,A}$ is stationary set preserving since:
\[
\NS_{\omega_1}^V=\NS_{\omega_1}^{M_{\omega_1^V}}\cap V=
%\NS_{\omega_1}^{M_{\omega_1^V}}\cap H_{\omega_2}^V=
\NS_{\omega_1}^{V[G]}\cap M_{\omega_1^V}\cap H_{\omega_2}^V= 
\NS_{\omega_1}^{V[G]}\cap H_{\omega_2}^V=\NS_{\omega_1}^{V[G]}\cap V.
\]
%Let us call from now on an iteration $\mathcal{K}$ such that 
%\[
%k_{0\omega_1}(\NS_{\omega_1}^{M_0})=\NS_{\omega_1}^V\cap k_{0\omega_1}[M_0]
%\]
%a \emph{correct iteration}.
Now the statement 
\[
\exists\mathcal{K}\,[(\mathcal{K}\text{ is a correct iteration})\wedge ((M_0,a)\in D)\wedge k_{0\omega_1}(a)=A]
\]
is a $\Sigma_1$-formula in parameter $A$ in the language 
$\tau_\ST\cup\bp{B,\omega_1,\NS_{\omega_1}}$ 
where $\tau_\ST$ is a signature admitting predicate symbols for all
$\Delta_0$-functions and relations and $B$ is a universally Baire set coding $D$ in some absolute manner.  

If $P_{D,A}$ has size $\kappa\geq\mathfrak{c}$, and $\MM^{++}(\kappa)$ holds in $V$ this $\Sigma_1$-statement is reflected to $V$.

Let us first address the issues of defining properly:
\begin{itemize}
\item
What is the meaning of \emph{``$\tau_\ST$ is a signature admitting predicate symbols for all
$\Delta_0$-functions and relations and $B$ is a universally Baire set coding $D$ in some absolute manner''}.
\item 
What is the meaning of
\emph{``The formula 
\[
\exists\mathcal{K}\,[(\mathcal{K}\text{ is a correct iteration})\wedge ((M_0,a)\in D)\wedge k_{0\omega_1}(a)=A]
\]
is a $\Sigma_1$-formula in parameter $A$ in the language $\tau_\ST\cup\bp{B,\omega_1,\NS_{\omega_1}}$''}. 
\item
Why $\MM^{++}$ holding in $V$ entails that $\Sigma_1$-formulae
in parameter $A\in H_{\omega_2}$ for the language $\tau_\ST\cup\bp{B,\omega_1,\NS_{\omega_1}}$ holding in a
generic extension by a stationary set preserving forcing are reflected to $V$.
\end{itemize}

\subsection{The signature $\tau_\ST$}\label{subsec:signtauST}

\begin{notation}\label{not:keynotation}
\emph{}

\begin{itemize}
\item
$\tau_{\mathsf{ST}}$ is the extension of the first order signature $\bp{\in}$ for set theory 
which is obtained 
by adjoining predicate symbols
$R_\phi$ of arity $n$ for any $\Delta_0$-formula $\phi(x_1,\dots,x_n)$, 
and constant symbols for 
$\omega$ and $\emptyset$.
%\item
%$\tau_{\kappa}$ is the extension of $\tau_{\ST}$
%obtained by adding a constant symbol for $\kappa$.

\item $\ZFC^{-}$ is the
$\in$-theory given by the axioms of
$\ZFC$ minus the power-set axiom.

\item
$T_{\ST}$ is the $\tau_{\ST}$-theory given by the axioms
\[
\forall \vec{x} \,(R_{\forall x\in y\phi}(y,\vec{x})\leftrightarrow \forall x(x\in y\rightarrow R_\phi(y,x,\vec{x}))
\]
\[
\forall \vec{x} \,[R_{\phi\wedge\psi}(\vec{x})\leftrightarrow (R_{\phi}(\vec{x})\wedge R_{\psi}(\vec{x}))]
\]
\[
\forall \vec{x} \,[R_{\neg\phi}(\vec{x})\leftrightarrow \neg R_{\phi}(\vec{x})]
\]
%\[
%(\forall \vec{x}\exists!y \,R_{\phi}(y,\vec{x}))\rightarrow (\forall \vec{x}\,R_{\phi}(f_\phi(\vec{x}),\vec{x}))
%\]
for all $\Delta_0$-formulae $\phi(\vec{x})$, together with the $\Delta_0$-sentences
\[
\forall x\in\emptyset\,\neg(x=x),
\]
\[
\omega\text{ is the first infinite ordinal}
\]
(the former is an atomic $\tau_{\ST}$-sentence, the latter is expressible as the atomic sentence for 
$\tau_{\ST}$ stating that
$\omega$ is a non-empty limit ordinal and all its elements are successor ordinals or $0$).
\item
$\ZFC^-_{\ST}$ is the $\tau_{\ST}$-theory 
\[
\ZFC^{-}\cup T_{\ST},
\] 
accordingly we define $\ZFC_{\ST}$.
\end{itemize}
\end{notation}

In $\ZFC_{\ST}$ many absolute concepts (such as that of being a function) are now expressed
by an atomic formula, while certain more complicated ones 
(for example those defined by means of transfinite recursion over an absolute property, such as 
\emph{$x$ is the transitive closure of $y$}) can still be expressed by means
of $\ZFC^{-}_{\ST}$-provably $\Delta_1$-properties of $\tau_{\ST}$ (i.e. properties which are $\ZFC^{-}_{\ST}$-provably equivalent
 at the same time to a $\Pi_1$-formula
and to a $\Sigma_1$-formula), hence are still absolute between any two models (even non-transitive)
$\mathcal{M}$, $\mathcal{N}$ of $\ZFC_{\ST}$  of which one is a substructure of the other. 
On the other hand many definable properties have truth values which may vary depending on which model of 
$\ZFC_{\ST}$ we work in  (for example \emph{$\kappa$ is an uncountable cardinal} is a $\Pi_1\setminus\Sigma_1$-property in 
$\ZFC_{\ST}$ whose truth value may depend on the choice of the model of $\ZFC_{\ST}$ to which $\kappa$ belongs).

\subsection{Absolute codings of subsets of $H_{\omega_1}$ and Shoenfield's absoluteness} 	\label{subsec:nonmodcompletHomega1}

We define second order number theory as the first order theory of the structure
\[
(\mathcal{P}(\mathbb{N})\cup\mathbb{N},\in,\subseteq,=,\mathbb{N}).
\]

$\Pi^1_n$-sets (respectively $\Sigma^1_n$, $\Delta^1_n$)  are
the subsets of $\mathcal{P}(\mathbb{N})\equiv 2^{\mathbb{N}}$ 
defined by a $\Pi_n$-formula (respectively by a $\Sigma_n$-formula, at the same time 
by a $\Sigma_n$-formula and a $\Pi_n$-formula in the appropriate language), if the formula defining a set 
$A\subseteq (2^{\mathbb{N}})^n$ has some parameter
$\vec{r}\in (2^{\mathbb{N}})^{<\omega}$ we accordingly write that $A$ is $\Pi^1_n(\vec{r})$ 
(respectively $\Sigma^1_n(\vec{r})$, $\Delta^1_n(\vec{r})$).
if the formula defining a set 
$A\subseteq (2^{\mathbb{N}})^n$ uses an extra-predicate symbol $B\subseteq (2^\omega)^k$ we write that
$A$ is $\Pi^1_n(B)$ (respectively $\Sigma^1_n(B)$, $\Delta^1_n(B)$).

$A\subseteq (2^{\mathbb{N}})^N$ is projective if it is defined by some $\Pi^1_n$-property for some $n$.
Similarly we define the notion of being projective in $\vec{r}\in (2^{\mathbb{N}})^{<\omega}$ or 
$B\subseteq  (2^\omega)^k$.
\begin{remark}
$A\subseteq (2^{\mathbb{N}})^k$ is projective in some $\vec{r}\in (2^{\mathbb{N}})^{<\omega}$
 if and only if it is obtained by a Borel subset of $(2^{\mathbb{N}})^m$
by successive applications of the operations of projection on one coordinate and complementation.
\end{remark}

\begin{definition}\label{def:cod}
Given $a\in H_{\omega_1}$, $r\in 2^{\mathbb{N}}$ codes $a$, if (modulo a recursive bijection
of $\mathbb{N}$ with $\mathbb{N}^2$), $r$ codes a well-founded extensional relation on 
$\mathbb{N}$
whose transitive collapse is the transitive closure of $\{a\}$.

\begin{itemize}
\item
 $\mathrm{Cod}:2^{\mathbb{N}}\to H_{\omega_1}$ is the map assigning $a$ to $r$ if and only if $r$ codes $a$
and assigning the emptyset to $r$ otherwise.
\item
$\mathrm{WFE}$ is the set of $r\in 2^{\mathbb{N}}$ which (modulo a recursive bijection
of $\mathbb{N}$ with $\mathbb{N}^2$) are a well founded extensional relation on 
$\mathbb{N}$
whose transitive collapse is the transitive closure of $\{a\}$.

\end{itemize}
\end{definition}

The following are well known facts\footnote{See \cite[Section 25]{JECHST} and in particular the statement and proof of Lemma 25.25, which contains all ideas on which one can elaborate to draw the conclusions below.}.
\begin{remark} 
The map $\mathrm{Cod}$ is defined by a $\ZFC^-$-provably $\Delta_1$-property  (with no parameters)
over $H_{\omega_1}$ and is surjective. Moreover $\mathrm{WFE}$ is a $\Pi^1_1$-subset of $2^{\mathbb{N}}$. 
%(in particular a Universally Baire set).
Therefore if $N$ is a transitive model of $ZFC^-$ existing in some transitive model $W$ of $\ZFC$, 
$N$ computes correctly $\mathrm{Cod}$ and $\mathrm{WFE}$, i.e. $\mathrm{Cod}^N=\mathrm{Cod}^W\cap N$ and 
$\mathrm{WFE}^N=\mathrm{WFE}^W\cap N$.
%(Elaborate on the proof of \cite[Lemma 25.25]{JECHST}).
\end{remark}

\begin{lemma}
Assume $B\subseteq 2^\omega$.
Let $A\subseteq 2^{\mathbb{N}}$ be a $\Sigma^1_{n+1}(B)$ set. Then
$A$ is  $\Sigma_{n}$-definable in the structure $(H_{\omega_1},\tau_\ST^V,B)$ in the language $\tau_{\ST}\cup\bp{B}$.
Conversely if $A$ is $\Sigma_{n}$-definable in the structure $(H_{\omega_1},\tau_\ST^V,B)$ in the language $\tau_{\ST}\cup\bp{B}$,
then $\Cod^{-1}[A]$ is a $\Sigma^1_{n+1}(B)$ set.
\end{lemma}

\begin{definition}\label{def:canpairs}
Let $M,W$ be transitive models of $\ZFC^-$.
The pair $(M,W)$ is a \emph{canonical pair} if $H_{\omega_1}^M\subseteq H_{\omega_1}^W$ and $\otp(M\cap\Ord)\geq\omega_1^W$.
\end{definition}
Note that $M$ may not be a subset of $W$ (and this case will occur in our proofs).

\begin{lemma}\label{lem:SHOABS}
Assume $(M,W)$ is a \emph{canonical pair}. Then
\[
(H_{\omega_1}^M,\tau_{\ST}^M)\prec_1(H_{\omega_1}^W,\tau_{\ST}^W)
\] 
\end{lemma}
\begin{proof}
Use Shoenfield's absoluteness Lemma for $\Sigma^1_2$-properties 
\cite[Thm. 25.20]{JECHST}
between
the transitive models of $\ZFC^-$ $M\subseteq W$
and the observation that
$\Sigma_1$-definable subsets of $H_{\omega_1}$ corresponds to $\Sigma^1_2$-properties
(see \cite[Lemma 25.25]{JECHST}). 
\end{proof}

\subsection{Universally Baire sets}\label{subsec:univbaire}

%We use the following results on the theory of universally Baire sets:
%
%\begin{definition}
%Let $(V,\in)$ be a model of $\ZFC$ and $N\subseteq V$ be a transitive class (or set) which is a model of 
%$\ZFC^-$.
%$\mathcal{A}\subseteq N$ is $N$-closed
%if whenever
%$B\subseteq (2^\omega)^k$ is such that for some 
%$\in$-formula $\phi(x_0,\dots,x_n)$
%\[
%B=\bp{(r_0,\dots,r_{k-1})\in (2^\omega)^k:\, (N,\in,A_0,\dots,A_{n-k})\models\phi(r_0,\dots,r_{k-1},A_0,\dots,A_{n-k})}
%\]
%with $A_0,\dots,A_{n-k}\in\mathcal{A}$, we have that $B\in\mathcal{A}$.
%\end{definition}
%
%
%
%
%
%\begin{example}
%Given a model $(V,\in)$ of $\ZFC+$\emph{there are class many Woodin cardinals},
%simple examples of $H_{\omega_1}$-closed families are:
%\begin{enumerate}
%\item
%All subsets of $\pow{\omega^k}$ for $k\in\mathbb{N}$ (this is trivially true with no large cardinal assumptions).
%\item
%$\UB^V$, i.e. the family of \emph{all} universally Baire sets of $V$ 
%\cite[Thm. 3.3.3, Thm. 3.3.5, Thm. 3.3.6, Thm. 3.3.13, Thm. 3.3.14]{STATLARSON}.
%\item
%The family of subsets of $\pow{\omega^k}$ for $k\in\mathbb{N}$ definable in $(L(\mathbb{R}),\in)$
%among which the projective sets 
%\cite[Thm. 3.3.3, Thm. 3.3.5, Thm. 3.3.6, Thm. 3.3.9, Thm. 3.3.13, Thm. 3.3.14]{STATLARSON}.
%\item The family $\UB\cap X$ for some $X\prec V_\theta$ with $\omega_2>2^\omega$.
%\end{enumerate}
%\end{example}
%
%Moreover the following holds:

\begin{definition}\label{def:UBset}
Let $V$ be a transitive model of $\ZFC$.
$B\subseteq 2^\omega$ is universally Baire in $V$ if there are class sized trees $S_B,T_B$ on $(2\times\Ord)^{<\omega}$ such that:
\begin{itemize}
\item $S_B\cap (2\times\alpha)^{<\omega}$ and $T_B\cap (2\times\alpha)^{<\omega}$ are elements of $V$ for any ordinal $\alpha$;
\item $B=p[T_B]$ (i.e. $r\in B$ if and only if there is $f\in \Ord^\omega$ such that $(r\restriction n,f\restriction n)\in T_B$ for all $n\in\omega$);
\item $p[T_B]\cap p[S_B]=\emptyset$;
\item for all $G$ $V$-generic for some forcing notion $P\in V$, $p[T_B]^{V[G]}\cup p[S_B]^{V[G]}=(2^\omega)^{V[G]}$.
\end{itemize}
Trees $T$ and $S$ satisfying the two latter conditions for all $G$ $V$-generic for a specific forcing notion $P$ are said to be projecting to complements for $P$.

\end{definition}

\begin{notation}
Given a universally Baire set $B$ and $G$ $V$-generic for some forcing notion $P$ of $V$, we let $B^{V[G]}= p[T_B]^{V[G]}$.
\end{notation}

\begin{theorem}\label{thm:genabshomega1}
Assume in $V$ there are class many Woodin cardinals. Let $\UB^V$ be the 
family of universally Baire sets of $V$ and $\tau_{\UB^V}=\tau_{\ST}\cup\UB^V$. 
Let $G$ be $V$-generic for some forcing notion $P\in V$ and $H$ be $V[G]$-generic for some forcing 
$Q$ in $V[G]$.

Then 
\[
(H_{\omega_1}^{V[G]},\tau_{\ST}^{V[G]},A^{V[G]}:A\in\UB^V)\prec
(H_{\omega_1}^{V[G][H]},\tau_{\ST}^{V[G]},A^{V[G][H]}:A\in\UB^V).
\]
\end{theorem}

For a proof of this theorem see for example \cite[Thm. 4.7]{VIATAMSTI}.

In particular assume $B$ is universally Baire and such that
\[
(H_{\omega_1},\tau_{\ST}^V, B)\models D=\Cod[B]\text{ is a dense subset of }\Pmax.
\]
Then 
\[
(H_{\omega_1}^{V[G]},\tau_{\ST}^{V[G]},B^{V[G]})\models \Cod[B^{V[G]}]\text{ is a dense subset of }\Pmax
\]
and $\Cod[B^{V[G]}]\cap H_{\omega_1}^V=D$.

\begin{definition}
$A\subseteq H_{\omega_1}$ is \emph{universally Baire in the codes} 
if $A=\Cod[B]$ for some universally Baire set $B$.
\end{definition}

%Recall that a first order theory $T$ for a signature $\tau$ is model complete if 
%whenever $\mathcal{M}\sqsubseteq\mathcal{N}$ are $\tau$-structures which model $T$, we have that
%$\mathcal{M}\prec\mathcal{N}$. 

\begin{notation}
Given $\mathcal{A}$ a family of universally Baire sets of $V$,
$\tau_{\mathcal{A}}=\tau_\ST\cup\mathcal{A}$ and
$T_{\mathcal{A}}$ is the theory of the $\tau_{\mathcal{A}}$-structure
\[
(H_{\omega_1}^V,\tau_{\ST},A: A\in \mathcal{A}).
\]
\end{notation}
%The following result will also play an important role in the proof:
%
%\begin{theorem}\label{thm:modcompHomega1}
%\cite[Def 1.1, Thm. 1.3]{VIATAMSTI}
%Assume there are class many Woodin cardinals\footnote{Note that this assumption is used just to argue that $\UB^X=\UB^V\cap X$
%is $H_{\omega_1}^X$-closed. The result holds in general for any family of subsets of $(2^\omega)^k$ which is
%$H_{\omega_1}^X$-closed.}. 
%Let $X\prec H_{\omega_2}$ for some $\omega_2>2^{\aleph_0}$. Then $\UB^X=\UB^V\cap X$ is such that
%$T_{\UB^X}$ is model complete.
%\end{theorem}

%
% be the family of lightface definable projective sets (i.e. the subsets of $(2^{\omega})^k$ for some
%$k$ which are the extension in $H_{\omega_1}$ of an $\in$-formula $\phi(x_1,\dots,x_k)$ in displayed free variable without parameters)
%Let $T_{\lUB}$ be $\tau_{\ST}\cup\lUB$-theory of the structure
%\[
%(H_{\omega_1}^V,\tau_{\ST},A: A\in \lUB.
%\]
%Then $T_{\lUB}$ is model complete.
%\end{theorem}

\subsection{$\MM^{++}$ and the reflection of $\Sigma_1$-sentences in $\tau_{\ST}\cup\UB\cup\bp{\omega_1,\NS_{\omega_1}}$}
\label{subsec:MM++BMM++}

We now show that $\MM^{++}$ entails a strong form of $\BMM$ which reflects not only $\Sigma_1$-properties, but
also $\Sigma_1$-properties expressed in the first order language expanding $\tau_\ST$ with
predicate symbols for $\NS_{\omega_1}$ and universally Baire sets. Let 
\[
\tau_{\NS_{\omega_1},\UB}=\tau_{\ST}\cup\UB\cup\bp{\omega_1,\NS_{\omega_1}}.
\]

Given a Woodin cardinal $\delta$, $\tow{T}_\delta$ denotes the full stationary tower of height $\delta$
(denoted as $\mathbb{P}_\delta$ in \cite{STATLARSON}) and for $K$ $V$-generic for
$\tow{T}_\delta$ $\Ult(V,K)$ denotes the generic ultrapower and $j_K:V\to\Ult(V,K)$ the generic ultrapower umbedding
(see \cite[Chapter 2]{STATLARSON} for the key definitions and results).

By \cite[Thm. 2.12]{VIAMMREV} assuming the existence of class many Woodin cardinals
$\MM^{++}(\kappa)$ can be formulated as follows:

\begin{definition}
$\MM^{++}(\kappa)$ holds if
for any (some9 Woodin cardinal $\delta>\kappa$ and for
any stationary set preserving forcing $P$ of size $\kappa$, 
there is a stationary set $T_P\in V_\delta$ and
a complete embedding
$i:P\to\tow{T}_\delta\restriction T_P$ such that
whenever $G$ is $V$-generic for $P$:
\begin{itemize}
\item 
The quotient forcing $\tow{T}_\delta\restriction T_P/_{i[G]}$ is stationary set preserving in $V[G]$,
\item
Letting $H$ be $V[G]$-generic for $\tow{T}_\delta\restriction T_P/_{i[G]}$ and $K$ be the
$V$-generic filter for $\tow{T}_\delta\restriction T_P$ induced by the forcing isomorphism of
$\tow{T}_\delta\restriction T_P$ with $P\ast(\tow{T}_\delta\restriction T_P/_{i[\dot{G}]})$, 
we have that
the critical point of $j_K:V\to\Ult(V,K)$ is $\omega_2$, $j_K(\omega_2)<\delta=j_K(\delta)$ and
$\Ult(V,K)^{<\delta}\subseteq\Ult(V,K)$ holds in $V[K]$.
\end{itemize}
\end{definition}

\begin{theorem}
Assume $\MM^{++}(\kappa)$ holds in $V$ and there are class many Woodin cardinals. 
Then for any stationary set preserving forcing $P\in V$ of size $\kappa$ and any $G$
$V$-generic for $P$
\[
(H_{\omega_2}^V,\tau_{\ST}^V,\NS_{\omega_1}^V,\UB^V)\prec_1 
(H_{\omega_2}^{V[G]},\tau_{\ST}^{V[G]},\NS_{\omega_1}^{V[G]},A^{V[G]}:A\in \UB^{V}).
\]
\end{theorem}

We sketch a proof:
\begin{proof}
Assume $\phi(x,y)$ is a quantifer free formula in the signature 
$\tau_{\NS_{\omega_1},\UB}$.
Let $B_1,\dots,B_k\in\UB^V$ be the universally Baire sets appearing in the formula
$\phi(x,y)$.
Fix $A\in H_{\omega_2}^V$, and assume that $P$ is stationary set preserving of size $\kappa$ 
and forces
$\exists y\phi(A,y)$.

Let $K$ be $V$-generic for $\tow{T}_\delta\restriction T_P$ and $G,H\in V[K]$ be such that
$G$ is $V$-generic for $P$ and $H$ is $V[G]$-generic for $\tow{T}_\delta\restriction T_P/_{i[G]}$
and $V[K]=V[G][H]$.

Now:
\begin{itemize}
\item 
$H_{\omega_2}^{V[G][H]}=H_{\omega_2}^{\Ult(V,K)}$ and 
$j_K(\NS_{\omega_1}^V)=\NS_{\omega_1}^{\Ult(V,K)}=\NS_{\omega_1}^{V[G][H]}$:
since $V[G][H]=V[K]$,
$j(\omega_2)<\delta$ and $V[K]$ models that $\Ult(V,K)^{<j(\delta)}\subseteq\Ult(V,K)$.
\item
$j_K(B)=B^{V[G][H]}$ for all universally Baire sets $B\in V$: assume $S_B,T_B\in V$ are trees on 
$(2\times\Ord)^{<\omega}$ 
projecting to complements in any forcing extension %of size less than $\kappa$ for some $\kappa$ bigger than $\delta$ 
with $B=p[T_B]$. 
Then $j_K(S_B),j_K(T_B)$ in $\Ult(V,K)$ are trees projecting to complements containing
$j_K[S_B],j_K[T_B]$; therefore if $p$ is the projection map of $(2\times\Ord)^{\omega}$ on $2^\omega$, 
\[
B^{V[K]}=p[T_B]=p[j_K[T_B]]=p[j_K(T)]\subseteq j_K(B),
\]
\[
(2^\omega\setminus B)^{V[K]}=p[S_B]=p[j_K[S_B]]\subseteq p[j_K(S_B)]=(2^\omega\setminus B)^{\Ult(V,K)}.
\]
Hence $B^{V[K]}=j_K(B)$.
\end{itemize}
This gives that the identity map defines an elementary embedding of the $\tau_{\NS_{\omega_1},\UB^V}$-structure
\[
(H_{\omega_2}^{V},\tau_{\ST}^{V},\NS_{\omega_1}^{V},\UB^{V})
\]
into the $\tau_{\NS_{\omega_1},\UB^V}$-structure
\[
(H_{\omega_2}^{V[G][H]},\tau_{\ST}^{V[G][H]},\NS_{\omega_1}^{V[G][H]},B^{V[G][H]}:B\in \UB^{V}).
\]
Since $H$ is $V[G]$-generic for a stationary set preserving forcing in $V[G]$
and $B^{V[G][H]}\cap V[G]=B^{V[G]}$ for all universally Baire sets in $V$, we get that
\[
(H_{\omega_2}^{V[G]},\tau_{\ST}^{V[G]},\NS_{\omega_1}^{V[G]},B^{V[G]}:B\in \UB^{V})
\]
is a substructure of
\[
(H_{\omega_2}^{V[G][H]},\tau_{\ST}^{V[G][H]},\NS_{\omega_1}^{V[G][H]},B^{V[G][H]}:B\in \UB^{V}).
\]
By choice of $P$
\[
(H_{\omega_2}^{V[G]},\tau_{\ST}^{V[G]},\NS_{\omega_1}^{V[G]},B^{V[G]}:B\in \UB^{V})\models\exists y\phi(A,y).
\]
Therefore since $\Sigma_1$-properties are preserved by superstructures
\[
(H_{\omega_2}^{V[G][H]},\tau_{\ST}^{V[G][H]},\NS_{\omega_1}^{V[G][H]},B^{V[G][H]}:B\in \UB^{V})\models 
\exists y\phi(A,y).
\]
We conclude that 
\[
(H_{\omega_2}^{V},\tau_{\ST}^{V},\NS_{\omega_1}^{V},\UB^{V})\models 
\exists y\phi(A,y).
\]
The proof is completed.
\end{proof}

\subsection{$G_A\cap D$ is a 
$\Sigma_1$-property in the signature $\tau_{\ST}\cup\UB\cup\bp{\omega_1,\NS_{\omega_1}}$}

%\subsection{Why $G_A\cap D\neq\emptyset$ for $A\in \pow{\omega_1}\setminus L(\UB)$ and
%$D$ dense subset of $\Pmax$ universally Baire in the codes is a $\Sigma_1$-statement for
%$\tau_{\ST}\cup\UB\cup\bp{\omega_1,\NS_{\omega_1}}$}

The statement 
\[
x=\bp{j_{\alpha\beta}:\,M_\alpha\to M_\beta: \,\alpha\leq\beta\leq\omega_1}
\text{ is an iteration of length $\omega_1$}
\]
is an atomic formula in free variable $x$ for the signature 
$\tau_{\ST}\cup\bp{\omega_1}$ given by the conjuctions of:
\begin{enumerate}
\item
\emph{$x$ is a function},
\item
\emph{$\dom{x}=\omega_1$},
\item
\emph{$\forall\alpha\in\omega_1$ $x(\alpha)$ is an ordered pair $(M_\alpha,G_\alpha)$ such that:}
\begin{itemize}
\item
\emph{$M_\alpha$ is a transitive set},
\item
\emph{$(M_\alpha,\in)$ models $\ZFC+\NS_{\omega_1}$ is precipitous},
\item
\emph{$G_\alpha$ is $M_\alpha$-generic for 
$\pow{\omega_1}^{M_\alpha}/_{\NS_{\omega_1}^{M_\alpha}}$,}
\item
\emph{$M_{\alpha+1}=\Ult(M_\alpha,G_\alpha)$ where  $\Ult(M_\alpha,G_\alpha)$ is the 
generic ultrapower induced by $G_\alpha$ on $M_\alpha$},
\item
\emph{if $\alpha$ is a limit ordinal, $M_{\alpha}$ is the direct limit of the iteration $x\restriction\alpha$}.
\end{itemize}
\end{enumerate}
We leave to the reader to check that the above properties are formalizable by $\Delta_0$-formulae.
\[
x=\bp{j_{\alpha\beta}:\,M_\alpha\to M_\beta:\,\alpha\leq\beta\leq\omega_1}
\text{ is a correct iteration}
\]
is the $\tau_\ST\cup\bp{\omega_1,\NS_{\omega_1}}$-atomic formula given by the conjuctions of
\begin{itemize}
\item
\emph{$x$ is an iteration of length $\omega_1$} 
\item
$\forall z [z\subseteq\omega_1\wedge z\in M_{\omega_1}\rightarrow
[(M_{\omega_1}\models z\text{ is stationary})\leftrightarrow\neg\NS_{\omega_1}(z))]]$.
\end{itemize}

Given $A\in\pow{\omega_1}\setminus L(\UB)$ and $D$ dense subset of $\Pmax$
universally Baire in the codes with
$D=\Cod[B]$ and $B\in\UB$, $G_A\cap D$ is non-empty is given by the 
$\tau_\ST\cup\bp{A,\omega_1,\NS_{\omega_1},B}$
$\Sigma_1$-formula:
\[
\exists x\exists r\,[(x\text{ is a correct iteration of length }\omega_1)\wedge B(r)\wedge 
\Cod(r)=(M_0,A\cap \omega_1^{M_0})\wedge j_{0\omega_1}(A\cap \omega_1^{M_0})=A.]
\]

\begin{proposition}\label{prop:compl(*)}
In any model of $\ZFC+$\emph{there are class many Woodin cardinals$+\NS_{\omega_1}$ is precipitous}
$\stUB$ is expressible by an axiom schema of 
$\Pi_2$-sentences  for the signature $\tau_{\NS_{\omega_1},\UB}$.

This axiom system is given by the $\Pi_2$-axioms of 
$\ZFC^-_\ST$ enriched with the $\Pi_1$-sentence for $\tau_\ST\cup\bp{\omega_1}$
\[
\forall f \,\qp{(f \text{ is a function})\wedge (\dom(f)=\omega)}\rightarrow (\ran(f)\neq\omega_1),
\]
the $\Pi_2$-sentence for $\tau_\ST\cup\bp{\omega_1,\NS_{\omega_1}}$
\[
\forall S\,\qp{\NS_{\omega_1}(S)\leftrightarrow \exists C \,\qp{(C\text{ is a club subset of $\omega_1$})\wedge (C\cap S=\emptyset)}}
\]
and 
(for any universally Baire set $B$ such that $\Cod[B]=D$ is a dense subset of $\Pmax$) 
 the 
$\Pi_2$-sentences  for $\tau_\ST\cup\bp{\omega_1,\NS_{\omega_1},B}$
\begin{align*}
\forall A &[\\
&\qp{(\omega_1\text{ is the first uncountable cardinal})^{L[A]}\wedge \forall r \forall\beta\, 
\qp{(r\subseteq\omega\wedge\beta\text{ is an ordinal})\rightarrow A\notin L_\beta[r])}}\\
&\rightarrow \\
&\exists (N_0,a_0)\,\exists \mathcal{J}\exists r
(B(r)\wedge \Cod(r)=(N_0,a_0)\wedge \mathcal{J} \text{ is a correct iteration of $N_0$ with $j_{0\omega_1}(a_0)=A$})\\
&]
\end{align*}
\end{proposition}
\begin{proof}
The unique thing to check is that \emph{$\NS_{\omega_1}$} is saturated follows from the above axioms, but this holds by 
\cite[Thm. 5.1]{HSTLARSON}.
\end{proof}

\section{Main result and semantic certificates}\label{sec:semcert}

The key concept introduced by Asper\`o and Schindler is that of a semantic certificate.

Given a dense set $D$ of $\Pmax$ universally Baire in the codes and
$A\in\pow{\omega_1^V}\setminus L(\mathbb{R}^V)$
it is not at all transparent how one can produce 
(even in an arbitrary forcing extension $V[G]$ of $V$ possibly not stationary set  preserving)
an iteration 
\[
\mathcal{J}=\bp{j_{\alpha\beta}:N_\alpha\to N_\beta:\, \alpha\leq\beta\leq\omega_1^V}
\] 
respecting the constraints 
$(N_0,A\cap\omega_1^{N_0})\in D^{V[G]}$ and
$j_{0\omega_1^V}(A\cap\omega_1^{N_0})=A$. Semantic certificates (to be defined below) 
impose even 
stronger constraints on $\mathcal{J}$.

\begin{convention}
Throughout the rest of this paper we assume $\NS_{\omega_1}$ is saturated, Martin's axiom, and 
$2^{\aleph_0}=2^{\aleph_1}=\aleph_2$ holds in $V$. 
\end{convention}
Note that all these assumptions hold assuming $\MM^{++}$ and are also a consequence of $\stUB$.

To introduce the notion of semantic certificate 
we must first select the class of transitive models of $\ZFC$ and 
the type of iterations of countable structures we will consider in the remainder of this paper, 
these are the canonical pairs introduced in Def. \ref{def:canpairs}.

%\begin{definition}
%Let $M,W$ be transitive models of $\ZFC$.
%The pair $(M,W)$ is a \emph{canonical pair} if $H_{\omega_1}^M\subseteq H_{\omega_1}^W$.
%\end{definition}

\begin{remark}
We will be interested in the following three combinations of canonical pairs (cfr. Def. \ref{def:canpairs}) $M,W$:
\begin{itemize}
\item $W=M[G]$ is a generic extension of $M$.
\item $M$ is the standard universe of sets $V$ and $W=\bar{M}$, where $j:V\to\bar{M}$ is an elementary embedding definable in $V[G]$
with critical point $\omega_1^V$ for $V[G]$ a generic extension of $V$.
\item $W=V[G][H]$ is a generic extension of $V$ with $G$ $V$-generic for $\Coll(\omega,\omega_2^V)$, $H$
$V[G]$-generic for $\Coll(\omega,\eta)$ with $\eta$ large enough, $\bar{j}:V\to\bar{M}$ an elementary embedding definable in $V[G]$ with critical point 
$\omega_1^V$ and $\omega_1^{\bar{M}}=\omega_1^{V[G]}$, and $M=\bar{M}[H]$.
\end{itemize}

Note that in the second case we can just assume that $H_{\omega_1}^V\subseteq \bar{M}$, in general the transitive classes $\bar{M},V$ definable in 
$V[G]$ do not have much else in common.
\end{remark}

\begin{definition}\label{def:witnessingtree}
Let: 
\begin{itemize}
\item 
$V$ be a transitive model of $\ZFC+\NS_{\omega_1}$\emph{is saturated},
\item
$B$ be a universally Baire set in $V$ and $D$ in $V$ be a dense subset of $\Pmax$ such that $\Cod[B]=D$,
\item
$A\in \pow{\omega_1}^V\setminus L(\mathbb{R})^V$ be such that
$\omega_1^{L[A]}=\omega_1^V$.
\end{itemize}

$T$ in $V$ is a witnessing tree for $A,D$ if:
\begin{itemize}
\item $T$ is a tree on $(2\times \omega_2^V)^{<\omega}$;
\item $p[T]\subseteq B^{V[G]}$ holds in all generic extensions $V[G]$ of $V$;
\item whenever $G$ is $V$-generic for $\Coll(\omega,\eta)$ for some $\eta\geq \lambda$,
there is $r\in p[T]^{V[G]}$ such that $\Cod(r)=(N_0,a_0)$ refines $(H_{\omega_2}^V,A)$ in the $\Pmax$ order of $V[G]$.
\end{itemize} 
\end{definition}

%A trivial remark is that whenever $Y\subseteq X$ are transitive models of $\ZFC^-$ containing $H_{\omega_2}^V$ with $|X|^V\leq\lambda$
%and $T$ is a witnessing tree for $A,X$, $T$ is also a witnessing tree for $A,Y$ (since if $\Cod(r)=(N_0,a_0)$ refines $(X,A)$, it refines also $(Y,A)$).

The following Fact encapsulates the unique place in the proof of Thm. \ref{thm:mainresult}
where the assumption that there are class many Woodin cardinals is essentially used and cannot be significantly reduced.
\begin{fact}\label{fac:treeTforP0}
Assume $V$ models
$\NS_{\omega_1}$\emph{ is saturated $+$ there are class many Woodin cardinals}.
%Let $X\in V$ be a transitive model of $\ZFC^-$ containing $H_{\omega_2}^V$ with $|X|^V\leq\lambda$.

Then
there is in $V$ a witnessing tree
$T$ on $(2\times\omega_2^V)^{<\omega}$ for $A,D$ for any $D$ dense subset of $\Pmax$ with $D=\Cod[B]$ for some $B\in \UB^V$, and 
$A\in\pow{\omega_1}^V\setminus L(\mathbb{R})$ such that $\omega_1^{L[A]}=\omega_1^V$.
\end{fact}

\begin{proof}

Let $G$ be $V$-generic for $\Coll(\omega,\omega_2^V)$. Then  
$(H_{\omega_2}^V,A)$ is a $\Pmax$-condition in $V[G]$ 
(any iteration of $X$ in $V[G]$ of length at most 
$\omega_1^{V[G]}$
can be extended to an iteration of $V$, by \cite[Lemma 1.5, Lemma 1.6]{HSTLARSON}). 

Find $(N,b)\leq(H_{\omega_2}^V,A)$ in $\Cod[B^{V[G]}]$ (which is dense in $\Pmax^{V[G]}$ by Thm. \ref{thm:genabshomega1}), 
$r\in B^{V[G]}$ such that
$\Cod(r)=(N,b)$, $h\in \lambda^\omega$ such that $(r,h)$ is  a branch through 
$T_B$ in $V[G]$ (where $T_B,S_B$ are class sized trees on $(2\times\Ord)^{<\omega}$ witnessing the universal Baireness of $B$ in $V$
according to Def. \ref{def:UBset}).

Let $\dot{r}$ and $\dot{h}$ be $\Coll(\omega,\omega_2^V)$-names\footnote{To prove the existence 
of such names independently of the choice of the $V$-generic filter $G$ one uses the homogeneity 
of $\Coll(\omega,\omega_2^V)$.} for $r$ and $h$ such that
every $p$ in $\Coll(\omega,\omega_2^V)$ forces that $(\dot{r}_G,\dot{\dot{h}}_G)$ is  a 
branch through $T_B$ such that 
$\Cod(\dot{r}_G)\leq_{\Pmax} (H_{\omega_2}^V,A)$ holds in $V[G]$
whenever $G$ is $V$-generic with 
$p\in G$. 

Let for each $n$ 
\[
Y_n=\bp{\beta\in\Ord: \exists p\in\Coll(\omega,\omega_2^V)\, p\Vdash \dot{h}(\check{n})=\check{\beta}}
\]
and $Z=\bigcup_{n\in\omega}Y_n$.
Then $Z\in V$ has size at most $\omega_2^V$ and $r$ belongs to the projection of 
$T^*=T_B\cap (2\times Z)^{<\omega}$ in $V[G]$.
We copy in $V$ $T^*$ to a tree $T$ on $(2\times\omega_2^V)^{<\omega}$ via an injection 
$g:Z\to\lambda$ in $V$,
and we let $\dot{f}$ be the $\Coll(\omega,\omega_2^V)$-name on $T$ obtained as an isomorphic image of 
$\dot{h}$ via $g$.

We get that:

\begin{quote}
\emph{$\dot{r},\dot{f}$ are $\Coll(\omega,\omega_2^V)$-names such that whenever $H$ is $V$-generic for 
$\Coll(\omega,\omega_2^V)$, $(\dot{r}_H, \dot{f}_H)$ is a branch through $T$ such that
$\Cod(\dot{r}_H)\in  \Cod[B^{V[H]}]$ is a $\Pmax$-condition refining $(H_{\omega_2}^V,A)$.}
\end{quote}

\begin{quote}
\emph{The map $g$ grants that
%$f:(s,t)\mapsto \ap{g^{-1}(s(i),t(i)):i\in\dom(s)=\dom(t)}$ is a $\sqsubseteq$-preserving injective 
%map of $T$ into $T_0$, therefore 
$p[T]^{V[K]}\subseteq p[T_B]^{V[K]}=B^{V[K]}$ holds in any forcing extension $V[K]$ of $V$ for any notion of forcing in $V$.}
\end{quote}

Hence $T$ works.
\end{proof}

\begin{remark}
We use here in an essential way the existence of class many
Woodin cardinals (specifically Thm. \ref{thm:genabshomega1}). 
We could also replace this assumption
with the existence of a Woodin cardinal $\delta$ which is a limit of Woodin cardinals (and then use
Thm. \cite[Thm 3.1.6]{STATLARSON}). What is crucially needed to prove the Fact 
is that the 
$\Pi_2$-sentence for $\tau_{\bool{ST}}$ 
\[
\forall x\,\qp{x\not\in \Pmax \vee \exists y\, (B(y)\wedge \Cod(y)\leq_{\Pmax}x)}
\]
holding in the structure
\[
(H_{\omega_1}^{V},\tau_\ST,B)
\]
remains true 
in the structure
\[
(H_{\omega_1}^{V[G]},\tau_\ST,B^{V[G]}).
\]
To see that the above is a $\Pi_2$-sentence for $\tau_{\bool{ST}}$ use \cite[Remark 1.4]{HSTLARSON} and 
\cite[Lemma 25.25]{JECHST}.
\end{remark}

%Key clauses in many of the arguments  to follow are \ref{def:semcert0-5} and \ref{def:semcert0-6}, more comments later on.

%Let $A,D$ in $V$ be such that:
%
%
%\begin{itemize}
%\item
%$(H_{\omega_1}^,\tau_{\ST}^M,A^M:A\in \UB^V)\sqsubseteq (H_{\omega_1}^W,\tau_{\ST}^W,A^W:A\in \UB^V)$,
%\item
%$\bar{T}$ is a tree on $(2\times\omega_2)^{<\omega}$ existing in $\bar{M}$,
%$H_{\kappa}^M$ for $\kappa$ least in $M$ satisfying the requirements of 
%Convention \ref{conv:convobjP0Pkappa} in $M$,
%

%\end{itemize}

\begin{definition}\label{def:weaksemcert0} 
Let $M$ be a transitive model of $\ZFC+\NS_{\omega_1}$\emph{ is saturated}:
\begin{itemize}
%\item
%$X$ be a transitive model of $\ZFC^-+\NS_{\omega_1}$\emph{ is saturated} with $H_{\omega_2}^X$ a definable class\footnote{I.e. $X$ is a transitive model of a large enough fragment of $\ZFC$ which contains its own version of $H_{\omega_2}$ as a subset.} in $(X,\in)$;
\item
$T$ in $M$ be a tree on $(2\times \omega_2^M)^{<\omega}$;
\item
$A\in \pow{\omega_1}^M\setminus L(\mathbb{R})^M$ be such that $\omega_1^{L[A]}=\omega_1^M$.
\end{itemize}

%\begin{itemize}[*]
%\item
A tuple $(\mathcal{J},r,f,T)$ existing in some transitive outer model $W$ of $M$
 is a \emph{weak semantic certificate for $A$} 
 if:
\begin{enumerate}
\item \label{def:semcert0-2bis}
$(r,f)$ is a branch through $T$.
\item \label{def:semcert0-3}
$\mathcal{J}=\bp{j_{\alpha\beta}:\,N_\alpha\to N_\beta:\, \alpha\leq\beta\leq\omega_1^M}$
is an iteration with $N_0$ iterable in $W$. 
\item \label{def:semcert0-4}
$\Cod(r)=(N_0,A\cap\omega_1^{N_0})$ holds.
\item \label{def:semcert0-4bis}
$j_{0\omega_1^M}(A\cap\omega_1^{N_0})=A$.
\item \label{def:semcert0-5}
$H_{\omega_2}^M\subseteq N_{\omega_1}$.
\item \label{def:semcert0-6}
$\NS_{\omega_1}^{M}=\NS_{\omega_1}^{N_{\omega_1^M}}\cap\pow{\omega_1}^M$.
\end{enumerate}
%\end{definition}
%
%
%\begin{definition}\label{def:semcert0} 
\item
%Let $(M,W)$ be a  canonical pair\footnote{Recall Def. \ref{def:canpairs}. First reading this definition one can suppose that $W$ is a generic extension of $V$ and $M$ is $V$.} be such that:
%\begin{itemize} 
%\item
%$M,W$ are models of $\ZFC$,
%\item 
%$M$ is a model of \emph{$\NS_{\omega_1}$ is saturated},
%%\item
%%$X$ is a transitive model of $\ZFC^-+\NS_{\omega_1}$\emph{ is saturated} containing $H_{\omega_2}^M$;
%\end{itemize} 
Let
$D$ in $M$ be a dense subset of $\Pmax^M$ such that $D=\Cod[B]$ for some fixed $B\in \UB^M$.

A weak semantic certificate
$(\mathcal{J},r,f,T)$ for $A$ existing in $W$ 
is a \emph{semantic certificate for $A,D$} (relative to $M$) if:
\begin{enumerate} 
\setcounter{enumi}{6}
\item \label{def:semcert0-7}
$M$ models that $T$ is a witnessing tree for $A,D$.
\item \label{def:semcert0-8}
$W$ is some generic extension of $M$.
\end{enumerate}
\end{definition}

Condition \ref{def:semcert0-5} and \ref{def:semcert0-6} are essential in many stages of the proof, for example we already use them in the proof of Corollary \ref{cor:semcert} right below.
Condition \ref{def:semcert0-7} hides all the large cardinal strength needed in our proofs. 
Condition \ref{def:semcert0-8} makes the property:
\begin{quote}
\emph{There exists a semantic certificate for $A,D$} 
\end{quote}
first order expressible without any reference to $D$ or to the universally Baire set $B$ which codes $D$: 
the relevant information on $B$ and $D$ is now encoded in the witnessing tree $T$ (whose projection may not be defining a universally Baire set 
of $M$, but it is still defining a subset of $B^{M[G]}$ in any generic extension of $M$).
This is going to be very handy, see Remark \ref{rmk:keyremsemcert} below.

There is no reason to expect that the statement \emph{there exists a weak semantic certificate for $A$ in some transitive outer model of $V$} is first order 
expressible in any reasonable way in $V$,
since the class of transitive outer models of $V$ is not first order definable.

The main result of the paper is the following:

\begin{theorem}\label{thm:mainthm00}
Assume $V$ models \emph{$\NS_{\omega_1}$ is saturated} and \emph{there are class many Woodin cardinals}.
Let $\kappa>\omega_2$ be a regular cardinal such that $\Diamond_\kappa$ holds.

Let $A\in\pow{\omega_1}\setminus L(\mathbb{R})$ and $D$ dense subset of $\Pmax$ be such that:
\begin{itemize}
\item
$D=\Cod[B]$ for some universally Baire set $B$, 
\item
$\omega_1^{L[A]}=\omega_1^V$.
\end{itemize}

Then there is a
stationary set preserving forcing $P_{D,A}$ in $V$ of size $\kappa$ 
producing in its generic 
extensions $V[H]$ a semantic certificate $(\mathcal{J}_H,r_H,f_H,T)$ for
$A,D$ such that
\[
\NS_{\omega_1}^{V[H]}\cap N_{\omega_1^V}=\NS_{\omega_1}^{N_{\omega_1}},
\]
where
\[
\mathcal{J}_H=\bp{j_{\alpha\beta}:N_\alpha\to N_\beta:\alpha\leq\beta\leq\omega_1^V}.
\]
\end{theorem}

\begin{corollary}\label{cor:semcert}
Assume $\MM^{++}+$\emph{there are class many Woodin cardinals}. 
Then $\stUB$ holds.
\end{corollary}

The corollary is immediate: 
\begin{proof}
By $\MM^{++}$, \emph{$\NS_{\omega_1}$ is saturated}. Find $\kappa>\omega_2$ such that $\Diamond_{\kappa}$ 
holds
%\footnote{$2^{\aleph_1}=\aleph_2$  follows from $\MM^{++}(\mathfrak{c})$.
%Shelah proved that $2^{\aleph_1}=\aleph_2$ entails $\Diamond_{\aleph_2}$ (see for example the following handout by 
%\href{http://web.cs.elte.hu/~kope/ss3.pdf}{Komjath}).}
in $V$.
%Hence $\kappa=\omega_2$ can be chosen to satisfy the assumptions of the Theorem. 
Fix $A\in \pow{\omega_1}^V$ as prescribed by the theorem, and let $D$ vary among the
universally Baire in the codes dense subsets of $\Pmax$. 
If $H$ is $V$-generic for $P_{D,A}$,
then 
%by Lemma \ref{fac:keyfacSEMCERTPLAMBDA-1}  
in $V[H]$ there is 
\[
\mathcal{J}_H=\bp{j_{\alpha\beta}:N_\alpha\to N_\beta:\alpha\leq\beta\leq\omega_1}
\]
iteration such that $(N_0,A\cap\omega_1^{N_0})\in D$ and 
$j_{0\omega_1}(A\cap\omega_1^{N_0})=A$ (cfr. Condition \ref{def:semcert0-4bis} of Def. \ref{def:weaksemcert0}).
Now 
\[
H_{\aleph_2}^V\subseteq N_{\omega_1^V},
\] 
and $\NS_{\omega_1}^V=\NS_{\omega_1}^{N_{\omega_1}}\cap H_{\omega_2}$
 (by Conditions \ref{def:semcert0-5} and \ref{def:semcert0-6} of a semantic certificate).
Therefore in $V[H]$ every stationary subset of $\omega_1$ in $V$
is stationary in $N_{\omega_1^V}$,
hence also in $V[H]$ (by Thm. \ref{thm:mainthmPkappa}). 
We conclude that $P_\kappa$ is stationary set preserving.

Since $\NS_{\omega_1}^{N_{\omega_1^V}}=\NS_{\omega_1}^{V[H]}\cap N_{\omega_1^V}$,
$\mathcal{J}_H$ is a correct iteration in $V[H]$ with 
\[
(N_0,A\cap\omega_1^{N_0})=\Cod(r_H), \qquad
r_H\in p[T]\subseteq p[T_B]=B^{V[H]}.
\]

By the results of Section \ref{subsec:MM++BMM++},
the existence of a correct iteration $\mathcal{J}$ such that 
$j_{0,\omega_1}(A\cap\omega_1^{N_0})=A$ and $(N_0,A\cap\omega_1^{N_0})=\Cod(r)$ with 
$r\in B^{V[H]}$ can be reflected to $V$, by $\MM^{++}$.
This concludes the proof.
\end{proof}

The rest of the paper is dedicated to the proof of Thm. \ref{thm:mainthm00}.

First of all let us outline some key observations on semantic certificates.

\begin{remark}\label{rmk:keyremsemcert}
Assume $(M,W)$ is a canonical pair and $H_{\omega_2}^M$ is
countable in $W$.
The statement
\begin{quote}
$(\mathcal{J},r,f,T)$ is a \emph{weak semantic certificate  for $A$.}
\end{quote}
is first order expressible in the structure 
\(
(H_{\omega_1}^{W},\tau_{\bool{ST}}^{W})
\)
by the sentence $\phi(A,T,H_{\omega_2}^M,\omega_1^M)$ in parameters  (the countable in $W$ sets) 
$A,T,H_{\omega_2}^M,\omega_1^M$:
\begin{align}\label{eqn:weaksemcert}
%\exists r,f,\mathcal{J},N_0[& \\ \nonumber
&((r,f) \text{ is a branch of }T)\wedge \\ \nonumber
&\wedge \mathcal{J}=\bp{j_{\alpha\beta}:N_\alpha\to N_\beta:\alpha\leq\beta\leq\omega_1^M} 
\text{ is an iteration}\wedge\\  \nonumber
&\wedge j_{0\omega_1^M}(A\cap\omega_1^{N_0})=A\wedge\\  \nonumber
&\wedge H_{\omega_2}^M\subseteq N_{\omega_1^M}\wedge\\  \nonumber
&\wedge \NS_{\omega_1}^{M}=\NS_{\omega_1}^{N_{\omega_1^M}}\cap H_{\omega_2}^M\wedge\\ \nonumber
&\wedge \Cod(r)=(N_0,A\cap\omega_1^{N_0})\\  \nonumber
%]&
%&\wedge p[T]\subseteq B \nonumber
\end{align}
\end{remark}

The first five properties listed above
of $\mathcal{J},r,f,T$ are expressed by 
$\Delta_0$-formulae in the relevant parameters %Moreover if (as it is possible) we can choose $T$ so that $T\subseteq B^W$ holds automatically, the above
(for example we already showed that this is the case for
$\mathcal{J}$\emph{ is an iteration of length $\omega_1^M$}). 
The last assertion is a provably $\Delta_1$-property in the signature 
$\tau_{\ST}$.

The unique delicate part (hiding significant large cardinal strength)
for a semantic certificate is the request $p[T]\subseteq B^{V[G]}$ holds in all generic extensions of $V$.
This is  expressible by the $\Pi^1_1(B,T)$-property
\[
\forall r\in 2^\omega\forall f\in \kappa^\omega\,\qp{(\forall n\in\omega (r\restriction n,f\restriction n)\in T)\rightarrow B(r)}.
\]
In principle $\Pi^1_1(B,T)$-properties are not universally Baire, even if $B$ is universally Baire and $T$ is countable. 
%However they so are
%if $V$ models that there class many Woodin cardinals (less suffices). 
%Model completeness of $T_{\mathcal{A}}$ is another way to capture the relevant (for us) consequences of this fact.

Now if $W$ is any transitive outer model of $M$, it is not at all clear whether we can make sense of 
$B^W$ for a universally Baire set $B$ of $M$. The notion of witnessing tree $T$ show that we do not have to worry about this issue, and 
we can just check whether in this outer model weak semantic certificates $(\mathcal{J},r,f,T)$ for $A$ 
with certain properties (depending only on the parameters $A,H_{\omega_2}^M,T$) exist in $W$.

Most of our arguments will run as follows: 
\begin{itemize}
\item We fix in advance in $V$ a witnessing tree $T$ for $A,D$.
\item
We define in some generic extension $V[G]$ of $V$ a generic 
$\bar{k}:V\to \bar{M}$ with $\omega_1^{\bar{M}}=\omega_1^{V[G]}$.
\item
We consider the canonical pair $(\bar{M},V[G])$ and 
we also produce in $V[G]$ a weak semantic certificate $(\mathcal{J},r,f,\bar{k}(T))$ for $\bar{k}(A)$ with certain extra properties (independent of the universally Baire set $B$) which
can be expressed as $\Sigma^1_2(\bar{k}(T),\bar{k}(A),\bar{k}(H_{\omega_2}^V))$-facts in $\bar{M}[H]$
%of $(H_{\omega_1}^{\bar{M}[H]},\tau_{\ST}^{\bar{M}[H]})$ 
whenever $H$ is an $\bar{M}$-generic filter  with  
$\bar{k}(A),\bar{k}(H_{\omega_2}),\bar{k}(T)$ countable in $\bar{M}[H]$. 
%these extra properties however ensure among other things that $(r,f)\in\bar{k}(T)$. 
\item
Since these properties of $(\mathcal{J},r,f,\bar{k}(T))$ (holding in $V[G]$) 
still hold in $H_{\omega_1}^{V[H]}$ for any $V$-generic extension $V[H]$ in which
$\bar{k}(A),\bar{k}(H_{\omega_2}),\bar{k}(T)$ are countable, 
we reflect them to $(H_{\omega_1}^{\bar{M}[H]},\tau_{\ST}^{\bar{M}[H]})$ 
using Shoenfield's absoluteness between $(H_{\omega_1}^{\bar{M}[H]},\tau_{\ST}^{\bar{M}[H]})$ and 
$(H_{\omega_1}^{V[H]},\tau_{\ST}^{V[H]})$.
\item 
Finally (since $\bar{k}(T)$ is in $\bar{M}$ a witnessing tree for $\bar{k}(A),\bar{k}(D)$), 
we use elementarity of $\bar{k}$ to reflect these properties to $V$ in the form:
\begin{quote}
\emph{Semantic certificates for $A,D$ with the required properties exist in any generic extension of $V$ by 
$\Coll(\omega,\omega_2^V)$.} 
\end{quote}
\end{itemize}
The first example of this modular argument  is given by Lemma \ref{lem:keylemASPSCH(*)0}.

\subsection{Existence of semantic certificates}\label{subsec:exsemcert}

The first key result of Asper\'o and Schindler is that semantic certificates for $A,D$
exist in some generic extension of $V$ (the proof below produces them in a generic extension collapsing $\omega_2$ to countable).

\begin{lemma}\label{lem:keylemASPSCH(*)0}
Assume in $V$ 
$\NS_{\omega_1}$ is saturated and Martin's axiom, $X$ is a transitive model of $\ZFC^-$ 
and $T$ is a witnessing tree on $(2\times\omega_2)^{<\omega}$ 
for $A,D$ as given by Fact \ref{fac:treeTforP0}, with 
$A\in\pow{\omega_1}\setminus L(\mathbb{R})$ 
such that $\omega_1^{L[A]}=\omega_1^V$ and $D=\Cod[B]$ dense subset of $\Pmax$ with $B$ universally Baire.

Then for any $G$
$V$-generic for $\Coll(\omega,\omega_2^V)$, there is in $V[G]$ a 
semantic certificate $(\mathcal{J},r,f,T)$ for $A,D$. 
\end{lemma}

\begin{proof}
Let $G$ be $V$-generic for $\Coll(\omega,\omega_2^V)$.
In $V[G]$ $H_{\omega_2}^V$ is countable and iterable.
This gives that $(H_{\omega_2}^V,A)$ is a $\Pmax^{V[G]}$ condition.
%Let $T$ be the tree on $(2\times\omega_2)^{<\omega}$ as given by Convention \ref{conv:convobjP0Pkappa-1}.

Since $T$ is a witnessing tree for $A,D$, there is
$(N,b)=\Cod(\dot{r}_G)\leq (H_{\omega_2}^V,A)$ with $r=\dot{r}_G\in p[T]$. %(hence with $(N,b)\in D^{V[G]}$).

Since $N$ is countable and iterable in $V[G]$ and $\gamma=(\omega_3)^V=\omega_1^{V[G]}$, we can find
in $V[G]$ an iteration 
\[
\mathcal{J}=\bp{j_{\alpha\beta}:N_\alpha\to N_\beta:\,\alpha\leq\beta\leq\gamma=\omega_1^{V[G]}}
\]
of $N=N_0$ % such that $\NS_{\omega_1}^{N_{\omega_1}}=\NS_{\omega_1}^{V[G]}\cap N_{\omega_1}$
(just apply \cite[Lemma 2.8]{HSTLARSON} in $V[G]$ to $N$).

Let $\mathcal{K}=\bp{k_{\alpha\beta}:\,\alpha\leq\beta\leq\omega_1^N}\in N$ be the iteration witnessing that
$(N,b)\leq(H_{\omega_2}^V,A)$ and 
\[
\bar{\mathcal{K}}=\bp{\bar{k}_{\alpha\beta}:\,\alpha\leq\beta\leq\gamma}=j_{0\gamma}(\mathcal{K}).
\]
Extend  in $V[G]$ $\bar{\mathcal{K}}$ to a full iteration of $V$ applying 
\cite[Lemma 1.5, Lemma 1.6]{HSTLARSON}, and
let $\bar{M}=\bar{k}(V)$ with $\bar{k}$ the unique extension of $\bar{k}_{0\gamma}$ to $V$.

Let $H$ be $V$-generic for $\Coll(\omega,\eta)$ with $\eta=\bar{k}(\omega_2^V)$ and 
$G\in V[H]$.
Now observe that:
\begin{itemize}
\item
$H$ is $\bar{M}$-generic for $\Coll(\omega,\eta)$, since $\bar{M}\subseteq V[G]$ and $H$ is also
$V[G]$-generic for $\Coll(\omega,\eta)$. 
\item $\bar{k}(T)$ is a tree in $(2\times\bar{k}(\omega_2^V))^{<\omega}$ in $\bar{M}$
which is countable in $\bar{M}[H]$ and in $V[H]$, and $r\in p[\bar{k}(T)]^{V[G]}$.
%\item
%$\bar{k}(T)\supseteq \bar{k}[T]$, hence $p[\bar{k}(T)]\supseteq p[T]$  holds in any outer model of 
%$\bar{M}$ to which $\bar{k}(T)$ belongs,

%\item
%$p[\bar{k}(T_0)]\supseteq  p[\bar{k}(T)]$ also holds in any generic extension of $\bar{M}$,
%where $\bar{k}(T_0)=\bigcup_{\alpha\in \Ord}\bar{k}(T_0\cap V_\kappa)$.
\item
$\mathcal{J},\bar{k}(H_{\omega_2}^V)$ are in $H_{\omega_1}^{V[H]}$, and $H_{\omega_1}^{V[H]}$ models that
$\mathcal{J}$ is an iteration such that $(N_0,b)=\Cod(r)$, 
$r\in p[\bar{k}(T)]$, and 
\[
j_{0\gamma}(b)=j_{0\gamma}(\bar{k}_{0\omega_1^{N}}(A))=\bar{k}(A).
\]
\end{itemize}

Consider the 
$\Sigma_1$-sentences in signature $\tau_\ST$ and parameters $\bar{k}(A),\bar{k}(T),
\bar{k}(H_{\omega_2}^V),\omega_1^{\bar{M}}$ (which are in 
$\bar{M}$ and are countable in $\bar{M}[H]$) 
\textit{
\begin{quote}
There is a branch $(r,\bar{k}[f])$ of $\bar{k}(T)$ and an iteration 
$\mathcal{J}=\bp{j_{\alpha\beta}:N_\alpha\to N_\beta:\alpha\leq\beta\leq\omega_1^{\bar{M}}}$ such that: 
\begin{itemize}
\item
$\Cod(r)=(N_0,a_0)$,
\item
$j_{0\omega_1^{\bar{M}}}(A\cap\omega_1^{N_0})=\bar{k}(A)$,
\item
$\bar{k}(H_{\omega_2}^V)\subseteq N_{\omega_1^{\bar{M}}}$,
\item
$\NS_{\omega_1}^{\bar{k}(H_{\omega_2})}=\NS_{\omega_1}^{N_{\omega_1^{\bar{M}}}}\cap \bar{k}(H_{\omega_2}^V)$.
\end{itemize}
\end{quote} 
}
This sentence is true in $V[H]\supseteq \bar{M}[H]$ and all its parameters are in 
$H_{\omega_1}^{\bar{M}[H]}$.

%Now remark that $\bar{M}[H]$ and $V[H]$ are both canonical models:
%\begin{itemize}
%\item
%For $V[H]$ we can apply directly Thm. \ref{thm:genabshomega1}.
%\item
%For $\bar{M}[H]$ we can use $\bar{k}$ to argue that 
%$(\bar{M},\in)$ is a canonical model, and then Thm. \ref{thm:genabshomega1} with
%$\bar{M}$ in the place of $V$. 
%\end{itemize}

Therefore it holds in $H_{\omega_1}^{\bar{M}[H]}$ by Shoenfield's 
absoluteness (cfr. Lemma \ref{lem:SHOABS}).

By homogeneity of $\Coll(\omega,\eta)$ we 
conclude that $\bar{M}$ models that: \emph{``in any $\bar{M}$-generic extension for $\Coll(\omega,\bar{k}(\omega_2^V))$
there is a weak semantic certificate for $\bar{k}(A)$ as witnessed by $\bar{k}(T)$''}.
By elementarity of $\bar{k}$ we get that the same holds in $V$ replacing 
$\bar{k}(A),\bar{k}(T)$ by $A,T$.
\end{proof}

\begin{remark}
Provided one is given in advance the witnessing tree $T$,  this Lemma 
does not need the existence of class many Woodin cardinals to go through. One just needs
to be able to define in $V[G]$ the iterations $\mathcal{J}$ and $\mathcal{K}$. This property of $\Pmax$-conditions holds if 
$V[G]$ satisfies very mild large cardinal properties (see \cite[Section 2]{HSTLARSON}).
\end{remark}

\section{The proof}\label{sec:proof}

We will show that STRATEGY \ref{strategy0} works.
$P_{D,A}$ is construed 
in what follows as the top element of an increasing chain of posets
\[
\bp{P_\alpha:\alpha\in C}
\]
each of size $\kappa$, 
where $\kappa>\omega_2$ is regular, $C$ is a club subset of $\kappa$, and
$\Diamond_\kappa$ holds.

\begin{convention}\label{conv:convobjP0Pkappa}
%\emph{}
Fix from now on:
\begin{enumerate}[(A)]
\item
$A\in\pow{\omega_1}\setminus L(\mathbb{R})$ such that $\omega_1^{L[A]}=\omega_1^V$, 
and $B$ universally Baire set such that 
$\Cod[B]=D$ is a dense subset of $\Pmax$ with $N$ a model of $\ZFC+$\emph{Martin's axiom} for any of its conditions $(N,b)$.
\item
$T_B,S_B$ trees on $(2\times\Ord)^{<\omega}$ projecting to complements in every set sized forcing extension of $V$ 
and such that $B=p[T_B]$.
\item $\kappa>\omega_2$ a regular cardinal such that $\Diamond_\kappa$ holds.
\item $T$ a witnessing tree on $(2\times\omega_2^V)^{<\omega}$ for $A,D$ according to Def. \ref{def:witnessingtree}.
\end{enumerate}
\end{convention}

We first focus on designing a forcing generically adding a weak semantic certificate for $A$ as witnessed by $T$
and prove that this forcing notion accomplishes all 
requests set forth in STRATEGY \ref{strategy0}, 
except possibly that of being
stationary set preserving. 
To overcome this problem one is then led to define a transfinite sequence of posets 
\[
\bp{P_\alpha:\alpha\in C\cup\bp{\kappa}}
\]
with $P_{\min(C)}$ generically adding a weak semantic certificate for $A$ as witnessed by $T$, and  
such that for $\alpha<\beta$ $P_\beta$ extends $P_\alpha$ adding many ``sufficiently $P_\alpha$-generic''
conditions.
One then ensures using $\Diamond_\kappa$
that at many stages $\lambda\in C$ the $P_\kappa$-names for clubs are guessed by   
$P_\lambda$-names for clubs. 
Now the sequence $\bp{P_\alpha:\alpha\in C}$ is designed in order that at all stages $\lambda<\beta$ there are 
$P_{\beta}$-conditions which \emph{``seal''} the $P_\lambda$-name $A_\lambda$ for a club given by the Diamond sequence 
$\bp{A_\alpha:\alpha<\kappa}$ at $\lambda$. 
Here \emph{sealing} means that for any fixed stationary set $S$ existing in $V$
one can admit as $P_{\beta}$-conditions many which will guarantee that any $P_{\beta}$-generic filter $G$
including them will be such that $G\cap P_\alpha$ is a (possibly not $P_\alpha$-generic) filter  ensuring that  
$(A_\alpha)_{G\cap P_\alpha}\cap S$ is non-empty.
In particular the $P_\lambda$ are more and more likely to be stationary set preserving as 
$\lambda$ increases;
$\Diamond_\kappa$ will be used to catch one's tail and be able to infer that 
$P_\kappa$ is stationary set preserving because all potential $P_{\kappa}$-names for 
clubs have been sealed at some previous stage.

The forcings $P_\alpha$ we borrow from Asper\`o and Schindler are here defined as
``consistency properties'', a central concept in
the analysis of infinitary logics.

\subsection{Consistency properties as forcing notions}\label{subsec:consprop}

It will be crucial for our purposes to focus on consistency properties for multi-sorted signatures for relational languages.

\subsubsection*{Multi-sorted logics}
Multi-sorted first order logic extends first order logic by introducing \emph{sorts} for elements of its structures, by specifying for each variable and constant symbol in the language
the \emph{sort} of this symbol, and for any function and relation symbol a declaration of the  \emph{sort} of its entries 
(and of its output for function symbols).

A typical relational signature for multi-sorted logics reads as
$\ap{c^i_k:k\in K_i, i\in I;\, R_j:j\in J}$ where the upper index of each constant symbol denotes the \emph{sort} of that symbol, and each relation symbol $R_j$ comes with a arity $n^j$ and a map $s^j:n_j\to I$ which specifies for each entry of the relation its sort.

Formulae of multi-sorted logic for relational language 
are built as in usual first order logic with the further extra-clauses that:
\begin{itemize}
\item
$(t=s)$ is allowed ony if $t,s$ are terms of the same sort,
\item
$R_j(t_1,\dots,t_n)$ is allowed if and only if $t_k$ is of sort $s^j(k)$ for all $k$,
\item
variable symbols are sorted and quantifiers are bounded to range over the given sort.
\end{itemize}

A structure
\[
\mathcal{M}=\ap{M_i:\, i\in I, R_j^{\mathcal{M}}:j\in J,\, (c^i_k)^{\mathcal{M}}:k\in K_i,i\in I}
\]
is a multi sorted structure for the multi-sorted signature $\ap{c^i_k:k\in K_i, i\in I;\, R_j:j\in J}$ 
in which each $M_i$ denotes the domain of the sort $i\in I$, each 
$(c^i_k)^{\mathcal{M}}$ is an element of $M_i$, and for all $j\in J$, $a_1,\dots,a_{n_j}$,
$R_j^{\mathcal{M}}(a_1,\dots,a_{n_j})$ holds if and only if $a_k\in M_{s_j(k)}$ for $k=1,\dots,n_j$.

Multi-sorted logic can be easily identified with a fragment of first order logic by identifying the multi-sorted structure 
\[
\mathcal{M}=\ap{M_i:\, i\in I, R_j^{\mathcal{M}}:j\in J,\, (c^i_k)^{\mathcal{M}}:k\in K_i,i\in I}
\]
for the multi-sorted signature $\ap{c^i_k:k\in K_i, i\in I;\, R_j:j\in J}$ with the first order structure:
\[
\mathcal{M}^*=\ap{\bigcup_{i\in I} M_i,\, M_i:\, i\in I, R_j^{\mathcal{M}}:j\in J,\, (c^i_k)^{\mathcal{M}}:k\in K_i,i\in I}
\]
for the signature $\ap{c^i_k:k\in K_i, i\in I;\, R_j:j\in J,\, X_i:i\in I}$ with each $X_i$ a new unary predicate symbol interpreted by
$M_i$ in the structure $\mathcal{M}$. But there is an inherent advantage in multi-sortd logic: for example by forbidding
the formulae $(t=s)$ for $t,s$ of different sorts one may render certain concepts which are 
first order definable in 
\[
\mathcal{M}^*=\ap{\bigcup_{i\in I} M_i,\, M_i:\, i\in I, R_j^{\mathcal{M}}:j\in J,\, (c^i_k)^{\mathcal{M}}:k\in K_i,i\in I}
\]
undefinable in  
\[
\mathcal{M}=\ap{M_i:\, i\in I, R_j^{\mathcal{M}}:j\in J,\, (c^i_k)^{\mathcal{M}}:k\in K_i,i\in I}.
\]
For example this makes the notion of morphism for multi-sorted structure much easier to handle than for first order structures
(a morphism
$k:\mathcal{M}\to\mathcal{N}$ of multi sorted structues is a multi-map mapping the domain of the sort $i$ in $\mathcal{M}$
into the domain of the sort $i$ in $\mathcal{N}$ and respecting the interpretation of atomic formulae).

We will crucially use this fact in the key part of the proof of Lemma \ref{fac:keyfacSEMCERTPLAMBDA-1} (cfr. Claim \ref{clm:fundclmrproof}). For the moment to appreciate the extra-flexibility given by multi-sorted logic we present the following
trivial example for the multi-sorted signature $\bp{0,1;A_0,A_1}$ with $0,1$ sorts and the $A_i$s unary predicates with sort $i$.
Consider the two $\bp{0,1;A_0,A_1}$-structures 
\[
\mathcal{M}_0=(\mathbb{N},\mathbb{N};\mathbb{N},\mathbb{N})
\]
\[
\mathcal{M}_1=(X,\mathbb{N}\setminus X;X,\mathbb{N}\setminus X)
\]
where $X$ is an infinite coinfinite subset of $\mathbb{N}$.
It is clear that these structures are isomorphic as multi-sorted structures $\bp{0,1;A_0,A_1}$ 
On the other hand
(following the identification suggested above of a multi-sorted structure $\mathcal{M}$ with a first order structure
$\mathcal{M}^*$) we get that
\[
\mathcal{M}_0^*=(\mathbb{N},\mathbb{N},\mathbb{N})
\]
\[
\mathcal{M}_1^*=(\mathbb{N},X,\mathbb{N}\setminus X)
\]
It is immediate to check that these two structures are not even elementarily equivalent.

The point being that in multi-sorted logic we do not have to worry abot the intersection of domains of differents sorts $i,j$
in a multi-sorted structure $\mathcal{M}$, these two domains may be (partly) overlapping, but multi-sorted logic is not able to detect this overlap.

We will use the extra-flexibility to produce a certain embedding $k:\mathcal{M}\to\mathcal{N}$ of multi-sorted structures both
having interpretation of the sorts given by partially overlapping domains. 
The collation of the maps of $i$ defined on the domains of the various sorts for $\mathcal{M}$ may not even be a function:
there could be overlap of the domain $M_i$ of sort $i$ with $M_j$ of sort $j$ and  for $a\in M_i\cap M_j$; the multimap $k$, 
may assign a value using its $i$-th component different from the value assigned by its $k$-th component.

\begin{notation}
Let $\mathcal{L}=\mathcal{S}\cup\mathcal{R}\cup\mathcal{K}$ be a multi-sorted relational signature with 
$\mathcal{S}$ its set of sorts, $\mathcal{R}$ its set of relation symbols and $\mathcal{K}$ its set of constant symbols.

We denote an $\mathcal{L}$-structure 
\[
\mathcal{M}=(M_s:s\in S,R^{\mathcal{M}}:R\in\mathcal{R},c^{\mathcal{M}}:c\in\mathcal{K})
\]
by $(\mathcal{S}^{\mathcal{M}},\mathcal{R}^{\mathcal{M}},\mathcal{K}^{\mathcal{M}})$.
\end{notation}

\begin{definition}
Given a signature $\mathcal{L}=\mathcal{S}\cup\mathcal{R}\cup\mathcal{K}$ %$\in$ only adding new constant symbols,
the $\mathfrak{L}_{\infty,\omega}$-logic for $\mathcal{L}$ is the smallest family of formulae containing the 
$\mathcal{L}$-atomic formulae and closed under quantifications, negation, infinitary disjunctions and conjunctions, i.e.:
\begin{itemize}
\item
If $\phi$ is an atomic $\mathcal{L}$-formula, it is also an $\mathfrak{L}_{\infty,\omega}$-formula for $\mathcal{L}$.
\item
If $\bp{\phi_i(\vec{x}_{ij}): i<\gamma,j<\eta}$ is a set of $\mathfrak{L}_{\infty,\omega}$-formulae for
 $\mathcal{L}$
with each
$\phi_i(\vec{x}_{ij})$ in displayed free variables 
$\bp{x_{ij}:j<\omega}$,
\[
\bigwedge_{i<\gamma}\phi_i(\vec{x}_{ij}),
\]
\[
\bigvee_{i<\gamma}\phi_i(x_0,\dots,x_n)
\]
are $\mathfrak{L}_{\infty,\omega}$-formulae for $\mathcal{L}$ in displayed free variables  
$\bp{x_{ij}:i<\gamma,j<\eta}$.
\item
If $\phi(y,\vec{x})$ is an $\mathfrak{L}_{\infty,\omega}$-formula for $\mathcal{L}$, so are $\exists y\phi(y,\vec{x})$ and
$\forall y\phi(y,\vec{x})$.
\item
If $\phi$ is an $\mathfrak{L}_{\infty,\omega}$-formula for $\mathcal{L}$, so is $\neg\phi$.
\end{itemize}
The semantics is defined on $\mathcal{L}$-structures 
$\mathcal{M}=(M_s:s\in\mathcal{S},R^{\mathcal{M}}:R\in\mathcal{R},c^{\mathcal{M}}:c\in\mathcal{K})$ 
and assignments of the free variables
to $M=\bigcup_{s\in\mathcal{S}}$ by the obvious rules: 
\begin{itemize}
\item
\[
\mathcal{M}\models R_j(\vec{x})[\vec{x}/\vec{a}]
\]
if and only if
\[
R_j^\mathcal{M}(\vec{a}) \text{ holds}.
\]

\item
\[
\mathcal{M}\models \neg\phi(\vec{x})[\vec{x}/\vec{a}]
\]
if and only if
\[
\mathcal{M}\not\models \phi(\vec{x})[\vec{x}/\vec{a}].
\]
\item
\[
\mathcal{M}\models \bigvee_{i<\gamma}\phi_i(\vec{x}_i)[\vec{x}_i/\vec{a}_i:i<\gamma]
\]
if and only if for some $i<\gamma$
\[
\mathcal{M}\models \phi_i(\vec{x}_i)[\vec{x}_i/\vec{a}_i].
\]
\item
\[
\mathcal{M}\models \bigwedge_{i<\gamma}\phi_i(\vec{x}_i)[\vec{x}_i/\vec{a}_i:i<\gamma]
\]
if and only if for all $i<\gamma$
\[
\mathcal{M}\models \phi_i(\vec{x}_i)[\vec{x}_i/\vec{a}_i].
\]
\item
\[
\mathcal{M}\models \forall x\phi(x,\vec{y})[\vec{y}/\vec{b}]
\]
with $x$ of sort $s$ if and only if  for all $a\in M_s$
\[
\mathcal{M}\models \phi(x,\vec{y})[x/a,\vec{y}/\vec{b}].
\]
\item
\[
\mathcal{M}\models \exists x\phi(x,\vec{y})[\vec{y}/\vec{b}]
\]
with $x$ of sort $s$ if and only if for some $a\in M_s$
\[
\mathcal{M}\models \phi(x,\vec{y})[x/a,\vec{y}/\vec{b}].
\]
\end{itemize}
\end{definition}

\begin{notation}
A $\mathfrak{L}_{\infty,\omega}$-formula $\psi(\vec{x})$ is of $\bigwedge\bigvee$-type if it is logically equivalent to a formula
\[
\bigwedge_{i\in I}\bigvee_{j\in J_i}\phi_{ij}(\vec{x})
\]
with each $\phi_{ij}$ an atomic or negated atomic formula.
\end{notation}

\begin{definition}\label{def:standtheory}
A $\mathfrak{L}_{\infty,\omega}$-theory $\Sigma$ for the signature $\mathcal{L}=(\mathcal{S},\mathcal{R},\mathcal{K})$
(with $\mathcal{K}=\bp{c^s_j:j\in K_s,\,s\in\mathcal{S}}$) is \emph{standard} if
it contains: 
\begin{description}
\item[$\mathcal{L}$-equality axioms]
\emph{}

\begin{enumerate}
\item $(c=c)$,
\item $(c=d)\rightarrow (d=c)$,
\item $\qp{(c=d)\wedge(d=e)}\rightarrow (c=e)$,
\item $\qp{R(c_1,\dots,c_n)\wedge\bigwedge_{j=1}^n(d_j=c_j)}\rightarrow R(d_1,\dots,d_n)$.
\end{enumerate}
for all constant symbols $c,d,e,c_1,\dots,c_n,c_1,\dots,c_n$ of  $\mathcal{L}$ and all permitted formulae of type $(c=d)$,
$R(c_1,\dots,c_n)$ of the multi-sorted signature
$\mathcal{L}$.
\item[$\mathcal{L}$-quantifier elimination axioms]
\[
(\bigvee_{j\in K_s}\psi(c^s_j))\leftrightarrow \exists x\psi(x),
\]
for all $\mathfrak{L}_{\infty,\omega}$-formulae $\psi(x)$ in displayed free variable $x$ of sort $s\in\mathcal{S}$;
\[
\forall x\bigvee_{j\in K_s}\,(x=c^s_j)
\]
for all free variables $x$ of sort $s$ and for all sorts $s\in\mathcal{S}$.
\end{description}

$\Sigma$ is of \emph{$\bigwedge\bigvee$-type} if it is standard and
all other axioms of $\Sigma$ are sentences of $\bigwedge\bigvee$-type.

%All other axioms of $\Sigma$ are sentences of type $\bigwedge\bigvee$.
\end{definition}

\begin{remark}
For a standard $\mathfrak{L}_{\infty,\omega}$-theory, the quantifier elimination axioms give that
any $\Pi_2$-sentence $\forall x\exists y\psi(x,y)$ with $\psi(x,y)$ quantifier free is equivalent
to a $\bigwedge\bigvee$-type sentence. By passing to infinitary logic we will be able to
express the $\Pi_2$-fragment of an ordinary multi-sorted $\mathcal{L}$-theory $T$ with the Henkin property (i.e. for any $\mathcal{L}$-formula 
$\psi(x)$
there is some constant symbol $c\in\mathcal{L}$ with
$(\exists x\,\psi(x))\rightarrow \psi(c)\in T$) using 
quantifier free sentences of $\bigwedge\bigvee$-type on the $\mathfrak{L}_{\infty,\omega}$-theory extending $T$ with
the equality and quantifier elimination axioms.
\end{remark}

\subsubsection*{Consistency properties}

We will be interested in studying certain forcing notions whose properties are best described by means of standard theories.

\begin{definition}
Let $\mathcal{L}=\mathcal{S}\cup\mathcal{R}\cup\mathcal{K}$ be a multi-sorted signature.

\begin{itemize}
\item
A $\nu:\mathcal{L}\to M\cup\bigcup_{n<\omega}\pow{M^n}$ existing in some generic extension of $V$ 
is a generic $\mathcal{L}$-assignment
if $\mathcal{M}=(\mathcal{S}^{\mathcal{M}},\mathcal{R}^{\mathcal{M}},\mathcal{K}^{\mathcal{M}})$ is an $\mathcal{L}$-structure 
with $\mathcal{S}^{\mathcal{M}}=\bp{M_s:s\in\mathcal{S}}$,
$M=\bigcup\mathcal{S}^{\mathcal{M}}$,  and:
\begin{itemize}
\item 
$\nu\restriction\mathcal{K}$ maps correctly and surjectively 
the constants of sort $s$ to $M_s$, 
\item
$\nu\restriction\mathcal{R}$ maps correctly 
the relation symbols $R_j$ of $\mathcal{R}$ of arity $n_j$ and type $s_j$
to subsets of $M_{s_j(0)}\times\dots\times M_{s_j(n_j-1)}$.
\end{itemize}
\item
Given a generic assignment $\nu$, $\Sigma_\nu$ is the family of $\mathfrak{L}_{\infty,\omega}$-sentences for $\mathcal{L}$
realized in the structure $\mathcal{M}$.
\end{itemize}
\end{definition}

\begin{warning}\label{war:1str}
In certain circumstances occurring below, 
when the interpretation of a certain relation symbol $R$ in $\mathcal{R}$ is clear, 
we consider as an $\mathcal{L}$-structure also one where the interpretation of $R$ is not explicitly given (because the information conveyed on the intepretations of the other symbols suffices to determine the natural interpretation of $R$).

The same occurs for generic assignments $\nu$, we will consider just the assignment restricted to the constants 
(and eventually to some predicates), the assignment $\nu$ on other predicates will be clear from the context.

This is a standard practice (consider for example the omission of the interpretation for the relation symbol $=$ in the usal 
notation for first order structures).
\end{warning}
%In particular it turns out that $\mathcal{M}$ is isomorphic to a 
%The key point is that in what follow we will have a certin freedom to define $\nu$ on the constant symbols while the interpretation of the realtion symbols is much less flexible, the above set up essentially treats the constant symbols as variables whose value can be rearranged changing 
%$\nu$, while the underlying structure $(M,\mathcal{R}^{\mathcal{M}})$ is essentially independent of $\nu$.

The following definition encapsulates the key provisions of a \emph{consistency property} generated by a class 
of structures according to \cite{KEIINFLOG}, and gives a powerful tool to produce ``generic'' models omitting a prescribed family of types.

\begin{definition}\label{def:conspropforcing}
Given a family of generic assignments $\mathcal{A}$ for a signature $\mathcal{L}=\mathcal{S}\cup\mathcal{R}\cup\mathcal{K}$:
\begin{itemize}
\item
The forcing $P_{\mathcal{A}}$ consists of all 
those finite sets of atomic or negated atomic $\mathcal{L}$-sentences which are contained in some 
$\Sigma_\nu$ with $\nu\in\mathcal{A}$.
\item
The order on $P_{\mathcal{A}}$ is given by reverse inclusion.
\item
A set of $\mathcal{L}$-sentences $\Sigma$ is $\mathcal{A}$-consistent  if 
$p\in P_{\mathcal{A}}$ for every finite subset $p$ of $\Sigma$.
\item
Given a maximally $\mathcal{A}$-consistent set of sentences $\Sigma$, the term model $\mathcal{M}_\Sigma$ for $\mathcal{L}$
is defined as follows:
\begin{itemize}
\item $[c]_\Sigma=\bp{d\in\mathcal{K}: (c=d)\in \Sigma}$;
\item $s^\Sigma=\bp{[c]_\Sigma: c\text{ of sort }s}$;
\item $R^\Sigma([c_1]_\Sigma,\dots,[c_n]_\Sigma)$ if and only if $R(c_1,\dots,c_n)\in \Sigma$.
\end{itemize}
The canonical assignment $\nu_\Sigma$ maps $c$ to $[c]_\Sigma$, $R$ to $R^\Sigma$.
\end{itemize}
\end{definition}

Note that a  filter $H$ on $P_{\mathcal{A}}$ is maximal if and only if $\Sigma_H=\cup H$ is maximally $\mathcal{A}$-consistent.
It is also clear that $\mathcal{A}$-consistent sets are finitely satisfiable, hence consistent in the ordinary sense.

%A consistency property according to \cite{KEIINFLOG} would include
\begin{lemma}\label{lem:keyconspropforcing}
Given a family of generic assignments $\mathcal{A}$ for a signature $\mathcal{L}=\mathcal{S}\cup\mathcal{R}\cup\mathcal{K}$, the following holds for the forcing
$P_{\mathcal{A}}$:
\begin{enumerate}
\item $\mathcal{M}_\Sigma$ is a well defined $\mathcal{L}$-structure for all maximally $\mathcal{A}$-consistent $\Sigma$.
\item 
Assume $H$ is $V$-generic for $P_{\mathcal{A}}$, $\Sigma_H=\cup H$.

Then for any $\mathcal{L}$-sentence 
\[
\bigwedge_{i\in I}\bigvee_{j\in J_i}\psi_{ij}
\]	 
of type $\bigwedge\bigvee$, TFAE:
\begin{enumerate}[(A)]
\item 
$\mathcal{M}_{\Sigma_H}\models\bigwedge_{i\in I}\bigvee_{j\in J_i}\psi_{ij}$.
\item
For all $i\in I$ and some $p\in H$,
$D_i$
is dense below $p$,
where
\[
D_i=\bp{q\in P_{\mathcal{A}}: \psi_{ij}\in q\text{ for some $j\in J_i$}}.
\]
\end{enumerate}
\end{enumerate}
\end{lemma}

\begin{notation}
Given a maximal filter $H$ on $P_\mathcal{A}$ denote
$\mathcal{M}_{\Sigma_H}$ by $\mathcal{M}_{H}$, $[c]_{\Sigma_H}$ by $[c]_H$, $R^{\Sigma_H}$ by $R^H$, 
$s^{\Sigma_H}$ by $s^H$,
$\nu_{\Sigma_H}$ by $\nu_H$.

$\mathcal{M}_{\dot{H}}$,  $\nu_{\dot{H}}$, $R^{\dot{H}}$,  $s^{\dot{H}}$, $[-]_{\dot{H}}$ denote the canonical $P_{\mathcal{A}}$-names for the
corresponding objects determined by the $V$-generic filter $H$ for  $P_{\mathcal{A}}$.
\end{notation}

\begin{proof}
\emph{}

\begin{enumerate}
\item
It is clear that $[-]_\Sigma$ is an equivalence relation, and also that $R^\Sigma$ is independent of the representative chosen in 
$[-]_\Sigma$. We spell out some details:
\begin{description}
\item[$\qp{-}_\Sigma$ is an equivalence relation]
Assume $(c=d),(d=e)\in \Sigma_H$. 
The sets $\bp{(c=d),(d=e),\neg(c=e)}$, $\bp{(c=d),\neg(d=c)}$, $\bp{\neg(c=c)}$ are not even consistent, 
therefore any maximally $\mathcal{A}$-consistent set (henceforth consistent) 
containing $(c=d),(d=e)$ must also contain $(c=e)$, $(d=c)$, $(c=c)$.
\item[$R^\Sigma$ is well defined] Assume $[c_i]_\Sigma=[d_i]_\Sigma$ for $i=1,\dots,n$ then:
\begin{quote}
\[
R^\Sigma([c_1]_\Sigma,\dots,[c_n]_\Sigma)
\]
if and only if
\[
R(c_1,\dots,c_n)\in\Sigma
\]
if and only if {\tiny (since $\bp{(c_1=d_1),\dots,(c_n=d_n), R(c_1,\dots,c_n),\neg R(d_1,\dots,d_n)}$ is not even consistent)}
\[
R(d_1,\dots,d_n)\in\Sigma.
\]
\end{quote}
\end{description}

\item 
Assume $H$ is $V$-generic for $P_{\mathcal{A}}$.

\smallskip

One direction is clear:
if for some $p\in H$ 
$D_i$ is predense below $p$ for all $i\in I$, then $H\cap D_i\neq\emptyset$ for all $i\in I$, hence
for all $i$ there is $j$ such that $\psi_{ij}\in\Sigma_H$.
This immediately yields that 
\[
\mathcal{M}_H\models\bigwedge_{i\in I}\bigvee_{j\in J_i}\psi_{ij}.
\]

\smallskip

For the converse assume 
\[
\mathcal{M}_H\models\bigwedge_{i\in I}\bigvee_{j\in J_i}\psi_{ij}.
\]
Then some $p\in H$ is such that
\[
p\Vdash_{P_{\mathcal{A}}} \gp{\mathcal{M}_{\dot{H}}\models\bigwedge_{i\in I}\bigvee_{j\in J_i}\psi_{ij}}.
\]
By definition (since $Z=\bp{(i,j):i\in I,j\in J}$ is in $V$ and so is $\bigwedge_{i\in I}\bigvee_{j\in J_i}\psi_{ij}$)
\[
p\Vdash_{P_{\mathcal{A}}} \forall \tau\in \check{I}\exists\sigma\in \check{J}_i
\gp{\mathcal{M}_{\dot{H}}\models\psi_{\tau\sigma}}.
\]
%which occurs 
%if and only if
%\[
%\text{For all $i\in I$ $p\Vdash_{P_{\mathcal{A}}} \exists \sigma\in \check{J}_i
%\gp{\mathcal{M}_{\dot{H}}\models\psi_{\check{i}\sigma}}.
%\]
By definition of the forcing relation (using the fact that $\check{Z}_H=Z$  for all $V$-generic filters $H$, 
and the forcing clauses for existential formulae), 
this holds if and only if
\[
D^*_i=\bp{q\in P: \exists j \, (q\Vdash \gp{\mathcal{M}_{\dot{H}}\models\psi_{\check{i}\check{j}}})}
\]
is open dense below $p$ for all $i\in I$.

Now if $q\Vdash \gp{\mathcal{M}_{\dot{H}}\models\psi_{\check{i}\check{j}}}$ and $H$ is $V$-generic for $P_{\mathcal{A}}$ with $q\in H$, 
we get that
$\psi_{ij}\in \Sigma_H$, therefore $q\cup\bp{\psi_{ij}}\in H$.

Hence we conclude that 
\[
D_i=\bp{q\in P: \exists j \, (\psi_{ij}\in q)}
\]
is dense below $p$ for all $i\in I$.
\end{enumerate}
\end{proof}

\begin{remark}
Note that the Lemma holds in this generality just for sentences of $\bigwedge\bigvee$-type.
For example consider a sentence
\[
\bigvee_{j\in J}\bigwedge_{i\in I_j}\psi_{ij}
\]
Running the proof as before for this sentence we get to the conclusion that
\[
\mathcal{M}_H\models\bigvee_{j\in J}\bigwedge_{i\in I_j}\psi_{ij}
\]
if and only if for some $j\in J$ and some $q\in H$
\[
q\Vdash \gp{\mathcal{M}_{\dot{H}}\models\bigwedge_{i\in I_j}\psi_{ij}}.
\]
Now this is in general false because typically a condition $q\in P_{\mathcal{A}}$ 
can decide only finitely many atomic $\mathcal{L}$-sentences (the unique effective method to check whether 
$q$ forces 
the truth in $\mathcal{M}_{\dot{H}}$ of an atomic $\mathcal{L}$-sentence is to check that it belongs to $q$).
On the other hand the above conditions require $q$ to decide 
the infinitely many atomic $\mathcal{L}$-sentences
$\psi_{ij}$ for all $i\in I_j$.
\end{remark}

\begin{definition}\label{def:SigmaAconsprop}
Given a family of generic assignments $\mathcal{A}$, we let $\Sigma_{\mathcal{A}}$ denote the family of
$\mathcal{L}$-sentences $\psi$ in $V$ of type $\bigwedge\bigvee$ such that 
\[
\emptyset\Vdash_{P_{\mathcal{A}}} \gp{\mathcal{M}_{\dot{H}}\models\psi}.
\]
\end{definition}

The following observation is very useful:
\begin{fact}\label{fac:SigmaAconsprop}
Assume $p\in P_{\mathcal{A}}$ and $\psi:=\bigwedge_{i\in I}\bigvee_{j\in J_i}\psi_{ij}$ is a sentence of 
$\bigwedge\bigvee$-type such that any $\mathcal{A}$-assignment $\nu$ 
with $p\subseteq\Sigma_\nu$ has $\psi$ in $\Sigma_\nu$. Then $\mathcal{M}_H\models\psi$ 
for any $H$ $V$-generic for $p$.

In particular if a family of $\mathcal{L}$-assignments 
$\mathcal{A}$ is axiomatized by a standard $\mathcal{L}$-theory $\Sigma$, $\mathcal{M}_H\models\Sigma$
for any $H$ $V$-generic for $\Sigma$.
\end{fact}
\begin{proof}
The sets 
\[
D_i=\bp{q\in P: \exists j \, (\psi_{ij}\in q)}
\]
required
to check $\mathcal{M}_H\models\psi$ are dense below $p$.
\end{proof}

\subsection{The forcing $P^*_0$}\label{subsec:P0}

First of all it is convenient to isolate the smallest transitive fragment $X_{(\mathcal{J},r,f,T)}$ of $V[G]$ having as elements 
all the objects 
required to define a semantic certificate $(\mathcal{J},r,f,T)$ existing in $V[G]$ and computing correctly all the properties of this 
semantic certificate we are interested in. %\footnote{Notice that $T$ is a witnessing tree also for $D,H_{\omega_2}$.}.
Now $P^*_0$ will be obtained by gluing together finite pieces of information describing in an appropriate multi-sorted signature
the properties of 
the structures $(X_{(\mathcal{J},r,f,T)},\in)$.
In fact $P^*_0$ will be $P_{\mathcal{A}}$ for a given family $\mathcal{A}$ of generic assignment 
$\nu:\mathcal{L}^0\to X_{(\mathcal{J},r,f,T)}$ satisfying certain natural 
prescriptions.
$\mathcal{L}^0$ will contain the sorts, predicates and constants needed to encode
what could be true or false in a generic structure $(X_{(\mathcal{J},r,f,T)},\in)$.
A generic filter for $P^*_0$ will add a term model for $\mathcal{L}^0$.
By the result of the previous section it is to be expected that if things 
are organized properly  the term model defined by the generic filter for $P^*_0$
will be isomorphic to the multi-sorted presentation in signature $\mathcal{L}^0$ of a semantic certificate (to get this we just have to ensure that the $\mathcal{L}^0$-theory associated to 
semantic certificates is of $\bigwedge\bigvee$-type and then apply Fact \ref{fac:SigmaAconsprop}, with some extra care to grant well-foundedness of the generic term model). 
If $P^*_0$ were stationary set preserving, one would be (almost) done.
However this basic set-up is not yet sufficient to argue that $P^*_0$ is stationary set preserving, just to infer that generic filters for
$P^*_0$ add semantic certificates $(\mathcal{J},r,f,T)$ for $A,D$. For the moment however let us analyze in detail the 
forcing $P^*_0$ to get accustomed with the more intricate arguments of the next section where the
chain of forcings $\bp{P_\alpha:\alpha\in C\cup\bp{\kappa}}$ will be introduced, with each one adding a semantic certificate, and with 
$P_\kappa$ 
being also stationary set preserving.

\begin{definition}\label{def:modelMnu}
Given a semantic certificate $(\mathcal{J},r,f,T)$ for $A,D$ existing in some $V[K]$,
let
\[
\mathcal{J}=\bp{j_{\alpha\beta}:N_\alpha\to N_\beta:\alpha\leq\beta\leq\omega_1^V}.
\]

$X_{(\mathcal{J},r,f,T)}$ is the set
\[ 
(\bigcup_{i\leq\omega_1^V}N_\alpha)\cup
\bp{N_\alpha:\alpha\leq\omega_1^V}\cup\bp{j_{\alpha\beta}:\alpha\leq\beta\leq\omega_1^V}
\cup\bp{G_\alpha:\alpha<\omega_1}\cup\bp{r,f}.
\]

$\mathcal{X}_{(\mathcal{J},r,f,T)}$ is the multi sorted structure of sorts $N_\alpha:\alpha\leq\omega_1^V,\bp{r,f}$ and
binary relations $\in\restriction N_\alpha$ of type $(N_\alpha,N_\alpha)$, unary relations $G_\alpha$ of type $N_\alpha$,
binary relations $j_{\alpha\beta}$ of type $(N_\alpha,N_\beta)$, unary relation $T$ of type\footnote{Denoting the extension of $T$,
which is a subset of $N_{\omega_1}$.} $N_{\omega_1}$,
binary relation $\sqsubseteq^2\restriction T$ of type\footnote{
$(s,t)\sqsubseteq^2(u,v)$ if and only if $(s,t),(u,v)\in T$, $s\sqsubseteq u$, $t\sqsubseteq v$.} 
$(N_{\omega_1},N_{\omega_1})$, ternary relation $\text{br}_T$ of type\footnote{With $\text{br}_T(x,y,z)$ holding just if
$x=(s,t)\in T$,$y=r$, $z=f$ and $s\sqsubseteq r$, $t\sqsubseteq f$.} $(N_{\omega_1},\bp{r,f},\bp{r,f})$,
ternary relation $\Cod^*$ of type\footnote{Recall $\Cod(r)=(N_0,A\cap\omega_1^{N_0})$; $\Cod^*(x,y,z)$ holds if and only if $x\in N_0$  
$y=A\cap\omega_1^{N_0}$, $z=r$.} $(N_0,N_0,\bp{r,f})$, satisfaction predicates\footnote{To be represented by the pairs $\ap{\gp{\psi},(a_1,\dots,a_m)}$ with $\psi(x_1,\dots,x_m)$ an $\in$-formula 
in displayed free variables and $(a_1,\dots,a_m)\in N_\alpha^m$ such that
\[
(N_\alpha,\in)\models\psi(x_1,\dots,x_m)[x_1/a_1,\dots,x_m/a_m].
\]
}
$\Sat_{N_\alpha}$ of type $(\dot{N}_{\alpha},\dot{N}_{\alpha})$ for $\alpha\leq\omega_1$, 
\[
\ap{N_\alpha:\alpha\leq\omega_1^V,\,\bp{r,f};\,\in\restriction N_\alpha,\Sat_{N_\alpha}:\alpha\leq\omega_1^V,\, G_\alpha:\alpha<\omega_1^V,\, 
j_{\alpha\beta}:\alpha\leq\beta\leq\omega_1^V, \sqsubseteq^2\restriction T,\,\text{br}_T,\,\Cod^*}.
\] 
%$\Cod$ is the binary relation mapping each well founded extensional realtion with a top element on $\omega$ to the top element of its transitive collapse, and $\sqsubseteq$
\end{definition}

\begin{remark}
\emph{}

\begin{itemize}
\item 
The list 
\[ 
\bp{N_\alpha:\alpha\leq\omega_1}\cup\bp{j_{\alpha\beta}:\alpha\leq\beta\leq\omega_1^V}\cup
\bp{G_\alpha,:\alpha<\omega_1}\cup\bp{r,f}
\]
gives the $\in$-maximal element of $(X_{(\mathcal{J},r,f,T)},\in)$ and the list
is not redundant, %(possibly with the exception of $H_{\omega_2},T$ which may (or most likely may not) belong to $N_{\omega_1}$),
i.e any element in the above list does not belong to any of the others, and any $\in$-maximal element
of the graph of the $\in$-relation over $(X_{(\mathcal{J},r,f,T)},\in)$ is in the above list.
\item
In essence $X_{(\mathcal{J},r,f,T)}$ is a set containing all the elements and predicates of the multisorted structure 
$\mathcal{X}_{(\mathcal{J},r,f,T)}$ except a few objects which can be easily reconstructed from the remaining ones (e.g. the sort
$\bp{r,f}$, the predicates $\Sat_{N_\alpha}$, $\text{br}_T$, $\Cod^*$).
\item
$H_{\omega_2}^V,T$ are both subsets of $N_{\omega_1^V}$ (by Condition \ref{def:semcert0-5} for the weak semantic certificate for 
$A$ as witnessed by $T$).
\item Note that $X_{(\mathcal{J},r,f,T)}$ is not transitive: if $\ap{a,b}=\bp{\bp{a},\bp{a,b}}\in j_{\alpha\beta}$,
$\bp{a}\in N_\alpha$, but $\bp{a,b}\in N_\alpha\cup N_\beta$ may not be neither in $N_\alpha$, nor in $N_\beta$. 
However with the exception of the elements in the graph of $j_{\alpha\beta}$,
all other elements of $X_{(\mathcal{J},r,f,T)}$
are contained in some $N_\alpha$.
%Note also that $N_\alpha\cap N_\beta$ is rather small, for example $N_\alpha\cap N_{\omega_1^V}\subseteq N_{\omega_1^V}$, since the 
%iteration is standard.

%all sets in $X_{(\mathcal{J},r,f,T)}$ have their transitive closure contained in some $N_\alpha$.
%\item
%Any element in $X_{(\mathcal{J},r,f,T)}$ is either in 
%$\bigcup_{\alpha\leq\omega_1^V}N_\alpha$ or a maximal element of the $\in$-relation on $X_{(\mathcal{J},r,f,T)}$ or of the form\footnote{Note that 
%Therefore the transitive closure of $j_{\alpha\beta}$ includes these types of sets which are not in $\bigcup_{\alpha\leq\omega_1^V}N_\alpha$.} 
%$\ap{a,b}$ or of the form $\bp{a,b}$ with $a\in N_\alpha,b\in N_\beta$ and $j_{\alpha\beta}(a)=b$. 
%
%In particular $X_{(\mathcal{J},r,f,T)}$ is the set
%\begin{align*}
%\bp{N_\alpha:\alpha\leq\omega_1}\cup\bp{j_{\alpha\beta}:\alpha\leq\beta\leq\omega_1^V}\cup
%\bp{G_\alpha,:\alpha<\omega_1}\cup\bp{r,f,T,H_{\omega_2}}\cup\\
%\cup\bp{\bp{a,b}, \ap{a,b}:a\in N_\alpha, b\in N_\beta: \alpha\leq\beta\leq\omega_1^V, j_{\alpha\beta}(a)=b}\cup
%(\bigcup_{\alpha\leq\omega_1^V} N_\alpha)
%\end{align*}
%\item We did not include $j_{\alpha\beta}$ among the elements of $X_{(\mathcal{J},r,f,T)}$ because
%if $j_{\alpha\beta}(a)=b$ the pair 
%$\ap{a,b}=\bp{{a},{a,b}}\in j_{\alpha\beta}$ is not in  $\bigcup_{\alpha\leq\omega_1^V}N_\alpha$ 
%if $\bp{a,b}$ is nor in $N_\alpha$ nor in $N_\beta$. Hence if we added the $j_{\alpha\beta}$ to the elements of 
%$X_{(\mathcal{J},r,f,T)}$ this set would not be transitive, and the transitive closure would be slightly more difficult to describe.
\item
$X_{(\mathcal{J},r,f,T)}$ has size bounded by $|\omega_2|^{V[K]}$.

\end{itemize}
\end{remark}

The key point is that (as we will see) the $\mathfrak{L}_{\infty,\omega}$-theory of the structures
$\mathcal{X}_{(\mathcal{J},r,f,T)}$ completely describes
the notion of semantic certificate up to type isomorphism.

Given a semantic certificate
$(\mathcal{J},r,f,T)$ for $A,D$ as witnessed by $T$ with
\[
\mathcal{J}=\bp{j_{\alpha\beta}:N_\alpha\to N_\beta:\, \alpha\leq\beta\leq\omega_1^V},
\]
$\mathcal{L}^0$ is in $V$ the multi-sorted signature with sorts,
constants, and predicate symbols to be interpreted by the relevant objects (or subsets ---  see Warning \ref{war:1str}) of 
$X_{(\mathcal{J},r,f,T)}$ which define the associated multi-sorted structure $\mathcal{X}_{(\mathcal{J},r,f,T)}$.

\begin{definition} \label{def:defL0}
$\mathcal{L}^0$ consists of the following sorts, constants, and predicate symbols:
\begin{description}
\item[Sorts]
\emph{}
\begin{enumerate}[(a)]
\item \label{def:defL0-a}
$\dot{N}_i$ to denote the extension of $N_i$ for each $i\leq\omega_1^V$,
\item  \label{def:defL0-b}
$\dot{\bp{r,f}}$ to denote the set given by the elements of the selected branch $(r,f)$ of $T$.
\end{enumerate}

\item[Constants]
\emph{}

\begin{enumerate}[(a)]
\setcounter{enumi}{2}

\item \label{def:defL0-c}
$\check{x}$ of sort $\dot{N}_{\omega_1^V}$ for $x\in H_{\omega_2}$ to denote the elements
of $H_{\omega_2}$.

\item \label{def:defL0-d}
%Extra constant 
symbols $\bp{\dot{c}_{n,i}:n\in\omega}$ of sort $\dot{N}_i$ for $i<\omega_1^V$ 
needed to denote the elements of 
$N_i$ (but only for $i<\omega_1^V$);

\item \label{def:defL0-e}
%%Extra constant 
symbols $\bp{\dot{c}_{j,\omega_1^V}:j<\omega_1^V}$ of sort 
$\dot{N}_{\omega_1^V}$;

\item \label{def:defL0-f}
symbols $\dot{r},\dot{f}$ of sort $\dot{\bp{r,f}}$ to be interpreted in some $V[G]$ 
by the selected branch $(r,f)$ through $T$ 
such that $\Cod(r)=(N_0,a_0)$.

\end{enumerate}

\item[Predicates]
%\item \label{def:defL0-d}
%%Extra constant 
%symbols $\dot{\omega_1}^{i}$ for any $i<\omega_1^V$ needed to denote the 
%$N_i$-cardinals $\omega_1^{N_i}$;
\emph{}

\begin{enumerate}[(a)]
\setcounter{enumi}{6}

\item \label{def:defL0-g}
A unary predicate 
symbol symbol $\dot{T}$ of type $\dot{N}_{\omega_1}$
to denote the extension of the ($2$-dimensional) witnessing tree $T\subseteq N_{\omega_1}$;

%\item \label{def:defL0-f}
%Unary predicate 
%symbols $\dot{N}_i$ to denote the extension of $N_i$ for each $i\leq\omega_1^V$;

\item \label{def:defL0-h}
Unary predicate 
symbols $\bp{\dot{G}_{j}:j<\omega_1}$ (each of type $\dot{N}_j$ for $j<\omega_1^V$) to denote the $N_\alpha$-generic filters
such that $N_{\alpha+1}$ is the ultrapower of $N_\alpha$ by $G_\alpha$;

\item \label{def:defL0-i}
Binary predicate symbols
$\dot{j}_{\alpha\beta}$ of type $(\dot{N}_\alpha,\dot{N}_\beta)$ to denote $j_{\alpha\beta}$ for 
$\alpha\leq\beta\leq\omega_1^V$;

\item \label{def:defL0-j}
A ternary predicate symbol $\text{br}_T$ of type $(\dot{N}_{\omega_1},\dot{\bp{r,f}},\dot{\bp{r,f}})$
to denote the initial segment relation between elements of $T$ and the branch $(r,f)$;

\item \label{def:defL0-jbis}
For each $\alpha\leq\omega_1^V$ a binary predicate $\in_{N_\alpha}$ of type $(\dot{N}_{\alpha},\dot{N}_{\alpha})$ to denote the $\in$-relation 
restricted to $N_\alpha$;

\item \label{def:defL0-k}
For each $\alpha\leq\omega_1^V$ 
a satisfaction predicate
\[
\Sat_{N_\alpha}(\gp{\psi(x_1,\dots,x_m)})
\] 
of type\footnote{It is clear that the interpretation of this predicate symbol will subsume the interpretation of 
$\in_{N_\alpha}$ by conisdering the formula $\Sat_{N_\alpha}(\gp{x\in y})$. However it is convenient (just for notational 
simplicity) to have a special symbol to denote $\in\restriction N_\alpha$.} $(\dot{N}_{\alpha},\dot{N}_{\alpha})$ 
to be represented by the pairs $\ap{\gp{\psi},(a_1,\dots,a_m)}$ with $\psi(x_1,\dots,x_m)$ an $\in$-formula 
in displayed free variables and $(a_1,\dots,a_m)\in N_\alpha^m$ such that
\[
(N_\alpha,\in)\models\psi(x_1,\dots,x_m)[x_1/a_1,\dots,x_m/a_m];
\]
%\item \label{def:defL0-k}
%A satisfaction predicate
%\[
%\Sat_{\dot{Q}}(\gp{\psi(x_1,\dots,x_m)})
%\] 
%of arity $2$ to be represented by the pairs $\ap{\gp{\psi},(a_1,\dots,a_m)}$ with $\psi(x_1,\dots,x_m)$ an $\in$-formula in displayed free variables and $(a_1,\dots,a_m)\in H_\kappa$ such that
%\[
%(H_\kappa,\in)\models\psi(x_1,\dots,x_m)[x_1/a_1,\dots,x_m/a_m].
%\]
\item \label{def:defL0-l}
A ternary predicate symbol $\Cod^*$ of type $(\dot{N}_0,\dot{N}_0,\dot{\bp{r,f}})$ to encode the fact that $\Cod(r)=(N_0,a)$ by holding
just of the triples $x,y,z$ with $y=A\cap\omega_1^{N_0}$ and $z=r$.
\end{enumerate}
\end{description}

\end{definition}

Note that $\mathcal{L}^0\in V$ has size $\omega_2$ in $V$.
%Note also that (for our convenience) we consider the set $N_0$ both as the extension of 
%a predicate (using $\dot{N}_0$) and as a constant (using $\dot{N}_0^*$).

\begin{definition}\label{def:keydefnu}
Let $(\mathcal{J},r,f,T)$ be a weak semantic certificate for $A,D$ with
\[
\mathcal{J}=\bp{j_{\alpha\beta}:N_\alpha\to N_\beta:\, \alpha\leq\beta\leq\omega_1^V}.
\] 
%Let $\bar{\mathcal{L}}^0$ be the signature removing from $\mathcal{L}^0$ the predicate symbols listed in clauses
%\ref{def:defL0-i}, \ref{def:defL0-j}, \ref{def:defL0-k}, \ref{def:defL0-l}.
A map $\nu$ with domain $\mathcal{L}^0$
%\[
%\nu:\mathcal{K}^0\cup\mathcal{R}^0\to V
%\]  
is an \emph{admissible interpretation} for $(\mathcal{J},r,f,T)$
if it respects the following constraints on the interpretations of constants and predicates:
\begin{description}
\item[Sorts]

\emph{}

\begin{enumerate}[(a)]

\item \label{def:keydefnu-a}
$\nu(\dot{N}_i)=N_i$ for each $i\leq\omega_1^V$.

\item \label{def:keydefnu-b}
$\nu(\dot{\bp{r,f}})=\bp{r,f}$.
\end{enumerate}
\item[Constant symbols]

\emph{}

\begin{enumerate}[(a)]
\setcounter{enumi}{2}

\item \label{def:keydefnu-c}
$\nu(\check{x})=x$
 for every $x\in H_{\omega_2}$.% with $T$ as fixed in Convention \ref{conv:convobjP0Pkappa-1}.

\item \label{def:keydefnu-d}
$\nu[\bp{\dot{c}_{n,\alpha}:n\in\omega}]=N_\alpha$ and $\nu(c_{0,\alpha})=A\cap\omega_1^{N_\alpha}$ for all $\alpha<\omega_1^V$. 

\item \label{def:keydefnu-e}
$\nu[\bp{\dot{c}_{j,\omega_1^V}:j\in\omega_1^V}]=N_{\omega_1^V}$  and $\nu(c_{0,\omega_1^V})=A$.

\item \label{def:keydefnu-f}
$\nu(\dot{r})=r\in p[T]$,
$\nu(\dot{f})=f$ with $(r,f)$ a branch of $T$ and $\Cod(r)=(N_0,A\cap\omega_1^{N_0})$.
\end{enumerate}

\item[Predicates]
\emph{}

\begin{enumerate}[(a)]
\setcounter{enumi}{6}

\item \label{def:keydefnu-g}
$\nu(\dot{T})=T$.

\item \label{def:keydefnu-h}
$\nu(\dot{G}_\alpha)=G_\alpha$ for 
$\alpha<\omega_1^V$.

\item \label{def:keydefnu-i}
$\nu(\dot{j}_{\alpha\beta})=j_{\alpha\beta}$
for all $\alpha\leq\beta\leq\omega_1^V$.

\item \label{def:keydefnu-j}
$\nu(\text{br}_T)=\bp{\ap{(s,t),r,f}:\, s\sqsubseteq r, t\sqsubseteq f}$.

\item  \label{def:keydefnu-k}
$\nu(\Sat_{N_\alpha})$ as prescribed by 
\ref{def:defL0-k} of Def. \ref{def:defL0}.

\item \label{def:keydefnu-l}
$\nu(\Cod^*)=\bp{\ap{x,A\cap\omega_1^{N_0},r}:\, x\in N_0}$.

%\item \label{def:keydefnu-i}
%Assignemnt as prescribed by \ref{def:defL0-i}.
%
%\item \label{def:keydefnu-j}
%Assignemnt as prescribed by \ref{def:defL0-j}.
%
%\item \label{def:keydefnu-k}
%Assignemnt as prescribed by \ref{def:defL0-k}.
%
%
%\item \label{def:keydefnu-k}
%Assignemnt as prescribed by \ref{def:defL0-l}.

\end{enumerate}
\end{description}

\begin{warning}\label{warn:signXJrfT}
%The predicate symbols $\Sat_{N_\alpha}$ of type $(\dot{N}_{\alpha},\dot{N}_{\alpha})$ for $\mathcal{L}^0$ 
%given in clause \ref{def:defL0-k} will be given by the natural extension of $\nu$ 
%to $\mathcal{L}^0$ their expected meaning (i.e. the one specified for them in the respective defining clauses).
%To avoid a cumbersome notation we will only pay attention to the definition of $\nu$ for the symbols interpreted by elements of 
%$X_{(\mathcal{J},r,f,T)}$, also we will consider $\mathcal{X}_{(\mathcal{J},r,f,T)}$ as an $\mathcal{L}^0$-structure, 
%while on the face of its definition it is only an $\bar{\mathcal{L}}^0$-structure. 
A key point is that once we fixed the structure $\mathcal{X}_{(\mathcal{J},r,f,T)}$
which is the target of $\nu$,
we have a certain freedom in changing the value of the assignments on constant symbols (for example for all those of the form
$\dot{c}_{n,\alpha}$ for $n>0$ and $\alpha\leq\omega_1^V$), 
while maintaining that the (so modified assignment) 
is still admissible;
but we have no freedom in changing the interpretation of the predicate symbols. 
In particular many of the constant symbols could also be considered sorted variables.
%From now on we will extend any admissible assignment $\nu$ defined only on  $\bar{\mathcal{L}}^0$ to an assignment
%defined on the whole $\mathcal{L}^0$, by interpreting the missing relation symbols according to the clauses 
%\ref{def:defL0-i}, \ref{def:defL0-j}, \ref{def:defL0-k}, \ref{def:defL0-l} defining them.
\end{warning}

\begin{itemize}
\item
Given an admissible assignment $\nu$,
$\Sigma_\nu$ is the set of $\mathfrak{L}_{\infty,\omega}$-sentences  for $\mathcal{L}^0$  realized in the structure
$\mathcal{X}_{(\mathcal{J},r,f,T)}$
by the assignment $\nu$.
\item
A set $\Sigma$ of $\mathfrak{L}_{\infty,\omega}$-sentences  for $\mathcal{L}^0$
is \emph{generically consistent} if in some generic extension $V[G]$ of $V$
there is a semantic certificate $(\mathcal{J},r,f,T)$ for $A,D$, 
and an admissible assignment $\nu:\mathcal{L}^0\to X_{(\mathcal{J},r,f,T)}$
such that $\Sigma\subseteq\Sigma_\nu$.
\end{itemize}
\end{definition}

The first key observation is that the $\mathcal{L}^0$-isomorphism types of a semantic certificate define an elementary class for 
$\mathfrak{L}_{\infty,\omega}$ which is axiomatized by a \emph{standard theory} (cfr. Def. \ref{def:standtheory}).

To accomodate our $\mathfrak{L}_{\infty,\omega}$-axiomatization of the notion of semantic certificate it is convenient to detail more the nature
of the map $\Cod:2^{\omega}\to H_{\omega_1}$.
\begin{convention}\label{conv:propCod}
We assume the following properties of $\Cod$ in case $\Cod(r)=(N,a)\in \Pmax$:
\begin{itemize} 
\item
$r$ is identified (modulo a recursive bijection $n\mapsto\ap{k_n,j_n}$ of $\omega$ with $\omega^2$)
with a well-founded extensional relation $E_r$ with domain $\omega$;
\item
the Mostowski collapsing map $\pi$ of the structure $(\omega,E_r)$ has as range $\trcl{\bp{\ap{N,a}}}$ and maps 
$0$ to $N$, $1$ to $\bp{N}$, $2$ to $\bp{N,a}$, $3$ to $\ap{N,a}$, $4$ to $\bp{\ap{N,a}}$, $5$ to $a$.
\end{itemize}
\end{convention}

\begin{fact}\label{fac:keyfacSEMCERTP0}
Let $\Sigma^0$ be the  $\mathcal{L}^0$ theory for $\mathfrak{L}_{\infty,\omega}$
given by the $\mathcal{L}^0$-equality axioms and the 
$\mathcal{L}^0$-quantifier elimination axioms (cfr. Def. \ref{def:standtheory}),
%\begin{enumerate}
%\item The equality axioms for $\mathcal{L}^0$:
%\begin{itemize}
%\item $\qp{(c=d)\wedge(d=e)}\rightarrow (c=e)$, 
%$(c=c)$, $(c=d)\rightarrow(d=c)$,
%\item $\qp{R(c_1,\dots,c_n)\wedge\bigwedge_{i=1}^n(c_i=d_i)}\rightarrow R(d_1,\dots,d_n)$
%\end{itemize}
%for all relation symbols $R$ of $\mathcal{L}^0$ and for all constant symbols $c,d,e$ of  $\mathcal{L}^0$.
%\item The quantifier elimination axioms for $\mathcal{L}^0$:
%\begin{itemize}
%\item
%$\bigvee_{c\in\mathcal{L}^0}\psi(c)\leftrightarrow \exists x\psi(x)$ for all $\mathcal{L}^0$-formulae $\psi(x)$ in one free variable,
%\item
%$\forall x\bigvee_{c\in\mathcal{L}^0}\,(x=c)$.
%\end{itemize}
%\end{enumerate}
and the following list of axioms:
\begin{enumerate}[(I)]
\item \label{fac:keyfacSEMCERTP0-V}
Axioms to characterize the extensions of $\dot{T}$ and of $H_{\omega_2}^V$ inside $\dot{N}_{\omega_1}$:
%\item \label{fac:keyfacSEMCERTP0-1}
%$\forall(x_i:i<\omega)\bigvee_{i<\omega} \neg x_{i+1}\in x_i$.
\begin{enumerate}
\item \label{fac:keyfacSEMCERTP0-2-V}
For all $x,y\in H_{\omega_2}$ 
\[
(\check{x}\in \check{y})\text{ if $x\in y$,}
\] 
\[
(\check{x}\not\in \check{y}) \text{ if $x\not\in y$.} 
\]
\item \label{fac:keyfacSEMCERTP0-17-VJ}
The infinitary conjunction of infinitary disjunctions
\[
\bigwedge_{x\in H_{\omega_2}}(\bigvee_{j<\omega_1^V}\check{x}=\dot{c}_{j,\omega_1}).
\]
%\item \label{fac:keyfacSEMCERTP0-3-V}
%For all constant symbols $c$ 
%\[
%\dot{Q}(c)\leftrightarrow\bigvee_{x\in H_\kappa}(c=\check{x}) 
%\]

\item \label{fac:keyfacSEMCERTP0-3-T}
For all constant symbols $c$ of sort $\dot{N}_{\omega_1^V}$ 
\[
\dot{T}(c)\leftrightarrow\bigvee_{(s,t)\in T}(c=\check{\ap{s,t}}) 
\]
\end{enumerate}

\item \label{fac:keyfacSEMCERTP0-J}
The following axioms to describe the properties of an iteration $\mathcal{J}$:
\begin{enumerate}
%\item \label{fac:keyfacSEMCERTP0-4-J}
%%For all constant symbols $c$ of $\mathcal{L}^0$ and $i<\omega_1^V$
%\[
%\bigwedge_{c\in\mathcal{L}^0,i<\omega_1}\qp{\dot{N}_i(c)\leftrightarrow [\bigvee_{n<\omega}c=\dot{c}_{n,i}]}.
%\]
%\item \label{fac:keyfacSEMCERTP0-5-J}
%%For all constant symbols $c$ of $\mathcal{L}^0$
%\[
%\bigwedge_{c\in\mathcal{L}^0}\qp{\dot{N}_{\omega_1^V}(c)\leftrightarrow [\bigvee_{j<\omega_1^V}c=\dot{c}_{j,\omega_1^V}]}.
%\]
\item \label{fac:keyfacSEMCERTP0-7-J}
For all $\alpha\leq \beta\leq\omega_1^V$
\[
%\bigwedge_{\alpha< \beta\leq\omega_1^V}\qp{
\bigwedge_{n<\omega}\bigvee_{j<\omega_1^V}[\dot{j}_{\alpha\omega_1^V}(\dot{c}_{n,\alpha})=\dot{c}_{j,\omega_1^V}],
\]
\[
\bigwedge_{n<\omega}\bigvee_{m<\omega}(\dot{j}_0(\dot{c}_{n,\alpha})=\dot{c}_{m,\beta}).
\]
\item \label{fac:keyfacSEMCERTP0-11-J}
For all $\alpha<\beta\leq\omega_1^V$ and constant symbols $\dot{c}_{n,\alpha},\dot{c}_{m,\beta}$ for $n,m<\omega$
\[
\qp{\dot{G}_\alpha(\dot{c}_{n,\alpha})\wedge \dot{j}_{\alpha\beta}(\dot{c}_{n,\alpha})=\dot{c}_{m,\beta})}\leftrightarrow
\qp{\Sat_{N_\beta}(\gp{\omega_1\in \dot{c}_{m,\beta}})\wedge 
\Sat_{N_\alpha}(\gp{\dot{c}_{n,\alpha}\subseteq\omega_1})}.
\] 
\item \label{fac:keyfacSEMCERTP0-11-Jbis}
For all $\alpha<\beta\leq\omega_1^V$ and constant symbols $\dot{c}_{m,\beta}$ for\footnote{This sentence states that 
$j_{\alpha\beta}[N_\alpha]=
\bp{j_{\alpha\beta}(f)(\omega_1^{N_\alpha}): f\in N_\alpha, \dom{f}=\omega_1^{N_\alpha}}$.} 
$m<\omega$
\[
\bigvee_{n,k<\omega}\qp{(\dot{c}_{k,\beta}=\dot{j}_{\alpha\beta}(c_{n,\alpha}))\wedge\Sat_{N_\beta}(\gp{\dot{c}_{k,\beta}(\omega_1)=c_{m,\beta}})}.
\]
\item \label{fac:keyfacSEMCERTP0-12-J}
For all $\alpha<\omega_1^V$
\[
\bigwedge_{n<\omega}
\Sat_{N_\alpha}(\gp{\dot{c}_{n,\alpha}\text{ is a dense subset of }\pow{\omega_1}/_{\NS_{\omega_1}}})
\rightarrow \bigvee_{m<\omega}\,(\dot{G}_\alpha(\dot{c}_{m,\alpha})\wedge \dot{c}_{m,\alpha}\in\dot{c}_{n,\alpha}).
\]
%\item \label{fac:keyfacSEMCERTP0-13-J}
%For all $\alpha<\omega_1^V$ 
%\[
%\dot{j}_{0\omega_1^V}(\check{a}_\alpha)=\check{a}_{\beta}.
%\]
%\item \label{fac:keyfacSEMCERTP0-14-J}
%For all $\alpha<\omega_1^V$ and constant symbol $c\in\mathcal{L}^1$ the sentences
%\[
%\dot{N}_\alpha(c)\wedge\dot{N}_{\omega_1^V}(c)\rightarrow \Sat_{N_{\omega_1^V}}(\gp{c\in H_{\omega_1}}).
%\]
%\[
%\Sat_{N_{\alpha}}(\gp{c \text{ is an ordinal}})\rightarrow c\in\omega_1^V.
%\]
%\[
%\dot{N}_\alpha(c)\wedge\dot{Q}(c)\rightarrow \Sat_{\dot{Q}}(\gp{c\in H_{\omega_1}}).
%\]
\end{enumerate}

\item \label{fac:keyfacSEMCERTP0-TJ}
The following sentences describing the properties an iteration $\mathcal{J}$, a real $r$ and an $f:\omega\to\kappa$
must satisfy to grant that $(\mathcal{J},r,f,T)$ is a semantic certificate:
\begin{enumerate}

\item \label{fac:keyfacSEMCERTP0-14-TJ}
The sentence
\[
\bigwedge_{n<\omega}
\bigvee_{(s,t)\in T\cap (2\times\omega_2^V)^n} \text{br}_T(\ap{s,t},\dot{r},\dot{f})
\]
\[
\bigwedge_{(s,t),(u,v)\in T} 
\qp{\text{br}_T(\check{\ap{s,t}},\dot{r},\dot{f})\wedge\text{br}_T(\check{\ap{u,v}},\dot{r},\dot{f})}\rightarrow 
(\check{\ap{s,t}}\sqsubseteq^2\check{\ap{u,v}}\vee\check{\ap{u,v}}\sqsubseteq^2\check{\ap{s,t}})
\]
%\item \label{fac:keyfacSEMCERTP0-15-TJ}
%The sentence
%\[
%\dot{j}_{0\omega_1^V}(c)=\check{A}.
%\]
\item \label{fac:keyfacSEMCERTP0-16-TJ}
The axioms\footnote{These axioms formalize that $\Cod(r)=(N_0,a_0)$ by imposing that $(N_0,\in)$ is isomorphic to
$E_r\restriction[5;\infty)$ in accordance with notation \ref{conv:propCod}, and that $c_{0,0}$ is assigned to $A\cap\omega_1^{N_0}$.}
\[
\bigwedge_{n<\omega}\Cod^*(c_{n,0},c_{0,0},\dot{r}),
\]
\[
\bigwedge_{n\in\omega,k_n,j_n\geq 5,\ap{s,t}\in T} (\text{br}_T(\check{\ap{s,t}},\dot{r},\dot{f})\wedge \check{s(n)}=\check{1})\rightarrow
\Sat_{N_0}(\gp{c_{k_n-5,0}\in c_{j_n-5,0}}),
\]
\[
\bigwedge_{n\in\omega,k_n,j_n\geq 5,\ap{s,t}\in T} (\text{br}_T(\check{\ap{s,t}},\dot{r},\dot{f})\wedge \check{s(n)}=\check{0})\rightarrow
\Sat_{N_0}(\gp{c_{k_n-5,0}\not\in c_{j_n-5,0}}).
\]
%together with the axiom
%\[
%\bigwedge_{c\in\mathcal{L}^0}\qp{(c\in \dot{N}_0^*)\leftrightarrow \dot{N}_0(c)}.
%\]
\item \label{fac:keyfacSEMCERTP0-10-J}
The sentence
\[
\dot{j}_{0\omega_1^V}(\check{c}_{0,0})=A.
\]

\item \label{fac:keyfacSEMCERTP0-1-V}
For\footnote{This axioms grants that $(H_{\omega_2},\in)$ sits inside $(N_{\omega_1}^V,\in)$ as a transitive subclass.} 
all constant symbols $c$ of sort $\dot{N}_{\omega_1}$ and $y\in H_{\omega_2}$
\[
c\in \check{y}\leftrightarrow\bigvee_{x\in y}\check{x}=c.
\]

\item \label{fac:keyfacSEMCERTP0-18-VJ}
The infinitary conjunction 
\[
\bigwedge_{S\in(\pow{\omega_1}\setminus \NS_{\omega_1})^V} %\qp{\Sat_{\dot{Q}}(\gp{\dot{S}\text{ is stationary}})\leftrightarrow 
\Sat_{N_{\omega_1^V}}(\gp{\dot{S}\text{ is stationary}}).
\] 
\end{enumerate}
\end{enumerate}

Then:
\begin{enumerate}
\item
$\Sigma^0$ is standard (cfr. Def. \ref{def:standtheory}).
\item 
$\Sigma^0$ is realized by any admissible assignement $\nu$.
\item 
Whenever
$\mathcal{N}=\ap{(\mathcal{S}^0)^\mathcal{N},(\mathcal{R}^0)^\mathcal{N},(\mathcal{K}^0)^\mathcal{N}}$ is a 
$\mathcal{L}^0$-model of 
$\Sigma^0$ existing in some transitive model $W$, we have that:
\begin{itemize}
\item
 The domain of each sort of $(\mathcal{S}^0)^\mathcal{N}$ has as extension the 
interpretations of the constant symbols of $\mathcal{K}^0$ of that sort.
\item
The sort $\dot{N}_\alpha$ of $\mathcal{N}$ is well founded and extensional for all $\alpha\leq\omega_1^V$.
\item 
The transitive collapse of the sorts $\dot{N}_\alpha$ of $\mathcal{N}$ define 
a semantic certificate $(\mathcal{J},r,f,T)$ for $A,D$ existing in $W$ with $r(n)=i$ and $f(n)=j$ if and only if for some
$(s,t)\in T$ with $s(n)=i$ and $f(n)=j$, $\text{br}_T(\check{\ap{s,t}},r,f)$.
\end{itemize}
\end{enumerate}
\end{fact}

%Let us point out that all the formulae appearing in the above list of axioms are $\in$-formulae unless they are of type
%$\dot{j}_{\alpha\beta}(x)=y$ in which case they are atomic $\mathcal{L}^1$-formulae.

\begin{proof}
The proof that $\Sigma^0$ is standard and that 
$\Sigma^0\subseteq\Sigma_\nu$ for any admissible assignment $\nu$ is left to the reader.
It consists in the tyresome checking that all the above sentences are of $\bigwedge\bigvee$-type and are in $\Sigma_\nu$ for any 
admissible assignment $\nu$.

Let $W$ be a transitive model of $\ZFC$ and 
$\mathcal{N}$ be in $W$ a structure satisfying $\Sigma^0$.
The axiom for quantifier elimination grants that $\mathcal{N}$ has as domain of its sorts exactly 
the interpretation of the constant symbols of 
$\mathcal{L}^0$ of that sort.
% and such that its domain is 
%$\bp{c^\mathcal{M}:c\in\mathcal{L}^0}$.
The axioms of Type %\ref{fac:keyfacSEMCERTP0-V} 
\ref{fac:keyfacSEMCERTP0-V} %, \ref{fac:keyfacSEMCERTP0-2-V}
grant that 
\[
(\bp{\check{x}^{\mathcal{N}}: x\in H_{\omega_2}},\in_{N_{\omega_1^V}}^{\mathcal{N}})
\]
is isomorphic to $(H_{\omega_2},\in)$ and also that
\[
(\bp{c^{\mathcal{N}}: \mathcal{N}\models\dot{T}(c)},(\sqsubseteq^2)^{\mathcal{N}})
\]
is isomorphic to
\[
(T,\sqsubseteq^2).
\]
These axioms combined with axioms \ref{fac:keyfacSEMCERTP0-14-TJ}
grant that:
\begin{itemize}
\item
The unique infinite sequence $(r,f)\in V[G]$
given by $r(n)=i$ and $f(n)=\alpha$ if 
\[
\mathcal{N}\models\dot{r}(\dot{n})=\dot{i}\wedge\dot{f}(\dot{n})=\check{\alpha}
\] 
is a branch through $T$.
Therefore in $V[G]$ $\Cod(r)=(N_0,a_0)$ with $a_0=A\cap\omega_1^{N_0}$.
% and $N_0$ may admit at most one iteration 
%(unique by \cite[Lemma 2.7]{HSTLARSON})
%\[
%\mathcal{J}=\bp{j_{\alpha\beta}:\alpha\leq\beta\leq\omega_1^V}
%\]
%such that $j_{0\omega_1}(a_0)=A$. 
\item
The structure
 \[
(\bp{c_{n,0}^{\mathcal{N}}: n\in\omega},\in_{\mathcal{N}^0}^{\mathcal{N}})
\]
is well-founded and its transitive collapse is exactly $N_0$ by Axioms \ref{fac:keyfacSEMCERTP0-16-TJ}.
\end{itemize}
%
%This gives that $\mathcal{M}$ models 
%\[
%\Sat_{N_0}(\gp{\NS_{\omega_1} \text{ is precipitous }+\MA+\ZFC}
%\]
%while
%$(N_0,a_0)$ is really in $V[G]$ a $\P_\max$-condition in $p[T]$.
Now the 
axioms of type \ref{fac:keyfacSEMCERTP0-J} 
%\ref{fac:keyfacSEMCERTP0-5-J},\dots,\ref{fac:keyfacSEMCERTP0-13-J} 
grant that an iteration $\mathcal{J}$ of $N_0$ can be defined so to extend
the partial $\in$-isomorphism $\Theta$ so far defined between the multi-sorted structures
$(H_{\omega_2}^V, N_0;\in\restriction H_{\omega_2}, \in\restriction N_{0})$ and
\[
(\bp{\check{x}: x\in H_{\omega_2}^V},\bp{\dot{e}_{n,0}:n\in\omega},\in_{N_{\omega_1^V}},\in_{N_0})
\]
%\mathcal{M}\models \dot{c}\in \dot{Q}\vee c\in\dot{N}_0\vee 
%c=\dot{N}_0\vee c=\dot{T}\vee c=\dot{Q}
%\vee c=\dot{r}\vee c=\dot{f}},\]
to
a full isomorphism $\Theta:\mathcal{X}_{(\mathcal{J},r,f,T)}\to\mathcal{N}$
with 
\[
\mathcal{J}=\bp{j_{\alpha\beta}:N_\alpha\to N_\beta:\,\alpha\leq\beta<\omega_1^V}
\] 
the unique iteration of $N_0$ of length $\omega_1^V$ such that $j_{0\omega_1}(A\cap\omega_1^{N_0})=A$.

Specifically these axioms impose inductively  the following for all $\alpha<\omega_1^V$:
\begin{itemize} 
\item
$(\bp{c_{j,\alpha}^\mathcal{N}:j\in\omega},\in_{N_\alpha}^{\mathcal{M}})$ is isomorphic to $(N_\alpha,\in)$,
\item
The Mostowski collapsing map of $\bp{c_{j,\alpha}^\mathcal{N}:j\in\omega}$ onto $N_\alpha$, 
maps $\bp{c^\mathcal{N}:\mathcal{N}\models \dot{G}_\alpha(c)}$ onto some $G_\alpha$ which is $N_\alpha$-generic for 
$(\pow{\omega_1}/_{\NS_{\omega_1}})^{N_\alpha}$.
\end{itemize}

This occurs because of 
Axioms \ref{fac:keyfacSEMCERTP0-11-J}, \ref{fac:keyfacSEMCERTP0-11-Jbis} holding in $\mathcal{M}$.
Using the fact that there is at most one iteration
in $V[G]$ satisfying the constraints given by these axioms and using \cite[Lemma 2.7]{HSTLARSON}
we conclude that $\Theta$ extends to a full isomorphism of the
$\mathcal{L}^0$-structure 
\begin{equation}\label{eqn:iterJ}
(N_i:i\leq\omega_1^V,\,
\bp{r,f};\,\,\in_{N_i}:i\leq\omega_1,\,G_i:i<\omega_1^V,\, j_{\alpha\beta}:\alpha\leq\beta\leq \omega_1^V,\,\Sat_{N_\alpha}:\alpha\leq\omega_1^V)
\end{equation}
with $\ap{(\mathcal{S})^\mathcal{N},(\mathcal{R}\setminus\bp{\dot{T},\Cod^*,\text{br}_T})^\mathcal{N}}$. 
%\begin{align*}
%H_{\omega_2}\cup\bp{T,r,f,H_{\omega_2}}\cup(\bigcup_{\alpha\leq\omega_1^V}N_\alpha)
%%\cup\bp{\bp{a,b}, \ap{a,b}:a\in N_\alpha, b\in N_\beta: \alpha\leq\beta\leq\omega_1^V, j_{\alpha\beta}(a)=b}
%%\cup\\
%\cup\bp{N_i:i\leq\omega_1^V}\cup\bp{G_i:i<\omega_1^V}%\cup\bp{j_{\alpha\beta}:\,\alpha\leq\beta<\omega_1^V}
%\end{align*}
Now the structure given by (\ref{eqn:iterJ}) uniquely determines the iteration
\[
\mathcal{J}=\bp{j_{\alpha\beta}:N_\alpha\to N_\beta:\,\alpha\leq\beta\leq\omega_1^V}.
\] 
%Axioms \ref{fac:keyfacSEMCERTP0-14-J} grant that $\mathcal{J}$ is standard.

Finally use the type \ref{fac:keyfacSEMCERTP0-TJ} axioms to get that
\begin{itemize}
\item
$A\cap\omega_1^{N_0}\in N_0$ and $j_{0\omega_1}(A\cap\omega_1^{N_0})=A$
(by Axioms \ref{fac:keyfacSEMCERTP0-16-TJ} and \ref{fac:keyfacSEMCERTP0-10-J}),
\item
$H_{\omega_2}^V\subseteq N_{\omega_1^V}$  (by Axiom \ref{fac:keyfacSEMCERTP0-1-V}),
\item
$\NS_{\omega_1}^V=\NS_{\omega_1}^{N_{\omega_1^V}}\cap V$  (by Axiom \ref{fac:keyfacSEMCERTP0-18-VJ}). 
\end{itemize}
This shows that $(\mathcal{J},r,f,T)$ is a semantic certificate for $A,D$ as witnessd by $T$ in $V[G]$.
\end{proof} 

\begin{remark}
In this proof we use crucially the fact that a semantic certificate existing in $V[G]$ is witnessed by a branch through
$T$. Notice that in principle $\dot{\in}_{N_{\omega_1}}^{\mathcal{N}}$ could be such that
$\dot{B}^{\mathcal{N}}$ (whatever definition one may try to give to $\dot{B}^{\mathcal{N}}$) and $B^{V[G]}$ 
may not have much in common. 
However this is not the case for the properties which can be described by $\mathcal{L}^0$-sentences in parameter 
$\dot{T}$.
This is one of the reasons to define the tree $T$.

Note also that we forgot to add the extensionality axiom for the interpretation of the $\in$-relation in the above list; 
so either one adds it, or one oberves that for the 
models of $\Sigma^0$ the domain is exactly given by the interpretation of the constants, hence extensionality comes for free since
the axioms grants that extensionality holds for the image under the transitive collapse of the $\in_{N_\alpha}^{\mathcal{N}}$-relation and
this map is injective.
% isomorphism of $T$ with the extension of $[\dot{T}]_H$ in $M^*_H$ allows to assert the 
%well-foundedness of $\in^H$ for all the parts of $M^*_H$ of interest to us.
\end{remark}

\begin{definition}\label{def:keydefP0}
Let $\mathcal{A}$ be the family of admissible assignments according to Def. \ref{def:defL0}.
$P^*_0$ is $P_\mathcal{A}$. 
\end{definition}
%$P^*_0$ is given by the finite familes\footnote{For the purposes of this paper we could consider arbitrary families 
%existing in $V$ of generically consistent sentences. For all our arguments it doesn't make any difference. 
%However we stick to finite for notational convenience. Note that Asper\`o and Schindler's result 
%shows that there is a semantic certificate in $V$ assuming $\MM^{++}$, 
%hence assuming this hypothesis this second proposed 
%version of the forcing becomes trivial.} $p$  of $\mathcal{L}^0$-sentences 
%such that $p$ is generically consistent.
%
%$p$ refines $q$ in $P^*_0$ if $q\subseteq p$.

%\begin{remark}
%%The permitted formulae enucleates an (almost) minimal list
%%of basic elementary formulae needed to check that the natural term model
%%for $\mathcal{L}^0$ induced by 
%%any generic filter by $P^*_0$ satisfies the axioms of
%%$\Sigma^0$ %(to check this we need all the formulae appearing in the list of permitted formulae after the first one), 
%%and the equality axioms. 
%%This will be shown in details in Lemma \ref{lem:admGenstructureP0} below.
%
%$P^*_0$ defines a consistency property according to \cite{KEIINFLOG}.
%%For reasons which will become transparent in the next section, it is convenient to consider just the minimal amount of basic 
%%formuale needed to characterize the type of a semantic certificate.
%\end{remark}

The key consequence of Fact \ref{fac:keyfacSEMCERTP0}, %Fact \ref{fac:SIGMA*bigwedgebigvee}, 
Lemma \ref{lem:keyconspropforcing}, and Fact \ref{fac:SigmaAconsprop} 
is the following:
\begin{lemma}\label{lem:admGenstructureP0}
%Let $\mathcal{M}_{\dot{H}}$ be a canonical $P^*_0$-name such that for all $H$
%$V$-generic filter for $P^*_0$
%\[
%\mathcal{M}_H=(M_H,\in^H)
%\]
%with $[c]_H=\bp{d\in \mathcal{L}^0:\, \bp{(d=c)}\in H}$, $[c]_H\in^H [d]_H$ if and only if $\bp{c\in d}$ in $H$,
%and $M_H=\bp{[c]_H:c\in \mathcal{L}^0}$.
Given $H$ $V$-generic for $P^*_0$,
$\mathcal{M}_H$ models $\Sigma^0$.

Therefore $\mathcal{M}_H$  is well founded and its unique isomorphism with $\mathcal{X}_{(\mathcal{J}_H,r_H,f_H,T)}$ gives
a uniquely defined semantic certificate $(\mathcal{J}_H,r_H,f_H,T)$ for $A,D$. 

Moreover the map
$\nu_H:c\mapsto \pi_{\alpha,H}([c]_H)$ (where $\pi_{\alpha,H}:\dot{N}_\alpha^{\mathcal{M}_H}\to X_{(\mathcal{J}_H,r_H,F_H,T)}$ is the Mostowski collapse of $\in_{N_\alpha}^H$)
is an admissible interpretation on the constants of $\mathcal{L}^0$ and can be naturally extended to a full interpretation of 
$\mathcal{L}^0$ by mapping $\nu_H(X)=\nu_H\qp{X^{\mathcal{M}_H}}$ for $X$ predicate symbol of $\mathcal{L}^0$.
\end{lemma}
\begin{proof}
By Lemma \ref{lem:keyconspropforcing} and Fact \ref{fac:SigmaAconsprop},
$\mathcal{M}_H$ realizes all sentences of type $\bigwedge\bigvee$ holding in $\Sigma_\nu$ for all admissible $\nu$.
$\Sigma^0$ is a subset of this family of sentences by Fact \ref{fac:keyfacSEMCERTP0}.
Hence $\mathcal{M}_H$ models $\Sigma^0$.

The remaining items of the Lemma are almost self-evident.
\end{proof}

By Lemma \ref{lem:keylemASPSCH(*)0} $P^*_0$ is non-empty; the problem is to check
whether it is also stationary set preserving. This natural question has a negative answer since $P^*_0$ has size
$\aleph_2$ and Woodin has shown that $\MM^{++}(2^{\aleph_0})$ is consistent with the negation of $(*)$ (cfr. \cite[Thm. 10.90]{woodinbook}).
To sidestep this difficulty, 
Asper\`o and Schindler
are led to the definition of
the sequence of forcings $\bp{P_\alpha:\alpha\leq\kappa}$, and we also follow their proof-pattern.

\subsection{The forcing $P_\kappa$} \label{subsec:Pkappa}

We fix (or already have fixed) in $V$: 
\begin{notation}
\emph{}

\begin{itemize}
\item
$\kappa>\omega_2$ regular and such that $\Diamond_\kappa$ holds.  
\item
$T$ tree on $(2\times\omega_2^V)^{<\omega}$, such that there are semantic certificates 
$(\mathcal{J},r,f,T)$ for $A,D$ in any generic extension $V[G]$ of $V$ by $\Coll(\omega,\omega_2)$.
\item
$C\subseteq \kappa$ club and
$\bp{Q_\alpha,A_\alpha:\alpha\in \kappa}$ a $\Diamond_{\kappa}$ sequence  
such that:
\begin{enumerate}
\item
For all $\eta\in C$
\begin{equation*}
(Q_\eta,\in,T, R\cap Q_\eta)\prec (H_{\kappa},\in,T, R),
\end{equation*}
where $R$ is a fixed well ordering of $H_{\kappa}$ in type $\kappa$.
\item
For all $\eta$ limit point of $C$
\[
Q_\eta=\bigcup_{\alpha\in C\cap\eta}Q_\alpha.
\]
\item 
For any $Y\subseteq H_{\kappa}$ there are stationarily many $\eta<\kappa$ such that $Y\cap Q_\eta=A_\eta$.
\end{enumerate}
\end{itemize}
\end{notation}

%\begin{remark}
%We can certainly find the requested $\bp{Q_\alpha,A_\alpha:\alpha\in C}$ if $\kappa$ is supercompact:
%Fix a tree 
%$T$ on $(2\times\omega_2)^{<\omega}$ for which there is a semantic certificate 
%$(\mathcal{J},r,f,T)$ for $A,D$ in any generic extension of $V$ by $\Coll(\omega,\kappa)$.
%We let $C$ be the set of regular cardinals $\eta$ below $\kappa$ such that:
%\begin{enumerate}
%\item \label{req:keyreq1}
%%\begin{equation*}
%$(H_\eta,\in,T\cap H_\eta )\prec (H_{\omega_2},\in,T)$,
%%\end{equation*}
%\item \label{req:keyreq2}
%$T\cap H_\eta$ is a tree on $(2\times\omega_2)^{<\omega}$
%such that there is a semantic certificate 
%$(\mathcal{J},r,T\cap H_\eta)$ for $A,D$, $H_\eta$ 
%in any generic extension of $V$ by $\Coll(\omega,\eta)$.
%\end{enumerate}
%$\eta$ satisfies both requirements 
%if there is $j:V_{\eta+1}\to V_{\kappa+1}$ elementary with $T$ in the target of $j$.
%
%We also let 
%$\bp{A_\eta:\eta\in C}$ be the
%$\Diamond_\kappa$-sequence given by a Laver Diamond $F:\kappa\to H_{\omega_2}$ 
%provided by the supecompactness of $\kappa$. Then $\bp{H_{\eta}, A_\eta:\eta\in C},T$ satisfy 
%our requests (letting $Q_\eta=H_\eta$ for all $\eta\in C$). 
%\end{remark}

%We will now need to use a semantic certificate which can produce iterations containing a larger inital segmant of $V$.

%We now fix a tree $T^*$ on $(2\times\omega_2)^{<\omega}$ such that
The forcings $\bp{P_\alpha:\alpha\in C\cup\bp{\kappa}}$ are obtained from $P^*_0$ adding new constant and predicate 
symbols so to have a stronger 
control on the type of structures generated by their $V$-generic filters. The key idea is to incorporate inside the extended 
language the names of the countable elementary substructures of $V$ which can be used to seal a $P_\kappa$-name for a club.
Specifically we will add to $\mathcal{L}^0$ a sort $\dot{Q}$ to denote (a subset of) $H_\kappa$ and
unary predicates $\bp{\dot{X}_i:i<\omega_1^V}$ in this new sort which will 
be used to ensure the following:
\begin{equation}\label{eqn:contlayerPalphahint0}
\text{}
\end{equation}
\begin{quote}
Whenever $H$ is $V$-generic for the forcing  $P_{\kappa}$, for  ``almost'' all\footnote{More precisely for any $S^*\subseteq\kappa$ stationary 
subset of $\kappa$ in $V$ and any $S\subseteq\omega_1^V$ stationary subset of $\omega_1^V$ in $V$, one can find $\eta\in S^*$ and $\alpha<\omega_1^V$ as prescribed.} 
$\eta\in C$ and $\alpha<\omega_1^V$:
\begin{itemize}
\item
$X_\alpha=(\dot{X}_\alpha)_H\prec (Q_\eta,\in,P_\eta,A_\eta)$;
\item 
$H\cap P_\eta$ is 
$X_\alpha$-generic\footnote{It meets inside $X_\alpha$ all dense sets of 
$P_\eta$ definable in $X_\alpha$.} for 
$P_\eta$;
\item
$X_\alpha\cap\omega_1^V=\alpha$.
\end{itemize}
\end{quote}
We will also ensure that the layering of the posets $P_\eta$ is continuous at limit stages, and that the $P_\eta$ are contained in $Q_\eta$
for all $\eta\in C$, so to get that
\begin{equation}\label{eqn:contlayerPalphahint}
(Q_\eta,\in,P_\eta,A_\eta)\prec (H_\kappa,\in,P_\kappa,X),
\end{equation}
holds in $V$ for stationarily many $\eta$ and any $X\subseteq H_\kappa$.

Now assume $S$ is a stationary subset of $\omega_1$ in $V$ and $\dot{C}$ is a $P_\kappa$-name for a club.
Let $H$ be $V$-generic for $P_\kappa$;
By \ref{eqn:contlayerPalphahint0} and \ref{eqn:contlayerPalphahint}, 
we can find $\eta\in C$  and $\alpha<\omega_1^V$ so that $\dot{C}\cap Q_\eta=A_\eta$ is a $P_\eta$-name for a club
and
\[
X_\alpha=(\dot{X}_\alpha)_H\prec (Q_\eta,\in,P_\eta,A_\eta)\prec (H_\kappa,\in,P_\kappa,\dot{C}),
\] 
$\alpha=X_\alpha\cap\omega_1^V\in S$,
$H\cap P_\eta$ is $X_\alpha$-generic for $P_\eta$. 
Then we should expect that $\dot{C}_H$ will have $\alpha\in S$ as a limit point
(note that $(\dot{X}_\alpha)_H$ will be countable in $V[H]$ and will not exist in $V$).
This will grant that $S\cap\dot{C}_H$ is non-empty. By repeating this argument for all stationary subsets of $\omega_1^V$ in $V$ and
all $P_\kappa$-names for a club subset of $\omega_1^V$, we will get that $P_\kappa$ is stationary set preserving.
On the other hand the forcing $P_\kappa$ will also add a semantic certificate for $A,D$ as witnessed by $T$.
This allows to infer, combining the information so far given, that $\MM^{++}$ instantiated for $P_\kappa$ can be used
to assert
that $G_A\cap D$ is non-empty.

We are now going to describe close approximations to the generic objects added by the various $P_\eta$; 
then we will define the 
signature  which characterizes these generic objects as models of a standard theory of $\bigwedge\bigvee$-type for
$\mathcal{L}^1_\eta$. 
The forcings $P_\eta$ will be of the form $P_{\mathcal{W}_\eta}$ for suitably chosen families of generic assignments
$\mathcal{W}_\eta$. 
%The stratification will be such that the forcings $P_\lambda$ will be able to ensure 
%the preservation of stationary subsets of 
%$\omega_1$ belonging to $Q_\lambda$, if $\kappa$ is $\aleph_2$, in our set up 
%only $P_\kappa$ has a chance to be truly stationary set preserving.

We assume and recall the following facts, notations and conventions:

%\emph{}
%\begin{itemize}
%\item The club $C$ on $\kappa$ we have fixed
%is also such that for $T\cap Q_\eta=T_\eta$ is a witnessing tree for semantic certificates 
%$(\mathcal{J},r,f,T)$ for $A,D,Q_\eta$ existing in some generic extension $V[G]$ of $V$ for 
%any $\eta\in C$.   
%
%
%\item 
%In the main proof we will consider also weak semantic certificates 
%$(\mathcal{J},r,f,T)$ for $A,Q_\lambda$ existing in $\UB^V$-canonical models (see Def. \ref{def:CanMod}).
\begin{notation}
%From now on for any $\lambda\in C$, we let $T_\lambda=Q_\lambda\cap T$.
Given a weak semantic certificate $(\mathcal{J},r,f,T)$ for $A,D$, we let
$X_{(\mathcal{J},r,f,T)}$ be the set:
\[ 
(\bigcup_{i\leq\omega_1^V}N_\alpha)\cup
\bp{N_\alpha:\alpha\leq\omega_1^V}\cup\bp{j_{\alpha\beta}:\alpha\leq\beta\leq\omega_1^V}
\cup\bp{G_\alpha:\alpha<\omega_1}\cup\bp{r,f},
\]
$\mathcal{X}_{(\mathcal{J},r,f,T)}$ be the structure:
\[
\ap{N_\alpha:\alpha\leq\omega_1^V,\,\bp{r,f};\,\in\restriction N_\alpha,\Sat_{N_\alpha}:\alpha\leq\omega_1^V,\, G_\alpha:\alpha<\omega_1^V,\, 
j_{\alpha\beta}:\alpha\leq\beta\leq\omega_1^V,\, \sqsubseteq^2\restriction T,\, \text{br}_T,\, \Cod^*}.
\] 
%\end{itemize}
\end{notation}

Recall that we have fixed in $V$ the diamond sequence $\bp{Q_\lambda:\lambda\in C}$ and the witnessing tree $T$.

\begin{definition} \label{def:semcert1}
For $\lambda\in C\cup\bp{\kappa}$ 
a tuple
\[
(\mathcal{J},r,f,F,K)
\]
existing in some transitive $\ZFC$-model $W\supseteq Q_\lambda$ is a \emph{weak $\lambda$-precertificate} 
(relative to $V$) if
\begin{itemize}
\item $(\mathcal{J},r,f,T)$ is a weak semantic certificate for $A$ (relative to $V$) existing in $W$ 
as witnessed by 
$T$ as in Def. \ref{def:weaksemcert0}.
\item $K\subseteq \omega_1^V$.
\item $F$ a function with domain $\omega_1^V$ defined by
\begin{align*}
F:\, &\omega_1^V\to C\cap\lambda\times\pow{Q_\lambda}^W\times Q_\lambda\\
&\alpha\mapsto\ap{\lambda_\alpha,X_\alpha,Z_\alpha}
\end{align*}
with $\ap{\lambda_\alpha,X_\alpha,Z_\alpha}=\ap{F_0(\alpha),F_1(\alpha),F_2(\alpha)}$ and:
\begin{itemize}
\item
$F\restriction (\omega_1^V\setminus K)$ constantly with value $\ap{0,1,0}$. %has range contained in $C\cap\lambda\times\pow{Q_\lambda}^2$
\item
$F_0\restriction K$ and $F_1\restriction K$ being
$\subseteq$-increasing,
\item
for all $\alpha\in K$, $F_2(\alpha)=Z_\alpha\subseteq Q_{F_0(\alpha)}$ is in $Q_\lambda$ and
$F_1(\alpha)=X_\alpha$ is the domain of an elementary substructure of 
\[
(Q_{F_0(\alpha)},\in,F_2(\alpha), A_{F_0(\alpha)})
\]
with\footnote{$X_\alpha$ will not be in $V$, and will be countable in $V[G]$. 
In the right set up (i.e. when the posets $\bp{P_\beta:\beta\in C\cap\lambda_\alpha}$ have been defined),
$Z_\alpha$ will be assigned to the poset $P_{\lambda_\alpha}$.  The very definition of $P_{\lambda_\alpha}$ 
is done by a recursion based on the notion of 
$\lambda$-precertificate and using the sequence $\bp{P_\beta:\beta\in C\cap\lambda_\alpha}$, in particular we cannot for the moment assign 
$Z_\alpha$ to its canonical intended meaning. 
Hence we let $Z_\alpha$ range over all possible elements of $Q_\lambda$.} 
$\alpha=X_\alpha\cap\omega_1^V$.
\end{itemize}

\item $(\mathcal{J},r,f,F,K)$ is a \emph{$\lambda$-precertificate for $A,Q_\lambda$} (relative to $V$) if it exists in a generic extension of $V$.
\end{itemize}
We denote by $X^\lambda_{(\mathcal{J},r,f,F,K)}$ the set
\[
X_{(\mathcal{J},r,f,T)}\cup\bp{K,F_0,Q_\lambda}\cup\bp{F_1(\alpha),F_2(\alpha):
\alpha\in K}.
\]
and by $\mathcal{X}^\lambda_{(\mathcal{J},r,f,F,K)}$ the multi-sorted structure obtained extending 
$\mathcal{X}_{(\mathcal{J},r,f,T)}$ with the sort $Q_\lambda$, the constants $\bp{Z_\alpha:\alpha<\omega_1^V}$ of sort $Q_\lambda$, the binary predicates $F_0,\in\restriction{Q_\lambda}$ both of type $(Q_\lambda,Q_\lambda)$, the unary predicates
$K$, $\bp{X_\alpha:\alpha<\omega_1^V}$ each of type $Q_\lambda$.
\end{definition}

$X^\lambda_{(\mathcal{J},r,f,F,K)}$ is the smallest set existing in $W$ and containing 
all the information needed to reconstruct the weak
semantic certificate $(\mathcal{J},r,f,T_\lambda)$ for $A,D$, the function $F$, the selected set $K$.
 $\mathcal{X}^\lambda_{(\mathcal{J},r,f,F,K)}$ is the multi-sorted structure which will encode this information using multi-sorted logic.

\begin{remark}

Note 
that letting 
\[
\mathcal{J}=\bp{j_{\alpha\beta}:N_\alpha\to N_\beta:\alpha\leq\beta\leq\omega_1^V},
\]
and $(G_i:i<\omega_1^V)$ list the $N_i$-generic ultrafilters associated to $\mathcal{J}$,
\begin{itemize}
\item
The top elements in the $\in$-graph of $X^\lambda_{(\mathcal{J},r,f,F,K)}$ are exactly the elements of the following set:
\[
\bp{r,f,K, F_0}\cup\bp{X_\alpha:\alpha<\omega_1^V}\cup\bp{N_i:  i\leq \omega_1^V}\cup\mathcal{J}\cup\bp{G_\alpha:\alpha<\omega_1}.
\]
%(again with the exception possibly of $T_\lambda,Q_\lambda$ if the latter were in $N_{\omega_1}$).
\item $X^\lambda_{(\mathcal{J},r,f,F,K)}$ is not transitive for the same reason for which transitivity fails for 
$X_{(\mathcal{J},r,f,T_\lambda)}$, but the new sorts, constants,
and predicates denote elements or subsets of $Q_\lambda$, hence they will not have an $\in$-ill-founded extension in models of
$\Sigma^0$ extended with axioms granting that the sort $Q_\lambda$ gets interpreted by an isomorphic copy itself.
\end{itemize}

Note also that the extra requirement imposed on a weak $\lambda$-precertificate are expressible by
$\Delta_0$-properties in parameters 
$C\cap\lambda,\lambda,\omega_1^V,\Diamond_\kappa\cap Q_\lambda,F_0,K,\bp{X_\alpha,Z_\alpha:\alpha<\omega_1^V}$, i.e.:
\begin{enumerate}[(i)]
\item $K\subseteq \omega_1^V$,
\item 
\begin{align*}
(F_0 \text{ is a function})\wedge(\dom(F_0)=\omega_1^V)\wedge(\ran{F}_0\subseteq C\cap\lambda\cup\bp{0})\wedge\\
\wedge\forall\alpha\in\beta\in\omega_1^V\, \qp{(\alpha\in K\wedge\beta\in K)\rightarrow (X_\alpha\subseteq X_\beta\wedge F_0(\alpha)\in F_0(\beta))}
%\wedge\\
%\wedge\forall\alpha\in \omega_1^V\setminus K\, (F_0(\alpha)=0),
\end{align*}
\item
$\forall\alpha\in\omega_1^V\setminus K \,\qp{(X_\alpha=Z_\alpha=1)\wedge (F_0(\alpha)=0)}$,
\item
\[
\forall\alpha\in K\,\qp{(X_\alpha\subseteq Q_{F_0(\alpha)})\wedge (X_\alpha\cap\omega_1^V=\alpha)\wedge (Z_\alpha\subseteq Q_{F_0(\alpha)})\wedge 
(Z_\alpha\in Q_\lambda)},
\]
%\item $\forall\alpha,\beta\in K\,\qp{(\alpha\in \beta)\rightarrow (X_\alpha\subseteq X_\beta\wedge F_0(\alpha)\in F_0(\beta))}$,
\item
\begin{align*}
&\forall\,\alpha\in K\,\forall n,m\in\omega\,\forall\vec{x}\in X_\alpha^{<\omega}&\\ 
&\qp{\gp{(X_\alpha,\in,\,Z_\alpha\cap X_\alpha,\,A_{F_0(\alpha)\cap X_\alpha})\models \phi_n(\vec{x})}\leftrightarrow 
\gp{(Q_{F_0(\alpha)},\in,Z_\alpha,A_{F_0(\alpha)})\models \phi_n(\vec{x})}}&.
\end{align*}
\end{enumerate}
Combining these observations with those in Remark \ref{rmk:keyremsemcert}, we get that
all the conditions needed to check whether $(\mathcal{J},r,f,F,K)$ is a weak $\lambda$-precertificate 
are expressible by provably 
$\Delta_1$-properties in $(H_{\omega_1}^W,\tau_\ST^W)$ for every $W$  transitive model of $\ZFC$ in which all the relevant parameters
are countable.
\end{remark}
%Similarly we can check 
%The difference between a $\lambda$-precertificate and an $\eta$-precertificate for $\lambda<\eta\in C\cup\bp{\kappa}$ is
%just in the range of the function $F$. In particular an $\eta$-precertificate is always a $\lambda$-precertificate.
%Conversely a if $(\mathcal{J},r,f,F,K)$ is an $\eta$-precertificate $F\restriction \gamma$ with 
%\[
%\gamma=\sup\bp{\alpha\in K: F(\alpha)\in\lambda\times\pow{Q_\lambda}^2}
%\]
%is an $\eta$-precertificate.

%\end{document}

%Note that $Q_\lambda$ is a subset of $X_{(\mathcal{J},r,f,F,K)}$.

\begin{definition} \label{def:defL1}
Let $\mathcal{L}^1$ be the extension of $\mathcal{L}^0$ (given in Def. \ref{def:defL0})
adding the following sort, constants and predicates:
\begin{description}

\item[Sort]
\emph{}

\begin{enumerate}[(a)]
\setcounter{enumi}{13}
\item \label{def:defL0-m}
A sort $\dot{Q}$ to denote the element of $Q_\lambda$ for some $\lambda\in C\cup\bp{\kappa}$.
\end{enumerate}

\item[Constants]
\emph{}

\begin{enumerate}[(a)]
\setcounter{enumi}{14}
\item \label{def:defL0-m1}
symbols $\check{\check{x}}$ of sort $\dot{Q}$ for each $x\in H_\kappa$ to be interpreted by $x$.
%as the elements of $Q_\lambda$ (which is going to be enumerated in type $\omega_1^V$ in the relevant outer models $W$).

\item \label{def:defL0-m2}
symbols $\dot{e}_{n,\alpha}$ of sort $\dot{Q}$ for each $n<\omega,\alpha<\omega_1^V$ to be interpreted 
as the elements of the various $X_\alpha$.
\item
symbols $\dot{Z}_\alpha$ of sort $\dot{Q}$ for each $\alpha<\omega_1^V$ to be interpreted 
as the sets $F_2(\alpha)$.

\end{enumerate}

\item[Predicates]
\emph{}

\begin{enumerate}[(a)]
\setcounter{enumi}{17}
\item \label{def:defL0-n}
a binary predicate symbol $\in_{\dot{Q}}$ of type $(\dot{Q},\dot{Q})$ to be interpreted by the $\in$-relation restricted to the interpretation of $Q$,

\item \label{def:defL0-n1}
a unary predicate symbol $\dot{K}$ of type $\dot{Q}$ to be interpreted by the selected subset of $\omega_1^V$,

\item \label{def:defL0-o}
a binary predicate symbol $\dot{F}_0$ of type $(\dot{Q},\dot{Q})$
to be interpreted by the selected map $\alpha\mapsto \lambda_\alpha$ from 
$\omega_1^V$ to $C$,

\item \label{def:defL0-p}
unary predicate symbols $\dot{X}_\alpha$  of type $\dot{Q}$ for each $\alpha<\omega_1$ to be interpreted by the sets 
$F_1(\alpha)$,

%\item \label{def:defL0-p*}
%a unary predicate symbol $\dot{P}$ to be interpreted by a certain specified subset $P$ of $Q_\lambda$ existing in $V$
%(i.e. $P$ will be the partial order $P_\lambda$ to be defined in Def. \ref{def:keydefPkappa} to follow),

\item\label{def:defL0-q}
Let $\bp{\in,U_0,U_1}$ be the signature with $U_0,U_1$ unary predicate symbols; we introduce
the satisfaction predicates
$\Sat_{i,\alpha}(\gp{\psi(x_1,\dots,x_m)})$ of type $(\dot{Q},\dot{Q})$  with $i=0,1$ and $\alpha<\omega_1^V$
to be represented by the pairs $\ap{\gp{\psi},(a_1,\dots,a_m)}$ with $\psi(x_1,\dots,x_m)$ an 
$\bp{\in,U_0,U_1}$-formula in displayed free variables and $(a_1,\dots,a_m)\in H_\kappa^m$ such that 
for all $\alpha<\omega_1^V$:

\begin{enumerate}[(i)]
\item\label{def:defL0-qi}
$\Sat_{0,\alpha}(\gp{\psi(a_1,\dots,a_m)})$ holds if and only if $\alpha\in K$, $a_1,\dots,a_m\in X_\alpha$ and 
\[
(X_\alpha,\in,Z_\alpha\cap X_\alpha,A_{F_0(\alpha)}\cap X_\alpha)\models\psi(x_1,\dots,x_m)[x_1/a_1,\dots,x_m/a_m].
\]
\item \label{def:defL0-qii}
$\Sat_{1,\alpha}(\gp{\psi(a_1,\dots,a_m)})$ holds if and only if $\alpha\in K$, $F_0(\alpha)=\eta$, $a_1,\dots,a_m\in Q_\eta$,
and 
\[
(Q_\eta,\in,Z_\alpha,A_\eta)\models\psi(x_1,\dots,x_m)[x_1/a_1,\dots,x_m/a_m].
\]
\end{enumerate}

\end{enumerate}
\end{description}

%%For any $\lambda\in C\cup\bp{\kappa}$, let
%$\mathcal{L}^1$ is in $V$ the first order signature extending $\mathcal{L}^1$  
%with
%%\item-(*) \label{def:defPkappa0}
%a constant symbol $\check{x}$ denoting $x$ for every $x\in H_{\omega_2}\cup\bp{H_{\omega_2},T}$.

For $\lambda\in C\cup\bp{\kappa}$ $\mathcal{L}^1_\lambda$ restricts $\mathcal{L}^1$ omitting the constant symbols $\check{\check{x}}$ for any $x\not\in Q_\lambda$, so
that $\mathcal{L}^1_\kappa$ is $\mathcal{L}^1$.

$\mathcal{R}^1_\lambda$ denotes the relations symbols of $\mathcal{L}^1_\lambda$ and $\mathcal{K}^1_\lambda$ its constant symbols.

$\SL{\lambda}$ the family of atomic sentences of  $\mathcal{L}^1_\lambda$.
\end{definition}

Note that (with a slight abuse of notation, see Warning \ref{warn:signXJrfT}) 
$\mathcal{X}^\lambda_{(\mathcal{J},r,f,F,K)}$ can be considered an $\mathcal{R}^1_\lambda$-structure whenever 
$(\mathcal{J},r,f,F,K)$ is a $\lambda$-precertificate.

\begin{notation}
From now on we assume 
$\mathcal{L}^1_\kappa$ being a definable class in $(H_{\kappa},\in,R)$. We may (and will) also assume that
$\mathcal{L}^1_\lambda\subseteq Q_\lambda$ is a definable class in $(Q_\lambda,\in,R\cap Q_\lambda)$ for all $\lambda\in C$.
\end{notation}

\begin{definition}\label{def:keydefnuPkappa}
\emph{}

\begin{itemize}
\item 
Given a $\lambda$-precertificate $(\mathcal{J},r,f,F,K)$ existing in $V[G]$,
 a function 
\[
\nu:\mathcal{L}^1_\lambda\to X^\lambda_{(\mathcal{J},r,f,F,K)}
\] 
 is a $\lambda$-admissible interpretation for $(\mathcal{J},r,f,F,K)$ if: 
\begin{itemize}
%\item
%$\nu(\dot{T})=T$,
\item
$\nu\restriction \mathcal{L}^0$ respects the interpretation of the symbols of 
$\mathcal{L}^0$ 
according to the prescription set forth in Def. \ref{def:keydefnu};
%\item
%$\nu(\Sat_{\dot{Q}})$ is the set of
%pairs $\ap{\gp{\psi},(a_1,\dots,a_m)}$ with $\psi(x_1,\dots,x_m)$ an $\in$-formula in displayed free variables and 
%$(a_1,\dots,a_m)\in Q_\lambda$ such that
%\[
%(Q_\lambda,\in)\models\psi(x_1,\dots,x_m)[x_1/a_1,\dots,x_m/a_m].
%\]

%\item
%$\nu(\dot{T})=T$;
%(with the obvious replacement in all required places of $Q_\lambda$ in the place of $H_{\omega_2}$ and 
%$T\cap Q_\lambda$ in the place of $T$); 
%\item 
%$\nu$ 
%%respects the interpretation of the symbols of $\mathcal{L}^0$ 
%%according to the prescription set forth in
%%Def. \ref{def:defL0}, Def. \ref{def:keydefnu},
%%and
%interprets all constant symbols $\check{x}$ by $x$ for all 
%$x\in H_{\omega_2}\cup\bp{T}$;%\cup\bp{Q_\lambda}$;
\item
$\nu$ respects the following assignments to the symbols of $\mathcal{L}^1\setminus \mathcal{L}^0$:
\begin{enumerate}[(a)]
\setcounter{enumi}{13}
\item
$\nu(\dot{Q})=Q_\lambda$,
\item
$\nu(\check{\check{x}})=x$ for all $x\in Q_\lambda$,
\item
$\nu[\bp{\dot{e}_{n,\alpha}:n\in\omega}]=F_1(\alpha)=X_\alpha$ for all $\alpha\in \omega_1^V$,
\item
$\nu(\dot{K})=K$, 
\item
$\nu(\dot{F}_0)=F_0$,
\item
$\nu(\dot{X}_\alpha)=F_1(\alpha)$ and
$\nu(\dot{Z}_\alpha)=F_2(\alpha)$ for all $\alpha\in \omega_1^V$.
%\item
%$\nu(\dot{P})$ is a subset of $Q_\lambda$ belonging to $V$. 
\item
\emph{}

\begin{enumerate}[(i)]
\item
$\nu(\Sat_{0,\alpha})$ is defined according to the prescriptions given in \ref{def:defL0-q}\ref{def:defL0-qi},
\item
$\nu(\Sat_{1,\alpha})$ is defined according to the prescriptions given in \ref{def:defL0-q}\ref{def:defL0-qii}.
\end{enumerate}
%\item
%$\nu(\Sat_{\dot{Q}})$ according to the prescriptions given in \ref{def:defL0-q}\ref{def:defL0-qiii}.
\end{enumerate}
\end{itemize}

%\item $\nu(\dot{P})=Z\in\pow{Q_\lambda}^V$.
%\item
%$\nu(\dot{H})=\bp{\psi\in\SL{\lambda}: }

\item Given a $\lambda$-admissible $\nu:\mathcal{L}^1\to X^\lambda_{(\mathcal{J},r,f,F,K)}$, 
$\Sigma_\nu$ is the set of $\mathcal{L}^1_\lambda$-sentences for $\mathfrak{L}_{\infty,\omega}$
belonging to $V$ and
realized in the multi-sorted $\mathcal{L}^1_\lambda$-structure $\mathcal{X}^\lambda_{(\mathcal{J},r,f,F,K)}$.
\item 
We denote by $\mathcal{A}_\lambda$ the family of $\lambda$-admissible interpretations.
%We say that  $\Sigma_\nu$ is the \emph{syntactic $\lambda$-certificate} associated to $\nu$.
\end{itemize}
\end{definition}

The forcings $P_\lambda$ will be defined as refinements of the forcing $P_{\mathcal{A}_\lambda}$, in which we enforce the 
structures $X_\alpha$ occurring in a $\lambda$-precertificate to admit
generic filters for larger and larger fragments of $P_\lambda$ by letting $K$ and $F_0$ grow in size.

\begin{definition}\label{def:keydefPkappa}
We define by transfinite recursion on $\alpha\in C\cup\bp{\kappa}$ the forcings 
$P_\alpha$ and the notions of (weak) $P_\alpha$-certificate $(\mathcal{J},r,f,F,K)$ and
(weak) $P_\alpha$-witness $\nu$ for the certificate $(\mathcal{J},r,f,F,K)$
as follows:
\begin{enumerate}[(A)]
%\item \label{def:keydefP0-1}
%Each $P_\alpha\subseteq [\SL{1}]^{<\omega}$.
\item \label{def:keydefP0-2}
$P_{\min(C)}=P_{\mathcal{A}_{\min(C)}}$ according to\footnote{By the requirements set forth for $K$ in Def. \ref{def:semcert1}, 
the $\min(C)$-precertificate witnessing $p\in P_{\min(C)}$ always assigns $\emptyset$ to $K$.}
Def. \ref{def:conspropforcing} and \ref{def:keydefnuPkappa}. 
\item \label{def:keydefP0-3a}
A weak $P_{\min(C)}$-certificate is a weak $\min(C)$-precertificate, and a
weak $P_{\min(C)}$-witness for the $P_{\min(C)}$-certificate is a $\min(C)$-admissible interpretation according to
Def. \ref{def:keydefnuPkappa}.

%(with $H_{\omega_2}$ replaced by $Q_{\lambda_0}$ and $T$ replaced by $T\cap Q_{\lambda_0}$).
\item \label{def:keydefP0-3}
For $\eta\in C\cup\bp{\kappa}$ a weak $\eta$-precertificate $(\mathcal{J},r,f,F,K)$ existing in $W$ is a weak $P_\eta$-certificate, and a weak 
$\eta$-admissible $\nu:\mathcal{L}^1_\eta\to X^\eta_{(\mathcal{J},r,f,F,K)}$ existing in $W$
is a weak $P_\eta$-witness for $(\mathcal{J},r,f,F,K)$ if the following holds for all $\alpha\in K$:
\begin{enumerate}[(i)]
\item \label{def:keydefP0-3-a}
\begin{align*}
F:&\omega_1^V\to\eta\times\pow{Q_\eta}^W\times Q_\eta\\
&\alpha\mapsto\ap{\lambda_\alpha,X_\alpha,Z_\alpha}=\ap{F_0(\alpha),F_1(\alpha),F_2(\alpha)}
\end{align*}
with $P_{\lambda_\alpha}=Z_\alpha$, 

\item \label{def:keydefP0-3-c}
\begin{align*}
& 
E\cap X_\alpha\cap [\Sigma_\nu]^{<\omega}\neq\emptyset & \nonumber
\end{align*}
for any
dense 
subset $E$ of $P_{\lambda_\alpha}$ definable
in 
\[
(Q_{\lambda_\alpha},\in,P_{\lambda_\alpha},A_{\lambda_\alpha})
\]
using parameters in $X_\alpha$.
\end{enumerate}
\item
A $P_\eta$-certificate (respectively a $P_\eta$-witness) is a weak $P_\eta$-certificate 
(respectively a weak $P_\eta$-witness) existing in a generic extension of $V$.

\item \label{def:keydefP0-4}
For $\eta\in C\cup\bp{\kappa}$, $\mathcal{W}_\eta$ is the family of $P_\eta$-witnesses and 
$P_{\eta}$ is $P_{\mathcal{W}_\eta}$ according to Def. \ref{def:conspropforcing}.
\end{enumerate}

\end{definition}

We can now state in precise terms the main technical result of this paper:

\begin{theorem}\label{thm:mainthmPkappa}
Assume $H$ is $V$-generic for $P_\kappa$.
Let 
\[
\mathcal{J}_H=\bp{j_{\alpha\beta}:N^H_\alpha\to N^H_\beta:\,\alpha\leq\beta\leq\omega_1^V}
\]
be the iteration induced by $H$ and accordingly define $r_H,f_H$.
Then $(\mathcal{J}_H,r_H,f_H,T)$ is a semantic certificate for $A,D$ and
for all $S\in \pow{\omega_1}^{N_{\omega_1}}$
\[
(N_{\omega_1},\in)\models \text{S}\text{ is stationary } \qquad \Longleftrightarrow \qquad (V[H],\in)\models \text{S}\text{ is stationary. }
\] 

In particular (since $H_{\omega_2}^V\subseteq N_{\omega_1}$ and $\NS^V=\NS^{N_{\omega_1}}\cap H_{\omega_2}^V$) 
$P_\kappa$ is stationary set preserving.
\end{theorem}

The key to the proof of Thm. \ref{thm:mainthmPkappa} will be a combination of straightforward generalizations
to the various $P_\lambda$ of the results on forcings induced by consistency properties which we already 
obtained for $P^*_0$, and of a
sophisticated variation of Lemma \ref{lem:keylemASPSCH(*)0} which will be used to run a master condition argument in 
the style of 
those one uses to establish the properness of a given poset.

Let us start outlining basic properties of the various forcings $P_\lambda$.

\begin{remark}\label{rmk:basicPkappa}
\emph{}

\begin{enumerate}[(i)]
\item
The key condition for the various $P_\eta$ is \ref{def:keydefP0-3}.

\item
The definition of the sequence $\bp{P_\eta:\eta\in C\cup\bp{\kappa}}$ is first order because ``existence in some generic extension'' is 
first order definable in $(V,\in)$ (while ``existence in a transitive outer model of $\ZFC$'' is apparently not).
On the other hand once we know what $P_\eta$ is in $V$, it makes perfect sense to check
whether in some outer transitive model $W$ a certain tuple is a weak $P_\eta$-certificate or $P_\eta$-witness.

\item\label{rmk:itemrelPkappa}
$\bp{P_\eta:\eta\in C\cup\bp{\kappa}}$ can be defined in any sufficiently correct $\in$-structure containing 
\[
H_\kappa,R,T,C,\Diamond_\kappa.
\]
More precisely this recursive definition can be carried in the structure
\[
(H_{\kappa},\in,R,T,C,\Diamond_\kappa).
\]

It will be convenient to relativize this definition to $\bar{M}$ where $\bar{k}:V\to \bar{M}$ is a generic 
embedding existing in a generic extension of $V$. 
In particular it makes sense to check in outer models of $\bar{M}$ whether there are weak 
$\bar{k}(P_\kappa)$-witnesses $\nu$, and
weak $\bar{k}(P_\kappa)$-certificates $(\mathcal{J},r,f,F,K)$ with  $(\mathcal{J},r,f,\bar{k}(T))$ a weak semantic certificate for 
$\bar{k}(A)$. This will occur in the statement of Claim \ref{clm:fundclmrproof} to follow.

\item
Each $P_\eta$ is a subset of $[\SL{\eta}]^{<\omega}\subseteq Q_\eta$. 

\item 
\emph{$\nu:\mathcal{L}^1_\lambda\to X_{(\mathcal{J},r,f,F,K)}$ is a weak $P_\lambda$-witness} is expressible by a provably $\Delta_1$-property 
$\psi(\nu,A, \bp{Q_\eta:\eta\in C\cap\lambda+1},T, \bp{Q_\eta,P_\eta:\eta\in C\cap\lambda},\nu)$ in $\tau_\ST$: 
it adds to the request that the target of $\nu$ is a weak $\lambda$-precertificate, and that $\nu$ is a $\lambda$-assignment, 
formulae for Condition \ref{def:keydefP0-3} given by
\[
\forall\alpha\in K\, (F_2(\alpha)=P_{F_0(\alpha)}),
\]
and
\begin{align*}
&\forall \alpha\in K\, \forall\eta\in (C\cap\lambda)\,\forall\vec{z}\in X_\alpha^{<\omega}&\\
&\qp{(F_0(\alpha)=\eta\wedge 
\qp{(Q_\eta,\in,P_\eta,A_\eta)\models\forall x\in P_\eta\,\exists y\in P_\eta \, (x\subseteq y\wedge \psi(y,\vec{z}))}}&\\
&\rightarrow&\\ 
&\exists y\in ([\Sigma_\nu]^{<\omega}\cap X_\alpha)\, \qp{(Q_\eta,\in,P_\eta,A_\eta)\models\psi(y,\vec{z})}&.
\end{align*}

\item
For all $\lambda<\eta \in C\cup\bp{\kappa}$,
\begin{enumerate}[(a)]
\item
%If
%\[
%(Q_\lambda,\in,T,C\cap\lambda,\bp{A_\alpha:\alpha<\lambda})\prec (Q_\eta,\in,T,C\cap\eta,\bp{A_\alpha:\alpha<\eta}),
%\]
$P_\lambda=P_\eta\cap Q_\lambda$ holds:
%nonetheless there could be 
%$P_\lambda$-witnesses that
%$p\in P_\lambda$ which are not $P_\eta$-witnesses (see the next items). 
\begin{itemize}

\item
If $\nu$ is a weak $P_\lambda$-witness for the weak $P_\lambda$-certificate $(\mathcal{J},r,f,F,K)$, 
then it is easily checked that $\nu^*$ is a weak $P_\eta$-witness 
(where 
$\nu^*(\dot{Q})=Q_\eta$, $\nu^*(\in_{\dot{Q}})=\in\restriction Q_\eta$, and the other assignments can be left unchanged).

\item \label{enum:propPkappa}
If
$\nu$ is a weak $P_\eta$-witness and $(\mathcal{J},r,f,F,K)$ its associated weak $P_\eta$-certificate, 
letting $\gamma=\sup\bp{\alpha\in K: F_0(\alpha)<\lambda}$, $K^*=K\cap\gamma$, the map $F^*$ such that 
$F^*\restriction K^*=F\restriction K^*$
and defined by $\alpha\mapsto\ap{0,1,1}$ for all other $\alpha<\omega_1^V$ is such that 
$(\mathcal{J},r,f,F^*,K^*)$ is a weak $P_\lambda$-certificate; moreover letting $\nu^*(c)=\nu(c)$ for all $c$ not among
$\bp{\dot{F}_0,\dot{K},\dot{X}_\alpha,\dot{Z}_\alpha:\alpha<\omega_1^V}$, 
$\nu^*(\dot{Q})=Q_\lambda$
$\nu^*(\dot{K})=K\cap \gamma$,
$\nu^*(\dot{F}_0)=F^*_0$, $\nu^*(\dot{X}_\alpha)=0$ and $\nu^*(\dot{Z}_\alpha)=0$ if $\alpha\not \in K^*$, 
$\nu^*(\dot{X}_\alpha)=\nu(\dot{X}_\alpha)$ and $\nu^*(\dot{Z}_\alpha)=\nu(\dot{Z}_\alpha)$ for 
$\alpha\in K\cap\gamma$, we get that $\nu^*$ is a  
weak $P_\lambda$-witness.
\end{itemize}
\end{enumerate}
% This will give us an additional layer of complications; for example we wil show that a 
%$P_\lambda$-generic filter generates
%a $P_\lambda$-witness and a $P_\lambda$-certificate; but in principle the generic iteration added by 
%$P_\lambda$ is such that $Q_\lambda\subseteq N_{\omega_1}$ but not such that 
%is such that $H_\kappa\subseteq N_{\omega_1}$. In particular the generic filter for 
%$P_\lambda$ adds a $P_\lambda$-witness that a certain condition 
%$p$ is in $P_\lambda$ which may fail to be a 
%$P_\kappa$-witness for $p\in P_\kappa$. 
%This is exactly one of the main issues we will encounter: we will produce by forcing with $P_\lambda$ 
%a generic $P_\lambda$-witness $\nu$,
%we will then use a sophisticated variation of 
%Lemma \ref{lem:keylemASPSCH(*)0} to get that $\bar{k}[\nu]$ can be extended to
%a $\bar{k}(P_\kappa)$-witness
%where $\bar{k}:V\to\bar{M}$
%is a generic ultrapower. 

\item
We should not expect that $P_\lambda$ is a complete suborder of $P_\kappa$ for $\lambda\in C$.
Condition \ref{def:keydefP0-3} is inserted exactly  to infer that in some cases the inclusion of
$P_\lambda$ into $P_\kappa$ is sufficiently well behaved; for example we will see that 
whenever 
$r_{\alpha,\lambda}=\bp{(\check{\alpha}\in\dot{K}),(\dot{F}_0(\check{\alpha})=\check{\lambda})}$ is in $P_\kappa$,
the map $r\mapsto r\cup r_{\alpha,\lambda}$ is well defined on a sufficiently large fragment of\footnote{But not on all of $P_\lambda$: for example if $\gamma>\alpha$ and
if $(\check{\gamma}\in \dot{K})$ is in $r$ then  $r\cup r_{\alpha,\lambda}$ is not in $P_\kappa$.}
$P_\lambda$
and can be used to seal a  $P_\kappa$-name $\dot{C}$ for a club at stage $\lambda$
if $A_\lambda=\dot{C}\cap Q_\lambda$.
\end{enumerate}
\end{remark}

All properties for being a $P_\lambda$-certificate (with the exception
of Condition \ref{def:keydefP0-3}\ref{def:keydefP0-3-c} of Def. \ref{def:keydefPkappa}) are expressible by
$\mathcal{L}^1_\lambda$-sentences of $\bigwedge\bigvee$-type, henceforth 
will be true in the term model given by a generic filter for $P_\lambda$ (which is a 
maximally $\mathcal{W}_\lambda$-consistent set of atomic $\mathcal{L}^1_\lambda$-sentences).
In order to assert that also \ref{def:keydefP0-3}\ref{def:keydefP0-3-c} is satisfied  in the generic term model
we use a density argument.

On the other hand (as for $P^*_0$), the $\mathcal{L}^{1}_\lambda$-structures whose domain is a 
$P_\lambda$-certificate can still be
characterized by a $\mathfrak{L}_{\infty,\omega}$-theory; to do so however we must expand the 
language adding a predicate symbol for the filter
$[\Sigma_\nu\cap\SL{\lambda}]^{<\omega}$ given by the $P_\lambda$-witness in order to formalize 
Condition \ref{def:keydefP0-3}\ref{def:keydefP0-3-c} of Def. \ref{def:keydefPkappa} in the extended signature.

%%however we need to add an extra-predicate symbol for the
%%filter $[\Sigma_\nu]^{<\omega}\cap\SL{\lambda}$ in order to express Condition \ref{def:keydefP0-3}\ref{def:keydefP0-3-c}.
%%
\begin{definition}
$\mathcal{L}^{1*}$ is obtained from $\mathcal{L}^{1}$ adding:

\begin{enumerate}[(a)]
%\emph{} 

\setcounter{enumi}{22}
\item  \label{def:defL0-r}
A unary predicate symbol $\dot{H}$ of type $\dot{Q}$ to be interpreted by $[\Sigma_\nu\cap\SL{\lambda}]^{<\omega}$
by any $\lambda$-admissible assignment $\nu$.
\end{enumerate}

$\mathcal{L}^{1*}_\lambda=\mathcal{L}^1_\lambda\cup\bp{\dot{H}}$
\end{definition} 

\begin{notation}
From now on we identify a $\lambda$-admissible $\nu$ by its restriction to the fragment given by
\[
\mathcal{L}^*=\bp{\dot{X}_{\alpha},\dot{Z}_\alpha:n<\omega,\alpha<\omega_1^V}\cup\bp{\dot{K},\dot{F}_0,\dot{r},\dot{f}}.
\]
We also consider that the $\lambda$-admissible $\nu$ are defined on the whole $\mathcal{L}^{1*}_\lambda$.
\end{notation}

\begin{remark}\label{rmk:domainlambdaadmnu}
it is clear that any $\nu_0$ defined on $\mathcal{L}^*$ admits at most one extension to a $\lambda$-admissible 
$\nu:\mathcal{L}^{1*}_\lambda\to X_{(\mathcal{J},r,f,F,K)}$: If $\nu(\dot{r})=r,\nu(\dot{f})=f$, $(r,f)\in T$, and
$\Cod(r)=(N_0,a_0)$, there is at most one iteration $\mathcal{J}$ of
$N_0$ such that $j_{0\omega_1^V}(a_0)=A$ such that $(\mathcal{J},r,f,T)$ is weak semantic certificate for $A$; now
there is at most one $\lambda$-precertificate $(\mathcal{J},r,f,F,K)$ such that $F_0=\nu(\dot{F}_0)$, $K=\nu(\dot{K})$,
$F_1(\alpha)=X_\alpha=\nu(\dot{X}_{\alpha})$, $F_2(\alpha)=Z_\alpha$ for all $\alpha<\omega_1^V$.
\end{remark}

\smallskip

We now introduce a $\mathfrak{L}_{\infty,\omega}$-theory whose models are (modulo isomorphism) 
exactly the $P_\lambda$-witnesses. We will use axioms expressible in the signature $\mathcal{L}^1_\lambda$ to capture all the properties characterizing 
a $\lambda$-precertificate satisfying also Condition \ref{def:keydefP0-3}\ref{def:keydefP0-3-a} of Def. \ref{def:keydefP0}. We will need
the extra predicate symbol $\dot{H}$ to axiomatize Condition \ref{def:keydefP0-3}\ref{def:keydefP0-3-c} of Def. \ref{def:keydefP0}.
%satisfying
%also Condition\footnote{It is not evident that Condition \ref{def:keydefP0-3}\ref{def:keydefP0-3-c} of Def. \ref{def:keydefP0}
%can be axiomatized by a $\mathfrak{L}_{\infty,\omega}$-sentence of $\bigwedge\bigvee$-type in the signature $\mathcal{L}^1_\lambda$, 
%because it is expressing a property linking the structure 
%$(X_{(\mathcal{J},r,f,F,K)},\in)$ to the admissible assignments $\nu:\mathcal{L}^1_\lambda\to X_{(\mathcal{J},r,f,F,K)}$; in particular it is apparently a property which is not determined by $(X_{(\mathcal{J},r,f,F,K)},\in)$ alone.  
%Condition \ref{def:keydefP0-3}\ref{def:keydefP0-3-c} can be axiomatized in language containing 
%an extra predicate symbol whose extension is  the filter $\Sigma_\nu\cap\SL{\lambda}$. 
%We will sidestep this definability problem using a genericity argument to infer that $V$-generic filters for $P_\lambda$ induce 
%$P_\lambda$-certificates and $P_\lambda$-witnesses.} 
%\ref{def:keydefP0-3}\ref{def:keydefP0-3-a} of Def. \ref{def:keydefP0}.
 
%%%%%%%%%%
%20/9/2020 11.44
%%%%%%%%%%%
 
 \begin{lemma}\label{fac:keyfacSEMCERTPLAMBDA}
Let $\Sigma^{1}_\lambda$ be the $\mathcal{L}^1_\lambda$-theory for 
$\mathfrak{L}_{\infty,\omega}$ containing:
\begin{itemize}
\item The $\mathcal{L}^1_\lambda$-equality axioms and the 
$\mathcal{L}^1_\lambda$-quantifier elimination axioms.
\item
All axioms of $\Sigma^0$. 
%of 
%type \ref{fac:keyfacSEMCERTP0-J},  \ref{fac:keyfacSEMCERTP0-TJ}.
%with last axiom of
%\ref{fac:keyfacSEMCERTP0-14-J} replaced by the sentence
%\[
%\dot{N}_\alpha(c)\wedge\dot{Q}(c)\rightarrow \Sat_{\dot{Q}}(\gp{c\in H_{\omega_1}})
%\]
%for all $\alpha<\omega_1^V$ and constant symbol $c\in\mathcal{L}^1_\lambda$. 
%\item
%The axioms obtained by replacing  $H_\kappa$ by $Q_\lambda$ and $T$ by $T_\lambda$ in all
%axioms of $\Sigma^0$ of type \ref{fac:keyfacSEMCERTP0-V}.
%\end{itemize}
%$\Sigma^1_\lambda$ extends $\Sigma^0_\lambda$ adding also 
\item The following list of axioms:

\begin{enumerate}[(I)]
\setcounter{enumi}{3}
\item \label{fac:keyfacSEMCERTP-LAMBDA-X}
\emph{}

\begin{enumerate}
\item \label{fac:keyfacSEMCERTP-LAMBDA-00}
For all constant symbols $c$ of sort $\dot{Q}$
\[
\bigvee_{x\in Q_\lambda}\check{\check{x}}=c.
\]
\item \label{fac:keyfacSEMCERTP0-LAMBDA-01}
For all $x, y\in Q_\lambda$ %\cup\bp{Q_\lambda}$ 
\[
(\check{\check{x}}\in \check{\check{y}})\text{ if $x\in y$,}
\] 
\[
(\check{\check{x}}\not\in \check{\check{y}}) \text{ if $x\not\in y$.} 
\]
\item \label{fac:keyfacSEMCERTP-LAMBDA-02}
For all constant symbols $c$ of sort $\dot{Q}$
\[
c\in \dot{K}\rightarrow c\in\check{\check{\omega}}_1^V.
\]
\item \label{fac:keyfacSEMCERTP-LAMBDA-03}
For all $\alpha\in\omega_1^V$
\[
(\check{\check{\alpha}}\not\in\dot{K})\rightarrow \qp{(\dot{F}_0(\check{\check{\alpha}})=\check{\check{0}})\wedge 
(\bigwedge_{n\in\omega}\dot{e}_{n,\alpha}=\check{\check{0}})\wedge \dot{Z}_\alpha=\check{\check{0}}}.
\]
\item \label{fac:keyfacSEMCERTP-LAMBDA-04}
For all constant symbols $c,d,e$ of sort $\dot{Q}$ 
\[
(\dot{F}_0(c)=d\wedge c\in \dot{K})\rightarrow\bigvee_{\eta\in C\cap \lambda}(d=\check{\check{\eta}}),
\]
\[
[(\dot{F}_0(c)=d)\wedge (\dot{F}_0(c)=e)]\rightarrow (d=e),
\]
%\[
%c\in\omega_1\rightarrow \bigvee_{\eta<\lambda}(\dot{F}_0(c)=\check{\eta}).
%\]

\item \label{fac:keyfacSEMCERTP-LAMBDA-1}
For all constant symbols $c$ of sort $\dot{Q}$ and $\alpha<\omega_1^V$
\[
c\in \dot{X}_\alpha\leftrightarrow\bigvee_{n<\omega}\dot{e}_{n,\alpha}=c.
\]
\item \label{fac:keyfacSEMCERTP-LAMBDA-2}
For all constant symbols $c$ of sort $\dot{Q}$and $\alpha<\omega_1^V$
\[
\check{\check{\alpha}}\in \dot{K}\rightarrow \qp{(c\in \dot{X}_\alpha\wedge c<\check{\check{\omega}}_1^V)\leftrightarrow c\in\check{\check{\alpha}}}.
\]

\item \label{fac:keyfacSEMCERTP-LAMBDA-3}
For all constant symbols $c$ of sort $\dot{Q}$ and $\alpha<\beta<\omega_1^V$
\[
(\check{\check{\alpha}}\in\dot{K}\wedge\check{\check{\beta}}\in\dot{K})\rightarrow (c\in \dot{X}_\alpha\rightarrow c\in \dot{X}_\beta).
\]

\item \label{fac:keyfacSEMCERTP-LAMBDA-4}
For all constant symbols $c,d$ of sort $\dot{Q}$ and $\alpha<\beta<\omega_1^V$
\[
\qp{\check{\check{\alpha}}\in\dot{K}\wedge\check{\check{\beta}}\in\dot{K}\wedge\dot{F}_0(\check{\check{\alpha}})=c\wedge \dot{F}_0(\check{\check{\beta}})=d}
\rightarrow c\in d.
\]
\item \label{fac:keyfacSEMCERTP-LAMBDA-5}
For $\alpha<\omega_1^V,\eta<\kappa$, 
$\mathcal{R}=\bp{\in,U_0,U_1}$ the signature extending $\bp{\in}$ with unary predicates
$U_0,U_1$, let $\Psi_{\alpha\eta}$ be the sentence
\[
\bigwedge_{\psi\in \mathrm{Form}_{\mathcal{R}},n<\omega,
%c_1,\dots,c_n\in\bp{\dot{e}_{n,\alpha}:n<\omega}
i_1,\dots,i_n<\omega}\,\qp{\Sat_{0,\alpha}(\gp{\psi(e_{i_1,\alpha},\dots,e_{i_n,\alpha})})
\leftrightarrow \Sat_{1,\alpha}(\gp{\psi(e_{i_1,\alpha},\dots,e_{i_n,\alpha})})} 
\]
where $\mathrm{Form}_{\mathcal{R}}$ denotes the set of $\mathcal{R}$-formulae.

We add the axioms
\begin{align*}
\qp{(\check{\check{\alpha}}\in\dot{K})\wedge (\dot{F}_0(\check{\check{\alpha}})=\check{\check{\eta}})}
\rightarrow \Psi_{\alpha\eta}
\end{align*}
for all $\alpha<\omega_1^V,\eta\in C\cap\lambda$.
%\end{enumerate}
%\end{enumerate}
%
%Now for each $\lambda\in C\cup\bp{\kappa}$, we let $\Sigma_\lambda$ be the theory extending $\Sigma^1_\lambda$ with the axioms:
%\begin{enumerate}[(I)]
%\setcounter{enumi}{4}
%\item \label{fac:keyfacSEMCERTP-LAMBDA-Y}
%
%\emph{}
%
%\begin{enumerate}
\item \label{fac:keyfacSEMCERTP-LAMBDA-6}
For all $\alpha<\omega_1^V$, $\eta\in C\cap\lambda$ 
\[
(\check{\check{\alpha}}\in\dot{K}\wedge \dot{F}_0(\check{\check{\alpha}})=\check{\check{\eta}})\rightarrow 
\qp{(\bigwedge_{n\in\omega}\bigvee_{x\in Q_\eta} \dot{e}_{n,\alpha}=\check{\check{x}})
\wedge \dot{Z}_{\alpha}=\check{\check{P}}_\eta}.
\]
%
%
%\item \label{fac:keyfacSEMCERTP-LAMBDA-6}
%For all $\alpha<\omega_1^V$, $\eta\in\lambda\cap C$
%\[
%(\dot{F}_0(\check{\alpha})=\check{\eta}\wedge \check{\alpha}\in\dot{K})\rightarrow (\dot{Z}_\alpha=\check{P}_\eta).
%\]

%\item \label{fac:keyfacSEMCERTP-LAMBDA-2}
%For all constant symbols $c$ of $\mathcal{L}^1$ and $\alpha<\omega_1^V$
%\[
%(\check{\alpha}\in \dot{K}\wedge c\in \dot{X}_\alpha)\rightarrow \bigvee_{x\in Q_\lambda}\check{x}=c.
%\]

%\item \label{fac:keyfacSEMCERTP-LAMBDA-4}
%For all constant symbols $c$ of $\mathcal{L}^1$ and $\alpha<\beta<\omega_1^V$
%\[
%\qp{\check{\alpha}\in\dot{K}\wedge\dot{F}_0(\check{\alpha})=c}\rightarrow \bigvee_{\eta\in C\cap\lambda}c=\check{\eta}.
%\]
\end{enumerate}
\end{enumerate}
\end{itemize}

For any atomic sentence $\psi\in\SL{\kappa}$, let $x_\psi=\bp{\psi}$.
%$\gn{\psi}$ denote the sentence as an element of the set $\SL{\lambda}$.
$\Sigma^{1*}_\lambda$ adds to $\Sigma^1_\lambda$ the axioms:
\begin{enumerate}[(I)]
\setcounter{enumi}{4}
\item \label{fac:keyfacSEMCERTP-LAMBDA-Y}
\emph{}
\begin{enumerate}
\item \label{fac:keyfacSEMCERTP-LAMBDA-Y1}
for all atomic sentences $\psi\in\SL{\lambda}$, the axiom
\[
\psi\leftrightarrow (\check{\check{x}}_\psi\in\dot{H})
\]
\item \label{fac:keyfacSEMCERTP-LAMBDA-Y2}
for all $\alpha<\omega_1^V$, $\eta\in C\cap \lambda$, $c_1,\dots,c_n\in\bp{\dot{e}_{n,\alpha}:n<\omega}$, $\psi(y,x_1,\dots,x_n)$ formula of
$\bp{\in,U_0,U_1}$ the axiom
\begin{align*}
&\qp{(\check{\check{\alpha}}\in \dot{K})\wedge(\dot{F}_0(\check{\check{\alpha}})=\check{\check{\eta}})\wedge
\Sat_{0,\alpha}(\gp{\forall x \qp{U_0(x)\rightarrow \exists y\, (U_0(y)\wedge x\subseteq y\wedge\psi(y,c_1,\dots,c_n))} })}&\\
&\rightarrow&\\ 
&\bigvee_{n<\omega} \qp{\Sat_{0,\alpha}(\gp{U_0(\dot{e}_{n,\alpha})\wedge\psi(\dot{e}_{n,\alpha},c_1,\dots,c_n)})\wedge \dot{e}_{n,\alpha}\in \dot{H}}&
\end{align*}
\end{enumerate}
\end{enumerate}

\smallskip

Then for any $\lambda\in C\cup\bp{\kappa}$:
\begin{enumerate}
\item \label{keyfactPkappa0}
$\Sigma^{1}_\lambda$ and $\Sigma^{1*}_\lambda$ are standard theories of $\bigwedge\bigvee$-type.
Moreover any model of $\Sigma^{1}_\lambda$ is such that $\in_{\dot{Q}}$ and $\in_{\dot{N}_\alpha}$ for 
$\alpha<\omega_1^V$ are interpreted by well-founded extensional relations whose transitive collapses determines 
a unique $\lambda$-precertificate $(\mathcal{J},r,f,F,K)$.

\item \label{keyfactPkappa1}
Given a $\lambda$-admissible $\nu:\mathcal{L}^1_\lambda\to X_{(\mathcal{J},r,f,F,K)}$, let 
$\nu^*$ extend $\nu$ to $\mathcal{L}^{1*}_\lambda$ by letting $\nu(\dot{H})=[\Sigma_\nu\cap\SL{\lambda}]^{<\omega}$.
Then
$\Sigma^{1*}_\lambda\subseteq\Sigma_{\nu^*}$ if and only if $\nu$ is a weak $P_\lambda$-witness for some weak 
$P_\lambda$-certificate
$(\mathcal{J},r,f,F,K)$.

\item \label{keyfactPkappa2}
Assume 
$\mathcal{N}$ is a $\mathfrak{L}_{\infty,\omega}$-model of 
$\Sigma^{1*}_\lambda$ existing in some transitive model $W$, and
let
%\begin{itemize}
%\item
%The domain $N$ of $\mathcal{N}$ is well founded and extensional, and the transitive collapse of $(N,\in^{\mathcal{N}})$
%can be extended to an isomorphism of the $\mathcal{L}^1_\lambda$-structure on 
%$\mathcal{N}$ with the $\mathcal{L}^1_\lambda$-structure
%$\mathcal{X}_{(\mathcal{J},r,f,F,K)}$.% is a weak $\lambda$-precertificate $(\mathcal{J},r,f,F,K)$ existing in $W$.
%\item
$\pi^s_{\mathcal{N}}$ denote the Mostowski collapsing map of the structure $(s^{\mathcal{N}},\in_{s})$ for each sort $s$ of 
$\mathcal{L}^1_\lambda$. 

Then the multi-map defined on each $c$ of sort
$s$ by 
\[
\nu_{\mathcal{N}}:c\mapsto \pi^s_{\mathcal{N}}(c^{\mathcal{N}})%, \qquad R\mapsto \pi_{\mathcal{N}}[R^{\mathcal{N}}]
\] 
and naturally defined on the predicates\footnote{I.e. in accordance with Warning \ref{war:1str}.} of $\mathcal{N}$
can be uniquely extended to a weak $P_\lambda$-witness for the weak $P_\lambda$-certificate $(\mathcal{J},r,f,F,K)$.
%\item 
%Moreover the unique condition possibly failing for $(\mathcal{J},r,f,F,K)$ being a $P_\lambda$-certificate and 
%$\nu_{\mathcal{N}}$ being a $P_\lambda$-witness for some $\lambda\in C\cup\bp{\kappa}$ is Condition \ref{def:keydefP0-3}\ref{def:keydefP0-3-c}
%of Def. \ref{def:keydefPkappa}.
%\end{itemize}
%\item
\end{enumerate}
\end{lemma}
\begin{proof}
Along the lines of the proof of Fact \ref{fac:keyfacSEMCERTP0}, using also Remark \ref{rmk:domainlambdaadmnu}.
\begin{description}
\item[\ref{keyfactPkappa0}]
The key point is that the new sorts, constants, and relation symbols define elements or subsets
of the structure $(Q_\lambda,\in)$, hence they cannot produce witnesses to the ill-foundedness of $\in_{\dot{Q}}^{\mathcal{N}}$.

\item[\ref{keyfactPkappa1}] Left to the reader: the new axioms are exactly added to check that any $\nu^*$ witnessing them is
such that $\nu^*\restriction \mathcal{L}^1_\lambda$ is a weak $P_\lambda$-witness.
 
\item[\ref{keyfactPkappa2}] Left to the reader: same argument as for the previous item.
\end{description}
 \end{proof}

This is the first non-trivial result for the forcings $P_\lambda$:

\begin{lemma} \label{lem:permittedPlambda}  
Let $\mathcal{M}_H$ be the canonical term model for $\mathcal{L}^1_\lambda$ induced by 
$\Sigma_H$ whenever 
$H$ is a
$V$-generic filter $H$ for $P_\lambda$ (according to Def. \ref{def:conspropforcing}).

Then $\mathcal{M}_H$ 
models $\Sigma^{1}_\lambda$ and $\in_s^{\mathcal{M}_H}$ is well-founded and extensional 
on each sort $s$ for which $\in_s$ is a predicate symbol of
$\mathcal{L}^1_\lambda$.

\smallskip
Moreover,
let: 
\begin{itemize}
\item
$(\mathcal{J}_H,r_H,f_H,F_H,K_H)$ be the precertificate induced by $H$ via the Mostowski collapse $\pi^s_H$ of 
each sort $s$ of $\mathcal{M}_H$ for which $\in_s$ is in $\mathcal{L}^1_\lambda$ 
with 
$F_H(\alpha)=\ap{\eta^H_\alpha,X^H_\alpha,Z^H_\alpha}$ for all $\alpha<\omega_1^V$, 
\item
$\nu_H:c\mapsto\pi^s_H([c]_H)$ be the associated admissible assignment naturally 
extended\footnote{I.e. in accordance with Warning \ref{war:1str}.} to the relation symbols of
$\mathcal{L}^1_\lambda$. 
\end{itemize}
Then
$(\mathcal{J}_H,r_H,f_H,F_H,K_H)$ is a $P_\lambda$-certificate and $\nu_H$ is a $P_\lambda$-witness
such that $H=[\Sigma_{\nu_H}]^{<\omega}\cap \SL{\lambda}$.
Hence the natural extension of $\mathcal{M}_H$ to an $\mathcal{L}^{1*}_\lambda$-structure makes it a model of $\Sigma^{1*}_\lambda$.
\end{lemma}

\begin{proof}
By Lemma \ref{lem:keyconspropforcing} and Fact \ref{fac:keyfacSEMCERTPLAMBDA}
$\mathcal{M}_H$ is a $\mathcal{L}^1_\lambda$-structure which
models $\Sigma^{1}_\lambda$ whenever $H$ is $V$-generic for $P_\lambda$, since $\Sigma^{1}_\lambda$ is standard of 
$\bigwedge\bigvee$-type, and is contained in
$\Sigma_\nu$ for all $P_\eta$-witnesses $\nu$. By Fact \ref{fac:keyfacSEMCERTPLAMBDA} the map 
$\nu_H:\mathcal{L}^1_\lambda\to X^\lambda_{(\mathcal{J}_H,r_H,f_H,F_H,K_H)}$ is a $\lambda$-admissible interpretation.

We must now check that the natural extension of $\nu_H$ to a $\mathcal{L}^{1*}_\lambda$-assignment $\nu^*_H$ 
obtained by letting $\nu^*_H(\dot{H})=H$
satisfies also the missing axioms of $\Sigma^{1*}_\lambda$. For this we cannot use a syntactic argument, since
the axioms we must check are not part of the theory $\Sigma_H$ one uses to define $\mathcal{M}_H$ and $\nu_H$
(the symbol $\dot{H}$ is not part of the signature for the theory $\Sigma_H$).

The Axioms \ref{fac:keyfacSEMCERTP-LAMBDA-Y1} are in $\Sigma_{\nu^*_H}$ almost by definition.

We must also check that Axioms \ref{fac:keyfacSEMCERTP-LAMBDA-Y2} are also in $\Sigma_{\nu^*_H}$.
For this we use a density argument on $P_\lambda$:

Assume $\alpha\in K_H$ and $F^H_0(\alpha)=\eta$ for some $\eta\in C\cap\lambda$.
Let $D^*$ be a dense subset  of $P_\eta$ definable in $(Q_\eta,\in,P_\eta,A_\eta)$ with parameters $a_1,\dots,a_n$ 
in $X^H_\alpha$.
Find $i_1,\dots,i_n$ such that $\check{\check{a}}_j=\dot{e}_{j,\alpha}\in \Sigma_{\nu_H}$ for all $j=1,\dots,n$. 
Then
\[
p_0=\bp{\check{\check{a}}_j=\dot{e}_{j,\alpha}:j=1,\dots,n}\cup\bp{(\dot{F}_0(\check{\check{\alpha}})=\check{\check{\eta}}),
\check{\check{\alpha}}\in\dot{K}}\in H
\]

Now consider the set 
\[
E=\bp{p\leq p_0: \text{ for some  $q_p\subseteq p$, $q_p\in D^*$ and there is $n_p$ such that }
(\check{q}_p=\dot{e}_{n_p,\alpha})\in p}
\]

We claim that $E\in V$ is dense in $P_\lambda$ below $p_0$:
Assume $p\in P_\lambda$ refines $p_0$.
Then any $P_\lambda$-witness $\nu$ with $p\subseteq\Sigma_\nu$ and associated $P_\lambda$-certificate
$(\mathcal{J},r,f,F,K)$ with $F_1(\alpha)=X_\alpha$
is such that
there is some $q_0\in [\Sigma_\nu]^{<\omega}\cap X_\alpha\cap D^*$, since $D^*$ is a dense subset of $P_\eta\cap Q_\eta$ definable
in parameters $a_1,\dots,a_n\in X_\alpha$ in the structure $(Q_\eta,\in,P_\eta,A_\eta)$, by 
Condition \ref{def:keydefP0-3}\ref{def:keydefP0-3-c} applied to $\nu$.
 
Since $\Sigma^1_\lambda\subseteq\Sigma_\nu$, $\check{q}_0=\dot{e}_{m,\alpha}\in\Sigma_\nu$ for some $m<\omega$.
Now we let
$q=p\cup q_0\cup\bp{(\check{\check{q}}_0=\dot{e}_{m,\alpha})}$. Then $q\in E$ refines $p$.

Therefore $E$ is dense below $p_0$.

Since $p_0\in H$ and $E$ is dense below $p_0$, we can find $p^*\leq p_0\in E\cap H$.
Since $p^*\subseteq \Sigma_{\nu_H}$, we conclude that
Condition \ref{def:keydefP0-3}\ref{def:keydefP0-3-c} is satisfied by $\nu_H$ and $X_{(\mathcal{J}_H,r_H,f_H,F_H,K_H)}$
for $E$ and $X^H_\alpha$  by 
$q_{p^*}\in[\Sigma_{\nu_H}]^{<\omega}\cap D\cap X^H_\alpha$.
$q_{p^*}$ witnesses that $\mathcal{M}_H$ satisfies the conclusion of Axiom \ref{fac:keyfacSEMCERTP-LAMBDA-Y2}
instantiated for $D$ and $\alpha$.

\end{proof}

The following is a key observation explaining how $\Diamond_\kappa$ and Condition \ref{def:keydefP0-3}\ref{def:keydefP0-3-c} 
can be used for proving that $P_\kappa$ is stationary set preserving. 
\begin{lemma}\label{fac:keyfacSEMCERTPLAMBDA-1}
Assume $\dot{X}$ is a $P_\kappa$-name for an unbounded subset of $\omega_1^V$.
Then there are stationarily many $\eta<\kappa$ such that $A_\eta=\dot{X}\cap Q_\lambda$ and
\[
(Q_\eta,\in,P_\eta,A_\eta)\prec (H_\kappa,\in,P_\kappa,\dot{X})
\]
Assume further that for some such $\eta$,
$r_{\alpha\eta}=\bp{\check{\alpha}\in\dot{K},\dot{F}_0(\check{\alpha})=\check{\eta}}$ is a condition in $P_\kappa$
for some $\alpha<\omega_1^V$.

Then 
\[
r_{\alpha\eta}\Vdash_{P_\kappa}\check{\alpha}\text{ is a limit point of $\dot{X}$}.
\]
%Then for any $p\in P_\kappa$ refining $r_{\alpha\eta}$ and any 
%$A\subseteq P_\eta$ dense in $P_\eta$ and definable by parameters in $P_\eta$, there is $s\in A$ such that
%$s\cup p\in P_\kappa$.
\end{lemma}
\begin{proof}
$\dot{X}$ is naturally represented by a subset of $P_\kappa\times\omega_1$ identifying the pairs $(p,\alpha)$ such that
$p\Vdash_\kappa\check{\alpha}\in \dot{X}$. Hence $\dot{X}\subseteq H_\kappa$.
By $\Diamond_\kappa$ there are stationarily many $\eta<\kappa$ such that
\[
(Q_\eta,\in,P_\eta,A_\eta)\prec (H_\kappa,\in,P_\kappa,\dot{X}).
\]
Let $H$ be $P_\kappa$-generic with $r_{\alpha\eta}\in H$.
By the previous Lemma we obtain that $H$ induces a unique $P_\kappa$-certificate $(\mathcal{J}_H,r_H,f_H,F_H,K_H)$
and a $P_\kappa$-witness $\nu_H$ associated to $(\mathcal{J}_H,r_H,f_H,F_H,K_H)$ such that
$H=[\Sigma_{\nu_H}]^{<\omega}\cap \SL{\kappa}$.

Now the sets 
\[
D_\xi=\bp{p\in P_\kappa: \text{ there is }\omega_1^V>\gamma>\xi \, (p\Vdash_{P_\kappa}\gamma\in \dot{X})}
\]
are dense in $P_\kappa$ and definable in parameter $\xi$ in the structure $(H_\kappa,\in, P_\kappa,\dot{X})$
for any $\xi<\omega_1^V$.
By assumption
\begin{equation}\label{eqn:useofdiamond}
(Q_\eta,\in,P_\eta,A_\eta)\prec (H_\kappa,\in,P_\kappa,\dot{X}),
\end{equation}
while (letting $F_H(\alpha)=\ap{\eta,X_\alpha,Z_\alpha}$) Condition \ref{def:keydefP0-3} of Def. \ref{def:keydefPkappa} ensures that:
\begin{enumerate}[(i)]\label{enum:useofdiamond}
\item\label{enum:useofdiamond-1}
$Z_\alpha=P_\eta$,
\item\label{enum:useofdiamond-2}
$(X_\alpha,\in,Z_\alpha\cap X_\alpha,A_\eta\cap X_\alpha)\prec (Q_\eta,\in,P_\eta,A_\eta)$,
\item\label{enum:useofdiamond-3}
$X_\alpha\cap\omega_1^V=\alpha$,
\item\label{enum:useofdiamond-4}
$A\cap X_\alpha\cap H$ is non empty for all $A$ dense subset of of $P_\eta$ definable by parameters in $X_\alpha$.
\end{enumerate}
By \ref{eqn:useofdiamond}, we get that $D_\xi\cap P_\eta$ is a dense subset of $P_\eta$ definable in 
$(Q_\eta,\in,P_\eta,A_\eta)$
by parameter $\xi<\omega_1^V$. By \ref{enum:useofdiamond-1}, \ref{enum:useofdiamond-2}, \ref{enum:useofdiamond-3}, we get that
$D_\xi\cap P_\eta$ is definable in 
\[
(X_\alpha,\in,Z_\alpha\cap X_\alpha,A_\eta\cap X_\alpha)
\]
for all $\xi<\alpha$.
By \ref{enum:useofdiamond-4} we get that there is some $q\in H\cap D_\xi\cap X_\alpha$.
Now observe that for any $q$ in this set, 
the least $\gamma>\xi$ such that $q$ forces $\gamma\in D_\xi$ is such that $\gamma<\alpha$ (by 
\ref{enum:useofdiamond-2} and \ref{enum:useofdiamond-3}).
This gives that $\alpha$ is a limit point of $\dot{X}_H$ in $V[H]$.
We are done.
\end{proof}

We are now ready to complete the proof of Theorem \ref{thm:mainthm00}.

%\begin{definition}
%$F:\kappa\to H_{\omega_2}$ is a Laver function for strong embeddings if $\kappa$ is a strong cardinal and for 
%every $A\in V$ there exists an elementary $j:V\to M$ such that:
%\begin{itemize}
%\item
%$\crit(j)=\kappa$,
%\item$M\subseteq V$,
%\item
%$A_\in V_\eta\subseteq M$,
%\item
%$j(F)(\kappa)=A$.
%\end{itemize}
%\end{definition}
%Assume $F:\kappa\to H_{\omega_2}$ is a Laver function for strong 
%embeddings\footnote{We actually need only the guessing property of $F$ for $A\in V_{\kappa+1}$.}.

%\begin{theorem}\label{thm:mainthmPkappa}
%Assume $H$ is $V$-generic for $P_\kappa$.
%Let 
%\[
%\mathcal{J}_H=\bp{j_{\alpha\beta}:N^H_\alpha\to N^H_\beta:\,\alpha\leq\beta\leq\omega_1^V}
%\]
%be the iteration induced by $H$.
%Then for all $S\in \pow{\omega_1}^{N_{\omega_1}}$
%\[
%(N_{\omega_1},\in)\models \text{S}\text{ is stationary } \qquad \Longleftrightarrow \qquad (V[H],\in)\models \text{S}\text{ is stationary. }
%\] 
%\end{theorem}
%
%The key to the proof of Thm. \ref{thm:mainthmPkappa} is a sophisticated variation of Lemma \ref{lem:keylemASPSCH(*)0}.

\begin{proof}
Fix $\dot{C}$ $P_\kappa$-name for a club subset of $\omega_1^V$.

Towards a contradiction assume that
for some $H$ $V$-generic for $P_\kappa$ there is $S\in N_{\omega_1}^H$ such that
\[
\dot{C}_H\cap S=\emptyset,
\]
but
\[
(N_{\omega_1}^H,\in)\models S\subseteq\omega_1^V\text{ is stationary}.
\]

Since $S\in N_{\omega_1}^H$, $S=\nu_H(\dot{c}_{j,\omega_1^V})$, 
 for some $j<\omega_1^V$.

This gives that
\[
\Sat_{N_{\omega_1^V}}(\gp{\dot{c}_{j,\omega_1^V}\subseteq\dot{\omega}_1^V\text{ is stationary}})\in \Sigma_H.
\]
Note that the latter is an atomic sentence of $\mathcal{L}^1$.
Therefore
%we can find 
% some $p_0\in P_\kappa$  
% \[
%(\dot{N}_{\omega_1^V},\in)\models \nu_H(c_{k,\omega_1^V})=\pow{\omega_1}\setminus \NS_{\omega_1}.
%\]
%$\dot{C}\cap \nu_H(c_{j,\omega_1^V})=\emptyset$. 
%
%This occurs if and only if 
\[
p=
\bp{\Sat_{N_{\omega_1^V}}(\gp{\dot{c}_{j,\omega_1^V}\subseteq\dot{\omega}_1^V\text{ is stationary}})}\in H
\]
and some $p_0\in H$ refining $p$ forces
\[
\dot{C}\cap \nu_{\dot{H}}(\dot{c}_{j,\omega_1^V})=\emptyset.
\]

Then for any $V$-generic filter $H$ for $P_\kappa$ with $p_0\in H$, $V[H]$ models that
\[
\dot{C}_H\cap \nu_H(\dot{c}_{j,\omega_1^V})=\emptyset
\]
and
\[
(N^H_{\omega_1^V},\in)\models \nu_H(\dot{c}_{j,\omega_1^V})\subseteq \omega_1
\text{ is stationary}.
\]
Our aim is to find $r\leq p_0$ forcing that $\dot{C}\cap \nu_H(c_{j,\omega_1^V})$ is non-empty, thus reaching a contradiction.

By $\Diamond_\kappa$, 
\[
S_{\dot{C}}=\bp{\lambda<\kappa: A_\lambda=\dot{C}\cap Q_\lambda}
\]
is stationary.

Pick $\lambda\in S_{\dot{C}}$ with $p_0\in P_\lambda=P_\kappa\cap Q_\lambda$ and
\[
(Q_{\lambda},\in,P_\lambda, A_\lambda)\prec (H_\kappa,\in,P_\kappa,\dot{C}).
\]

Note that for $\alpha<\omega_1^V$ the sets
\[
D_\alpha=\bp{p\in P_\kappa: \, \exists \xi>\alpha \, (p\Vdash_{P_\kappa}\check{\xi}\in\dot{C})}
\]
are dense open in $P_\kappa$, hence such that $D_\alpha\cap Q_\lambda$ is dense in $P_\lambda$,
and definable in parameter $\alpha$ 
in the structure $(Q_\lambda,\in,P_\lambda,A_\lambda)$.

Now to find this $r$ we proceed along the lines of the proof of Lemma \ref{lem:keylemASPSCH(*)0} and use 
Lemma \ref{fac:keyfacSEMCERTPLAMBDA-1} in the generic ultrapower $\bar{M}$ to be defined below.

Let $G$ be $V$-generic for $\Coll(\omega,\kappa)$. Now $2^\lambda\leq\kappa$ (by $\Diamond_\kappa$), 
and $p_0\in P_\lambda=P_\kappa\cap Q_\lambda$ with $P_\lambda$ in $V$ of size $|\lambda|<\kappa$.
Hence
we can fix in
$V[G]$ a $V$-generic filter $H$ for $P_\lambda$ with $p_0\in H$. 
By Lemma \ref{lem:permittedPlambda}, the structure
$\mathcal{M}_H$ for $\mathcal{L}^{1}_\lambda$ induced by 
$H$ is such that all the relations $\in_s$ are well founded, and the transitive collapse of their domains via the Mostowski map $\pi^s_H$ extend to a multi-map which is an isomorphism of 
$\mathcal{M}_H$ with
the $\mathcal{L}^{1}_\lambda$-structure 
\[
\mathcal{X}_{(\mathcal{J}_H,r_H,f_H,F_H,K_H)}
\]
with $(\mathcal{J}_H,r_H,f_H,F_H,K_H)$ a $P_\lambda$-certificate and
$\nu_H:c\mapsto \pi^s_H([c]_H)$ for $c$ a constant of sort $s$ a $P_\lambda$-witness.
We take note that $(\mathcal{J}_H,r_H,f_H,T)$ is a semantic certificate for  $A,D$.
To simplify notation we let $\nu_H=\nu$ in what follows.

%let $\nu=\nu_H$ be the $P_\kappa$-witness
%
%$\kappa$-admissible $\nu$ and find again in $V[G]$ a 
%$P_\kappa$-witness $(\nu,F_0,K_0)$ such that $p_0\subseteq\Sigma_\nu$ and $s_0\subseteq F_0$.
%
%Let $(\mathcal{J},r,f,T)$ be the semantic certificate associated to $\nu$.

%
%
%%From now on we fix a $(p,s)\in P_\kappa$ and we denote $H_{(p,s)}$ by $H$.
%
%%To find this $r$ we do a sophisticated variation of the proof of Lemma \ref{lem:keylemP0gen}.
%  
%
%%Let 
%%\[
%%\Sigma_\nu=\bp{\psi\in\SL{\kappa}: \exists \,(p,s)\in H\text{ with }\psi\in p},
%%\] 
%%and
%$\mathcal{M}_H$ be the $\mathcal{L}^1$-structure induced by $\Sigma_H$.
%By (the same argument of) Lemma \ref{lem:admGenstructureP0}
%$\mathcal{M}_H$ models $\Sigma_0$, hence it is well founded and has a uniquely associated semantic certificate
%$(\mathcal{J}_H,r_H,f_H,T)$ with (the restriction to the signature $\in$ of) $\mathcal{M}_H$
%isomorphic to
%\[
%(X^\kappa_{(\mathcal{J}_H,r_H,f_H,T)},\dot{\in}^H).
%\]
%with $X^\kappa_{(\mathcal{J}_H,r_H,f_H,T)}=H_{\omega_2}\cup\bp{H_{\omega_2}}\cup X_{(\mathcal{J}_H,r_H,f_H,T)}$. 
%Moreover the map $\nu_H: \mathcal{L}^1 \to X^\kappa_{(\mathcal{J}_H,r_H,f_H,T)}$ defined by $c\mapsto \pi_H[([t]_H)$
%is a $\kappa$-admissible evaluation for the semantic certificate $(\mathcal{J}_H,r_H,f_H,T)$.

%By the same arguments of the proof of Lemma \ref{lem:keylemP0gen} we get that
%$(\mathcal{J}_H,r_H,f_H,T)$ is a semantic certificate in $V[H]$ as witnessed by the assignment
%$\nu_H:t\mapsto \pi_H([t]_H)$.
  
Let us denote
\[
\mathcal{J}_H=\bp{j_{\alpha\beta}:N_\alpha\to N_\beta:\, \alpha\leq\beta\leq\omega_1^V}.
\] 

Note that all these objects (i.e. $A,\dot{C},S,T,\Diamond_\kappa,P_\lambda,\nu, \mathcal{J}_H,r_H,F_H,K_H,\dots$) 
belong to $H_{\omega_1}^{V[G]}$. 
%Let also 
%\begin{equation}
%F_0:K\to \kappa\times\pow{H_{\omega_2}^V}^2
%\end{equation}
%be in $V[G]$ a map witnessing that
%$(p_0,s_0)$ is in $P_\kappa$, i.e. such that Condition \ref{def:keydefP0-3-c} of definition \ref{def:keydefP0-3}
%is satisfied by $(p_0,s_0)$ as witnessed by $F_0$
%($F_0$ exists in $V[G]$ since $V$ models that $(p_0,s_0)\in P_\kappa$ hence
%in any generic extension of $V$ in which $H_{\omega_2}$ is countable such a map $F_0$ must be found).

Let now $G_0$ be in $V[G]$ an $N_{\omega_1^V}$-generic filter for $(\pow{\omega_1}/\NS_{\omega_1})^{N_{\omega_1^V}}$ with 
$S=\nu_H(\dot{c}_{j,\omega_1^V})\in G_0$ 
($G_0$ exists because $S$ is stationary in $N_{\omega_1^V}$). 
Let
$k:N_{\omega_1^V}\to N^*_1$ be the ultrapower embedding induced by $G_0$.
This gives that $\omega_1^V\in k(S)$. 
Extend $k$ in $V[G]$ to an iteration
\[
\mathcal{K}=\bp{k_{\alpha\beta}:N^*_\alpha\to N^*_\beta:\, \alpha\leq\beta\leq\omega_1^{V[G]}=\rho}
\] 
of
$N_{\omega_1^V}=N^*_0$ with 
$k=k_{01}$ using the fact that $N_0$ (and therefore also $N_{\omega_1^V}=N_0^*$) is iterable in $V[G]$.

Note that
\begin{equation}\label{eqn:omega1VinS}
\omega_1^V\in k_{0\rho}(S)=k_{0\rho}(\nu_H(\dot{c}_{j,\omega_1^V})).
\end{equation}

Since $H_{\omega_2}^V\subseteq N_{\omega_1^V}$, $\NS^V=\NS^{N_{\omega_1^V}}\cap V$, and 
$V$ models $\NS_{\omega_1}$\emph{ is saturated}, we can extend uniquely
$k_{0\rho}\restriction H_{\omega_2}^V$ to an elementary 
$\bar{k}:V\to\bar{M}$ applying \cite[Lemma 1.5, Lemma 1.6]{HSTLARSON} 
to\footnote{More precisely,
these assumptions on $V$ and $N_{\omega_1^V}$ give that all maximal antichains of
$(\pow{\omega_1}/_\NS)^{V}$ are in $N_{\omega_1^V}$ (hence $G_0\cap V$ is $V$-generic for
$(\pow{\omega_1}/_\NS)^V$). 
Inductively, letting
$\bp{G_\alpha:\alpha<\rho}$ be the family of filters on $(\pow{\omega_1}/_\NS)^{k_{0\alpha}(N_{\omega_1}^V)}$ defining
$\mathcal{K}$, 
and assuming the existence of
$\bar{k}_{0\alpha}:V\to\bar{M}_\alpha$ unique extension to $V$ of 
$k_{0\alpha}\restriction H_{\omega_2}^V$, 
one can check (by elementarity of $\bar{k}_{0\alpha}$ and of $k_{0_\alpha}$)
that 
\[
\NS_{\omega_1}^{\bar{M}_\alpha}=\NS_{\omega_1}^{N^*_\alpha}\cap\bar{k}_{0\alpha}(H_{\omega_2}^V),
\] 
and that
$\bar{M}_\alpha$ models  $\NS_{\omega_1}$\emph{ is saturated}; this gives that all maximal antichains of
$(\pow{\omega_1}/_\NS)^{\bar{M}_{\alpha}}$ are maximal antichains of 
$(\pow{\omega_1}/_\NS)^{N^*_\alpha}$; therefore  one can also inductively infer that
$G_\alpha\cap\bar{M}_\alpha$ 
is $\bar{M}_\alpha$-generic for $(\pow{\omega_1}/_\NS)^{\bar{M}_\alpha}$.
We conclude that $\bp{G_\alpha\cap\bar{M}_\alpha:\alpha<\rho}$ gives rise to a unique iteration 
\[
\bp{\bar{k}_{\alpha\beta}:\alpha\leq\beta\leq\rho}
\]
of $V$
such that $k_{0\beta}\restriction H_{\omega_2}^V=\bar{k}_{0\beta}\restriction H_{\omega_2}^V$ for all $\beta\leq\rho$.
We can set $\bar{k}=\bar{k}_{0\rho}$.} 
$k_{0\rho}\restriction H_{\omega_2}^V$.

%since $\bar{k}(S)=k_{0\omega_1}(S)=k_{1\omega_1}\circ k(S)$.

Let us denote by
\[
\bar{\mathcal{J}}=\bp{\bar{j}_{\alpha\beta}:N_\alpha\to N_\beta:\, \alpha\leq\beta\leq\omega_1^{V[G]}=\rho}
\] 
the iteration obtained letting $\bar{j}_{\alpha\beta}=j_{\alpha\beta}$ if $\beta\leq\omega_1^V$,
$\bar{j}_{\omega_1^V+\alpha\beta}=k_{\alpha\beta}$,
$\bar{j}_{\alpha\beta}=k_{0\beta}\circ j_{\alpha\omega_1^V}$ if $\beta\geq\omega_1^V\geq\alpha$.

%We now encounter a first delicate point: we can easily argue that 
%$\bar{k}(Q_\lambda)\subseteq N^*_\rho$ in $V[G]$, by elementarity of $\bar{k}$ and Lemma \ref{fac:keyfacSEMCERTPLAMBDA}.\ref{keyfactPkappa0},
%but we cannot at all assert that $\bar{k}(H_\kappa)\subseteq N^*_\rho$ which is what we would need to hope for 
%but we can also infer that
%it is also $\bar{M}$-standard (hence $\bar{k}(H_\kappa)$-standard) 
%using Fact \ref{fac:standiter} for the canonical pair\footnote{Ultimately the fact that $\bar{\mathcal{J}}$ is $\bar{k}(H_\kappa)$-standard
%(while we are not able to infer that $\bar{\mathcal{J}}\restriction \omega_1^V+1$ is $k_{01}(H_\kappa)$-standard) is 
%one of the key argument requiring us to prolonge $k_{01}$ to $\bar{\mathcal{K}}$.} $(\bar{M},V[G])$.

%Note that $\bar{\mathcal{J}}$ is a $\bar{M}$-standard iteration in $V[G]$ whose first model is $N_0$.
%Let $\bar{k}:V\to \bar{M}$ be the induced ultrapower embedding given by the
%iteration of $V$ according to 
%$\bar{\mathcal{J}}$ to be defined below.

%Now the very definition of $\bar{X}$ is set up to obtain the following: 
%\begin{claim}\label{fac:auxfacPkappa}
%For any $\bar{G}$ $V$-generic for $\Coll(\omega,\kappa^+)$
%$V[\bar{G}]$ is a $\UB^{\bar{X}}$-canonical model for $\bar{M}$ and
%$H_{\omega_1}\bar{M}[\bar{G}]$ is a $\tau_{\UB^{\bar{X}}}$-substructure of $H_{\omega_1}^{V[G]}$.
%\end{claim}
%\begin{proof}
%We must show that for all $U\in \UB^{\bar{X}}$
%\[
%\bar{k}(U)^{\bar{X}[\bar{G}]}=B^{V[G]}\cap \bar{M}[\bar{G}].
%\]
%towards this aim observe that $\bar{k}(T_U),j(S_U)$ project to complement in 
%\end{proof}

We can easily check this first property of $\bar{\mathcal{J}}$ (recall Def. \ref{def:weaksemcert0}):

\begin{claim}\label{fac:auxfacPkappa}
$(\bar{\mathcal{J}},r_H,\bar{k}[f_H],\bar{k}(T))$ is in $V[G]$ 
a weak semantic certificate for $\bar{k}(A)$.
\end{claim}

\begin{proof}
The $\Delta_1$-properties outlined in Remark \ref{rmk:keyremsemcert} are easily checked for $(\bar{\mathcal{J}},r_H,\bar{k}[f_H],\bar{k}(T))$
using the elementarity of $\bar{j}_{0\rho}$ and of 
$\bar{k}$. 
%$\bar{\mathcal{J}}$ is standard by Fact \ref{fac:standiter}.
\end{proof}

Our aim will be to reinforce this conclusion to the following:

\begin{claim}\label{clm:fundclmrproof}
%Let $\bar{k}(\Diamond_\kappa)(\lambda)=A_{\bar{k}(\lambda)}$. 
Let:
\begin{itemize}
\item
$f^*=\bar{k}[f_H]$;
\item
$K^*=K_H\cup\bp{\omega_1^V}$; 
\item for $\alpha\in \omega_1^V$
\[
F^*_0(\alpha)=\bar{k}((F_H)_0(\alpha)),\qquad F^*_1(\alpha)=\bar{k}[(F_H)_1(\alpha)],\qquad F^*_2(\alpha)=\bar{k}((F_H)_2(\alpha)),
\]
\item
$F^*_0(\omega_1^V)=\bar{k}(\lambda),\qquad F^*_1(\omega_1^V)=\bar{k}[Q_\lambda],\qquad F^*_2(\omega_1^V)=\bar{k}(P_\lambda)$;
\item
$F^*(\beta)=\ap{0,1,0}$
for all $\rho>\beta>\omega_1^V$.
\end{itemize}
Then %$(\bar{\mathcal{J}},r_H,f^*,F^*,K^*)$ is a weak $\bar{k}(\kappa)$-precertificate for $\bar{k}(A),\bar{k}(Q_\kappa)$
%$\bar{k}(A),\bar{k}(T),\bar{k}(H_{\omega_2}),\bar{k}(\Diamond_{\kappa},\bar{k}(C\cap\lambda))$ 
in $V[G]$ there is a $\bar{k}(\kappa)$-assignment 
\[
\bar{\nu}:\bar{k}(\mathcal{L}^{1}_\kappa)\to X_{(\bar{\mathcal{J}},r_H,f^*,F^*,K^*)}
\] 
such that:
\begin{itemize}
\item
\begin{equation}\label{eqn:barnuproperty}
\Sigma_{\bar{\nu}}\supseteq\bar{k}[\cup H]\supseteq \bar{k}(p_0),
\end{equation}
\[
(\dot{\omega}_1^V\in\dot{K}),(\dot{F}_0(\dot{\omega}_1^V)=\bar{k}(\lambda))\in \Sigma_{\bar{\nu}},
\]
\item
$\bar{\nu}$ is a 
weak $\bar{k}(P_{\kappa})$-witness making $(\bar{\mathcal{J}},r_H,f^*,F^*,K^*)$ a weak $\bar{k}(P_\kappa)$-certificate 
relative to\footnote{See Remark \ref{rmk:basicPkappa}\ref{rmk:itemrelPkappa}.} $\bar{M}$.
\end{itemize}
\end{claim}

%$\bar{k}(M),\bar{k}(A),\bar{k}(T),\bar{k}(\Diamond_{\kappa})$.
%Moreover 
%\[
%\Sigma_{\bar{\nu}}\supseteq \bar{k}(p_0)\cup\bp{\dot{\omega}_1^V\in\dot{K},\dot{F}_0(\dot{\omega}_1^V)=\bar{k}(\check{\lambda})}.
%\]

%In particular $\nuwitnesses 
%in $\bar{M}[\bar{G}]$ that $p^*=\bar{k}(p_0)\cup \bp{\dot{F}_0(\omega_1^V)=\dot{\bar{k}(\lambda)}})$
%is an element of $\bar{k}(P_\kappa)$.
% and 
%$p^*$ forces in $\bar{k}(P_\kappa)$ that $\omega_1^V$ is a limit point of $\bar{k}(\dot{C})$. 

Suppose we succeed to prove the Claim. 
Let $\bar{G}$ be $V[G]$-generic for $\Coll(\omega,\rho)$. Then all the parameters needed to check the above
are in $H_{\omega_1}^{V[\bar{G}]}$. 
%By the argument we gave in the proof of Lemma \ref{lem:keylemASPSCH(*)0}, 
%%(Cfr. the part of the proof including and following Claim \ref{clm:keyclm0P0}),
%we obtain that 
%\[
%(H_{\omega_1}^{\bar{M}[\bar{G}]},\tau_{\ST}^{\bar{M}[\bar{G}]})
%\]
%%,\bar{k}(B)^{\bar{M}[\bar{G}]}: B\in \UB^{V})\prec
%is a substructure of
%\[
%(H_{\omega_1}^{V[\bar{G}]},\tau_{\ST}^{V[\bar{G}]},B^{V[\bar{G}]}: B\in \UB^{V})
%\]

%Applying again Shoenfield's absoluteness Lemma to $\bar{M}[\bar{G}]\subseteq V[\bar{G}]$
Now
\[
(H_{\omega_1}^{\bar{M}[\bar{G}]},\tau_{\ST}^{\bar{M}[\bar{G}]}) %,\bar{k}(C)^{\bar{M}[\bar{G}]}:C\in \UB^V)
\]
is a $\Sigma_1$-substructure  of
\[
(H_{\omega_1}^{V[\bar{G}]},\tau_{\ST}^{V[\bar{G}]}). %,C^{V[\bar{G}G]}:C\in \UB^V).
\]
Therefore it
models 
\begin{quote}
There are a weak $\bar{k}(P_\kappa)$-certificate $(\mathcal{J}^*,r^*,f^*,F^*,K^*)$ and a weak 
$\bar{k}(P_\kappa)$-witness
$\bar{\nu}$ for $(\bar{\mathcal{J}},r^*,f^*,F^*,K^*)$  such that 
\[
p^*=(\bp{(\omega_1^V\in K^*), 
(F^*_0(\omega_1^V)=\bar{k}(\lambda))}\cup \bar{k}(p_0))\subseteq \Sigma_{\bar{\nu}}.
\]
\end{quote}
This gives that $\bar{M}$ models that $\Coll(\omega,\rho)$ forces that there are a weak $\bar{k}(P_\kappa)$-certificate 
$(\mathcal{J}^*,r^*,f^*,F^*,K^*)$ and 
$\bar{k}(P_\kappa)$-witness $\bar{\nu}$ with $p^*\subseteq \Sigma_{\bar{\nu}}$. 
%By elementarity of $\bar{k}$ we also get that
%$p[\bar{k}(T)]\subseteq \bar{k}(B)^{\bar{M}[G]}$ holds in any forcing extension of $\bar{M}$.

We conclude that $\bar{M}$ models that
$p^*=\bar{k}(p_0)\cup\bp{\dot{\omega}_1^V\in\dot{K},\dot{F}_0(\dot{\omega}_1^V)=\bar{k}(\check{\lambda})}$ 
is in $\bar{k}(P_\kappa)$ as witnessed by $\bar{\nu}$ and $(\mathcal{J}^*,r^*,f^*,F^*,K^*)$.

We also observe that 
\begin{equation}\label{clm:keyclmPkappaSSP-2}
\bar{M}\models p^*\Vdash\omega_1^V\in\bar{k}(\dot{C})
\end{equation}
by applying in $\bar{M}$ Lemma \ref{fac:keyfacSEMCERTPLAMBDA-1} to $\bar{k}(\dot{C})$, $\bar{k}(\lambda)$, and observing that
$(\dot{F}_0(\dot{\omega}_1^V)=\bar{k}(\check{\lambda})), (\dot{\omega}_1^V\in\dot{K})$ is in $p^*$.

Once \ref{clm:keyclmPkappaSSP-2} is proved, by elementarity of $\bar{k}$, we get that in $V$ there is a condition in 
$P_\kappa$ refining $p_0$
and forcing that $\nu_{\dot{H}}(\dot{c}_{j,\omega_1^V})\cap\dot{C}$ is non-empty 
(since we already know by \ref{eqn:omega1VinS}, 
that $\omega_1^V\in k_{0\rho}(S)$). This is the desired contradiction.

To complete the proof of the Theorem we need only to prove Claim \ref{clm:fundclmrproof}.
%To do this we get into gory details.....
 
\subsubsection{Proof of Claim \ref{clm:fundclmrproof}}

First of all we easily check that 
$(\bar{J},r_H,f^*,F^*,K^*)$ is 
 a $\bar{k}(\kappa)$-precertificate relative to $\bar{M}$ which satisfies also Condition \ref{def:keydefP0-3}
\ref{def:keydefP0-3-a} of Def. \ref{def:keydefPkappa}.

This is not hard: 
$(\bar{J},r_H,f^*,\bar{k}(T))$ is a weak semantic certificate for 
$\bar{k}(A)$ by Claim \ref{fac:auxfacPkappa}. 
Now all the other requirements needed to establish that  $(\bar{J},r_H,f^*,F^*,K^*)$ is 
 a $\bar{k}(\kappa)$-precertificate relative to $\bar{M}$ are easily checked using the elementarity of $\bar{k}$,
 for example: 
 \begin{itemize}
 \item
 if $\alpha\in K_H$, and $F_H(\alpha)=\ap{\eta,X_\alpha,P_\eta}$ (it is the case that $Z_\alpha=P_\eta$ since 
 $(\mathcal{J}_H,r_H,f_H,K_H,F_H)$ is a $P_\lambda$-certificate with $\nu_H$ a $P_\lambda$-witness)
 \[
 (X_\alpha,\in,P_\eta\cap X_\alpha,A_\eta\cap X_\alpha)\prec (Q_\eta,\in,P_\eta,A_\eta);
 \]
 by elementarity of $\bar{k}$
 \[
 (\bar{k}[X_\alpha],\in,\bar{k}[P_\eta\cap X_\alpha],\bar{k}[A_\eta\cap X_\alpha])\prec 
 (\bar{k}[Q_\eta],\in,\bar{k}[P_\eta],\bar{k}[A_\eta])\prec
 (\bar{k}(Q_\eta),\in,\bar{k}(P_\eta),\bar{k}(A_\eta)).
 \]
\item Similarly if $\alpha=\omega_1^V$, $X_\alpha=\bar{k}[Q_\lambda]$ and
 \[
 (\bar{k}[Q_\lambda],\in,\bar{k}[P_\lambda],\bar{k}[A_\lambda])= 
 (\bar{k}[Q_\lambda],\in,\bar{k}(P_\lambda)\cap \bar{k}[Q_\lambda],\bar{k}(A_\lambda)\cap \bar{k}[Q_\lambda])\prec
 (\bar{k}(Q_\lambda),\in,\bar{k}(P_\lambda),\bar{k}(A_\lambda)).
 \]
\end{itemize}

Now Condition \ref{def:keydefP0-3}\ref{def:keydefP0-3-c} would also be easy to check if we can ensure
the existence of a $\bar{\nu}$ satisfying \ref{eqn:barnuproperty} as in the Claim.

Assume this is the case. Then
for $\alpha\in K_H$ with $F_0(\alpha)=\eta$ and $E$ dense subset of $\bar{k}(P_\eta)\cap \bar{k}[X_\alpha]$ definable in 
\[
(\bar{k}[X_\alpha],\in,\bar{k}[P_\eta\cap X_\alpha],\bar{k}[A_\eta\cap X_\alpha]),
\]
we have that $E=\bar{k}(U)\cap k[X_\alpha]$ for some $U$ dense subset of $P_\eta$ definable in the structure
\[
(X_\alpha,\in,P_\eta\cap X_\alpha,A_\eta\cap X_\alpha).
\]
Then $H\cap U\cap X_\alpha\neq\emptyset$, giving that 
$k[H]\cap k(U)\cap k[X_\alpha]\neq\emptyset$.

Similarly since $H$ is $V$-generic for $P_\lambda$, we get that
$k[H]\cap k(U)\cap k[Q_\lambda]\neq\emptyset$ for any $U\in V$ dense subset of $P_\lambda$.

This would show that $(\bar{J},r_H,f^*,F^*,K^*)$ is a weak $\bar{k}(P_\kappa)$-certificate relative to $\bar{M}$
as witnessed by 
$\bar{\nu}$.

So we are left with the definition of $\bar{\nu}$.

We modify in $V[G]$, 
$\nu_H$ (from now on denoted as $\nu$) 
to a weak $\bar{k}(P_\kappa)$-certificate 
\[
\bar{\nu}:\bar{k}(\mathcal{L}^{1}_\kappa)\to X_{(\bar{\mathcal{J}},r_H,f^*,F^*,K^*)}
\] 
witnessing that\footnote{To avoid a heavy notation we will free to identify certain symbols $c$ of $\mathcal{L}^1$ 
whose intended meaning is transparent with the corresponding symbol
 $\bar{k}(c)$ of $\bar{k}(\mathcal{L}^1)$. For example we write $\dot{K}$ instead of $\bar{k}(\dot{K})$.
 Similarly we write $\dot{e}_{n,\omega_1^V}$ with the intended meaning of denoting the constants which represents the
 elements of the $\omega_1^V$-th point of the sequence $\bar{k}(\bp{\dot{X}_\alpha:\alpha<\omega_1^V})$, etc.}
$(\bar{\mathcal{J}},r_H,f^*,F^*,K^*)\in V[G]$ is a weak $\bar{k}(P_\kappa)$-certificate.

We first define $\bar{\nu}$ on
\[
\bar{k}[\mathcal{L}^1_\kappa]\cup\bp{\dot{X}_{\omega_1^V},\dot{Z}_{\omega_1^V}}%, \dot{Z}_{\omega_1^V}, \dot{G}_{\omega_1^V}, \dot{e}_{n,\omega_1^V}:n<\omega}
\] 
according to the following five types of sort, constant, or predicate symbols:

\begin{enumerate}[(a)]
\item  \label{const:type1}
$\bar{\nu}(\bar{k}(c))$ is defined ad hoc if $c\in\mathcal{L}^1$ is a sort, constant, or predicate symbol of $\mathcal{L}^1$ among
$\dot{\bp{r,f}},\dot{F}_0,\dot{K}$, $\dot{X}_i$ for each $i< \omega_1^V$, and the constant symbols $\dot{r}$, $\dot{f}$,
or if $c$ is either $\dot{X}_{\omega_1^V}$ or if $c$ is $\text{br}_T$.

\item  \label{const:type3}
$\bar{\nu}(\bar{k}(c))=\bar{k}(\nu(c))$ if $c$ is any symbol of $\mathcal{L}^1\setminus\mathcal{L}^0$ such that 
$c$ has not been already considered in the list \ref{const:type1}.

\item  \label{const:type2}
$\bar{\nu}(\bar{k}(c))=\nu(c)$ if $c$ is a sort or predicate symbol of $\mathcal{L}^0$ among
\[
\dot{N}_\alpha,\dot{G}_\alpha, \Sat_{N_\alpha},\in_{N_\alpha},\Cod^*, 
\dot{j}_{\alpha\beta},
\] 
or $c$ is a constant of sort $\dot{N}_\alpha$ for $\alpha\leq\beta<\omega_1^V$.

% is not of type \ref{const:type1}.

\item  \label{const:type4}
$\bar{\nu}(c)=j_{\omega_1^V\rho}(c)$ if $c$ is a constant symbol of $\mathcal{L}^0$ and $\nu(c)\in N_{\omega_1}$
(e.g. $c$ is $\dot{c}_{i,\omega_1^V}$ for some $i<\omega_1^V$, or $c=\check{x}$ for some $x\in H_{\omega_2}$).

\item  \label{const:type5}
$\bar{\nu}(\dot{j}_{\alpha\omega_1^V})=\bar{j}_{\alpha\rho}$ for $\alpha\leq\omega_1^V$, 
$\bar{\nu}(\dot{N}_{\omega_1}^V)=N^*_\rho$, and
\[
\bar{\nu}(\Sat_{N_{\omega_1^V}})=\bp{\ap{\gp{\psi},(a_1,\dots,a_n)}: (N^*_{\rho},\in)\models\psi(a_1,\dots,a_n)}.
\]
\end{enumerate}

$\bar{\nu}$ is defined as follows on the symbols of type \ref{const:type1}:

\begin{enumerate}[(i)]
% \item
% $\bar{\nu}(\dot{J})=\bar{\mathcal{J}}$, 
%\item\label{rule1nu*}
%$\bar{\nu}(\bar{k}(c))=\bar{k}(\nu(c))$ if  $c$ is of type \ref{const:type2},
%\item 
%$\bar{\nu}(\bar{k}(c))=\bar{j}_{\omega_1^V\rho}(\nu(c))$ if  $c$ is of type \ref{const:type3},
%\item \label{rule2nu*}
%$\bar{\nu}(\bar{k}(c))=\nu(c)$ if  $c$ is of type \ref{const:type4}, % (i.e. $\nu(c)\in\bigcup_{\alpha<\omega_1^V}N_\alpha$), 
%\item \label{rule2nu*-bis}
%$\bar{\nu}(c)=\bp{\bar{\nu}(a),\bar{\nu}(b)}$ if $c$ is of type \ref{const:type5} and
%$\nu(c)=\bp{\nu(a),\nu(b)}$ for $a,b$ constants of type \ref{const:type3}, \ref{const:type4},
%\item \label{rule2nu*-ter}
%$\bar{\nu}(c)=\ap{\bar{\nu}(a),\bar{\nu}(b)}$ if $c$ is of type \ref{const:type5} and
%$\nu(c)=\ap{\nu(a),\nu(b)}$ for $a,b$ constants of type \ref{const:type3}, \ref{const:type4},
%$c\in\mathcal{L}^0$ and $c$ occurs in the 
%list of $\mathcal{L}^0$-symbols given in \ref{eqn:fixedc},
%\item $\bar{\nu}(\dot{X}_i))=\bar{k}(F_0(i))$ for all $i<\omega_1^V$,
%\item \label{rule3nu*}
%$\bar{\nu}(\bar{k}(\dot{N}_{\alpha}))=\nu(\dot{N}_\alpha)$ for $\alpha<\omega_1^V$,
%\item $\bar{\nu}(\bar{k}(\dot{G}_{\alpha}))=\nu(\dot{G}_\alpha)$ for $\alpha<\omega_1^V$,
%\item $\bar{\nu}(\bar{k}(\dot{j}_{\alpha\beta}))=\nu(\dot{j}_{\alpha\beta})$ for $\alpha\leq\beta<\omega_1^V$,
%\item $\bar{\nu}(\bar{k}(\dot{j}_{\alpha\omega_1^V}))=\bar{j}_{\alpha\rho}$ for $\alpha<\omega_1^V$,
%\item $\bar{\nu}(\bar{k}(\dot{N}_{\omega_1}))=N_\rho$,
%\item $\bar{\nu}(\dot{Q})=\bar{k}(H_\kappa)$,

\item $\bar{\nu}(\dot{K})=K^*=K_H\cup\bp{\omega_1^V}$,

\item $\bar{\nu}(\dot{X}_i)=\bar{k}[F_1(i)]$ for all $i\in K_H$,

\item $\bar{\nu}(\dot{X}_{\omega_1^V})=\bar{k}[Q_\lambda]$,

%\item $\bar{\nu}(\dot{Z}_i)=\bar{k}(F_2(i))=\bar{k}(P_{F_0(i)})$ for all $i\in K_H$,

\item $\bar{\nu}(\dot{Z}_{\omega_1^V})=\bar{k}(P_\lambda)$,

\item $\bar{\nu}(\dot{F}_0)(\alpha)=\bar{k}((F_H)_0(\alpha))$ for $\alpha\in K_H$,

\item $\bar{\nu}(\dot{F}_0)(\omega_1^V)=\bar{k}(\lambda)$,

\item $\bar{\nu}(\dot{F}_0)(\beta)=0=\bar{\nu}(\dot{Z}_\beta)$, $\bar{\nu}(\dot{X}_\beta)=1$ if $\beta\notin K^*$,

\item $\bar{\nu}(\dot{r})=r_H$,

\item $\bar{\nu}(\dot{f})=\bar{k}[f_H]=f^*$,

\item \label{rulelastnu*}
$\bar{\nu}(\dot{T})=\bar{k}(T)$,

\item
$\bar{\nu}(\text{br}_T)=\bp{\ap{(r_H\restriction n,f^*\restriction n),r_H,f^*}: n\in\omega}$.

\end{enumerate}
We must therefore check that:
 \begin{enumerate}
 \item\label{eqn:keypropnu*-1}
 $\bar{\nu}$ is well defined (i.e. $\bar{\nu}$ respects
the equivalence class $[\cdot]_H$ on the constant symbols of $\mathcal{L}^1_\lambda$) and is consistent with the 
 constraints set forth in Def. \ref{def:keydefnu} for the symbols in its domain
 in order to be $\bar{\kappa}$-admissible for
$(\bar{J},r_H,f^*,F^*,K^*)$ relative to $\bar{M}$ (the clauses of Def. \ref{def:keydefnuPkappa} must be satisifed).

 \item\label{eqn:keypropnu*-2} 
Any extension of 
$\bar{\nu}$ in $V[G]$ to a total assignment 
$\bar{\nu}:\bar{k}(\mathcal{L}^1)\to X_{(\bar{J},r_H,f^*,F^*,K^*)}$ respecting the constraints
of Def. \ref{def:keydefnuPkappa} witnesses that $(\bar{J},r_H,f^*,F^*,K^*)$
is a weak $\bar{k}(P_\kappa)$-certificate relative to $\bar{M}$, hence $\bar{\nu}$ is also a weak $\bar{k}(P_\kappa)$-witness.
\end{enumerate}

 %Item \ref{eqn:keypropnu*-1} above follows almost immediately using the elementarity of $\bar{k}$.
 The key to establish items \ref{eqn:keypropnu*-1} and \ref{eqn:keypropnu*-2} is the following
 \begin{subclaim}\label{subclm:keyXXX} 
 For $\psi(c_1,\dots,c_n)$ an atomic  $\mathcal{L}^1_\lambda$-sentence 
 \[
 \psi(c_1,\dots,c_n)\in \Sigma_H\text{ if and only if }\mathcal{X}^{\bar{k}(\lambda)}_{(\bar{J},r_H,f^*,F^*,K^*)}\models \psi(\bar{\nu}(c_1),\dots,\bar{\nu}(c_n)).
 \]
 \end{subclaim}
 
 Suppose for the moment the Subclaim is proved and let us complete the proof of the Theorem.  
 By the Subclaim $\bar{\nu}$ can be extended to the whole of $\bar{k}(\mathcal{L}^1_\kappa)$ so to maintain the
 constraints set forth in Def. \ref{def:keydefnuPkappa} to be a $\bar{k}(\kappa)$-assignment
 because these constraints are easily checked for $\bar{\nu}$ on its partial domain. 
 Then all conditions required to make $\bar{\nu}$  a weak $\bar{k}(P_\kappa)$-witness are met, and we are done.
%We first check that 
% $(\bar{J},r_H,f^*,F^*,K^*)$ is a weak $\bar{k}(\kappa)$-precertificate relative to $\bar{M}$. 
%% (the axioms of 
%% $\bar{k}(\Sigma^1_\lambda)$ are satisified by 
%% \[
%% (X_{(\bar{J},r_H,f^*,F^*,K^*)},\in,\bar{\nu}[\bar{k}(\mathcal{L}^1])
%% \] 
%% if $\bar{\nu}$ is a $\bar{k}(\lambda)$-admissible %$\bar{k}(\kappa)$-
%% assignment).
% We give some more details on why this is the case: 
%% the first key point is that the definition
%% of $\bar{\nu}$ is set up in such a way that all the relevant requirements for being a $\bar{k}(\kappa)$-precertificate
%% pass over to $(\bar{J},r_H,f^*,F^*,K^*)$ using $\bar{k}$ from the correponding requirements for 
%% $(\mathcal{J}_H,r_H,f_H,F_H,K_H)$. 
%

So we are left with the proof of Subclaim \ref{subclm:keyXXX}.

\begin{proof}
It is essentially a tiresome matter of checking all possible cases for all atomic sentences of $\mathcal{L}^1_\lambda$, appealing to the fact that $\bar{k}$, $\bar{j_{\omega_1^V\rho}}$ are elementary maps, and the observation
that the cases for predicates with a multi-sorted type are easy to handle.
Let $\mathcal{S}^1$ denote the sorts of $\mathcal{L}^1$ and
define the multimap $(j^*_s:s\in\mathcal{S}^1)$ by
 \[
 j^*_s:\nu(c)\mapsto \bar{\nu}(c)
 \] for $c$ a constant symbol of 
 $\mathcal{L}^1$ of sort $s$.
Now:
\begin{enumerate}[(i)]
\item
Sorts are mapped to sorts: 
\begin{itemize}
\item
$\bar{\nu}(\dot{Q})=\bar{k}(Q_\lambda)=\bar{k}(\nu(\dot{Q}))$,
\item
$\bar{\nu}(\dot{N}_{\omega_1^V})=\bar{j}_{\omega_1^V\rho}(N_{\omega_1^V})=
\bar{j_{\omega_1^V\rho}}(\nu(\dot{N}_{\omega_1^V}))$,
\item
$\bar{\nu}(\dot{\bp{r,f}})=\bp{r,\bar{j}_{\omega_1^V\rho}[f]}$,
\item
$\bar{\nu}(s)=\nu(s)$
for all other sorts in $\mathcal{S}^1$.
\end{itemize}
\item
 For atomic formulae given by  a predicate $R_j$ whose type $s_j$ has entries all of the same sort, one easily checks the preservation of the formula appealing to the properties of the multi-map $(j^*_s:s\in\mathcal{S}^1)$. 
 For example:
 \begin{itemize}
 \item
 $(F_H)_0(\nu(c))=\nu(d)$ holds if and only if $F^*_0(\bar{k}(\nu(c)))=\bar{k}(\nu(d))$ if and only if
 $F^*_0(\bar{\nu}(c))=\bar{\nu}(d)$ for $\nu(c)\in Q_\lambda$; similarly one handles the case of the formula $\dot{K}(c)$.
 \item
 $\Sat_{0,\omega_1^V}(\gp{\psi(\nu(c_1),\dots,\nu(c_m))})$ holds if and only if
 \[
 (Q_\lambda,\in,P_\lambda,A_\lambda)\models\psi(\nu(c_1),\dots,\nu(c_m))
 \]
 if and only if 
 \[
 (\bar{k}(Q_\lambda),\in,\bar{k}(P_\lambda),\bar{k}(A_\lambda))\models\psi(\bar{k}(\nu(c_1)),\dots,
 \bar{k}(\nu(c_m)))
 \]
 if and only if 
  $\Sat_{0,\rho}(\gp{\psi(\bar{\nu}(c_1),\dots,\bar{\nu}(c_m))})$ holds,
given that $\bar{\nu}(c)=\bar{k}(\nu(c))$ for 
  all $c$ of sort $\dot{Q}$ and $\bar{k}$ is elementary.
  \item
  $\Sat_{N_{\omega_1^V}}(\gp{\psi(\nu(c_1),\dots,\nu(c_m))})$ holds if and only if
 \[
 (N_{\omega_1^V},\in)\models\psi(\nu(c_1),\dots,\nu(c_m))
 \]
 if and only if 
 \[
 (N_{\rho},\in)\models\psi(\bar{j}_{\omega_1^V\rho}(\nu(c_1)),\dots,
 \bar{j}_{\omega_1^V\rho}(\nu(c_m)))
 \]
 if and only if 
  $\Sat_{N_\rho}(\gp{\psi(\bar{\nu}(c_1),\dots,\bar{\nu}(c_m))})$ holds,
  given that $\bar{\nu}(c)=\bar{j}_{\omega_1^V\rho}(\nu(c))$ for 
  all $c$ of sort $\dot{N}_{\omega_1^V}$ and $\bar{j}_{\omega_1^V\rho}$ is elementary.

 \item For $\alpha<\omega_1^V$
   $\Sat_{N_{\alpha}}(\gp{\psi(\nu(c_1),\dots,\nu(c_m))})$ holds if and only if
 \[
 (N_{\alpha},\in)\models\psi(\nu(c_1),\dots,\nu(c_m))
 \]
 if and only if  
  $\Sat_{N_\alpha}(\gp{\psi(\bar{\nu}(c_1),\dots,\bar{\nu}(c_m))})$ holds, given that $\bar{\nu}(c)=\nu(c)$ for 
  all $c$ of sort $\dot{N}_\alpha$.
 \item
 We leave to the reader to handle the other atomic formulae whose type has entries of just one fixed sort.
 \end{itemize}
 \item
It remains to handle the case of formulae whose predicate has a multi-sorted type. There are just
three clauses in Def. \ref{def:defL0} and Def. \ref{def:defL1} defining them:
\begin{enumerate}[(a)]
\item $\dot{j}_{\alpha\beta}(c)=d$ (Clause \ref{def:defL0-i}),
\item $\text{br}_T(c,d,e)$ (Clause \ref{def:defL0-j}),
\item $\Cod^*(c,d,e)$ (Clause \ref{def:defL0-l}),
\end{enumerate}
We handle them as follows:
\begin{enumerate}[(a)]
\item
We observe that:
\begin{itemize}
\item if $\beta<\omega_1^V$,
$\nu(c)=\bar{\nu}(c)$ and $\nu(d)=\bar{\nu}(d)$,
\item if $\beta=\omega_1^V$,
$\nu(c)=\bar{\nu}(c)$ and $j_{\omega_1^V\rho}(\nu(d))=\bar{\nu}(d)$.
\end{itemize}
It is immediate to check that this type of formulae is preserved in both cases by the multi-map $(j_{\dot{N}_\alpha}^*,j_{\dot{N}_\beta}^*)$.
\item
It is immediate to check their preservation 
observing that $\bar{\nu}(\dot{f})=\bar{j}_{\omega_1^V\rho}[f]$.
\item
It is immediate to check their preservation
observing that $\bar{\nu}(\dot{r})=r=\nu(\dot{r})$ and
$\bar{\nu}(c)=\nu(c)$ for all constants $c$ of sort $\dot{N}_0$.
\end{enumerate}
\end{enumerate}
The remaining details are left to the reader.
\end{proof}

We now closed all open ends of the proof and 
Theorem \ref{thm:mainthmPkappa} is proved.
\end{proof}

%
%Briefly the heart of the matter is to check that the maps  can be glued in a multi-map which is an embedding of the $\mathcal{L}^1$-structure
% $\mathcal{X}^{\lambda}_{(\mathcal{J}_H,r_H,f_H,F_H,K_H)}$ into 
% $\mathcal{X}^{k(\lambda)}_{(\bar{\mathcal{J}},r_H,f^*,F^*,K^*)}\restriction\bar{k}[\mathcal{L}^1_\lambda]$.
% This amounts to check that sorts are mapped to sorts and that the atomic formulae are preserved, i.e.
% that the following holds:
% \begin{quote}
% For all atomic formuale $R(c_1,\dots,c_n)$ of the multi-sorted signature $\mathcal{L}^1$,
% \[
% \mathcal{X}^{\lambda}_{(\mathcal{J}_H,r_H,f_H,F_H,K_H)}\models R(\nu(c_1),\dots,\nu(c_n))
% \]
% if and only if 
%  \[
% \mathcal{X}^{\bar{k}(\lambda)}_{(\bar{\mathcal{J}},r_H,f^*,F^*,K^)}\models R(\bar{\nu}(c_1),\dots,\bar{\nu}(c_n)).
% \]
% \end{quote}
% 
% \end{proof}

\begin{remark}
Subclaim \ref{subclm:keyXXX} is the very reason why one has to resort to multi-sorted logic to formalize the notion of $P_\lambda$-certificate in
infinitary logic. Consider the first order structures $(\mathcal{X}^\lambda_{(\mathcal{J}_H,r_H,f_H,F_H,K_H)})^*$
 and $(\mathcal{X}_{(\bar{\mathcal{J}},r_H,f^*,F^*,K^*)}^{\bar{k}(\lambda})^*$ associated respectively to 
 $\mathcal{X}_{(\mathcal{J}_H,r_H,f_H,F_H,K_H)}$
 and $\mathcal{X}_{(\bar{\mathcal{J}},r_H,f^*,F^*,K^*)}$.
 In this case there is overlap betwen  the domains of the various sorts $s$ (for example $\nu(\dot{Q})\cap\nu(\dot{N}_{\omega_1^V})\supseteq H_{\omega_2}^V$), and these domains could be mapped in possibly different way by the corresponding
 maps $j^*_s$. For example if $x\in N_{\omega_1^V}\setminus H_{\omega_2}^V$ it is possible that 
 $x=\nu(\check{\check{x}})$ and $x=\nu(c_{i,\omega_1^V})$, while it is a priori conceivable that
 \[
 j^*_{\dot{Q}}(\check{\check{x}})=\bar{k}(x)\neq \bar{j}_{\omega_1^V\rho}(x)=j^*_{\dot{N}_{\omega_1}}(c_{i,\omega_1^V}).
 \] 
 But this potential conflict is avoided exactly because we resorted to multi-sorted logic: the critical first order 
 formula $\check{\check{x}}=c_{i,\omega_1^V}$ is not among the ones whose preservation should be checked in 
 Subclaim \ref{subclm:keyXXX} exactly
 because it is instantiated for constants of different sort.
 \end{remark}

\section{Comparing the proofs}
The reader familiar with Asper\'o and Schindler's proof will notice that our proof is just rephrasing 
their argument taking advantage of the relation between density argument and the use of 
consistency properties to 
produce models of sentences for infinitary multi-sorted logic. 
%However there is a major complication in our proof which does not appear in theirs (i.e. the proof of Subclaim \ref{subclm:keyXXX}): the corresponding argument they give for the proof
%of Claim \ref{clm:fundclmrproof} is very brief (see the short paragraph after Equation (23) on page 28 of \cite{ASPSCH(*)}). 
%Why they can avoid this terrible and tiresome case analysis?
%I believe the key point is that their term model does not satisfy the equality axioms we have been accepting so liberally; in particular the formula $c=d$ is admitted only for certain type of constants for which the proof-check 
%is much easier. 
%However the prize to pay by them is that the definition of the term model induced by a $V$-generic filter for 
%$P_\kappa$ is much more intricate and the
%$\in$-structure of its transitive collapse is not so simple to describe as it is in our case. 
A small twist with respect to their proof is also given by our use of the apparatus set forth in \ref{subsec:consprop}  
to establish many properties of the generic extension by 
$P_\lambda$ by means of a syntactic analysis of the $\mathfrak{L}_{\infty,\omega}$-theory of the $P_\lambda$-certificates.
This makes a large swath of density arguments easy consequences of Fact \ref{fac:SigmaAconsprop}. 

Taken aside these considerations, the proof presented here rephrases with a different terminology 
the arguments given in
\cite{ASPSCH(*)}.

\section{On the consistency strength of $(*)$}

This part is joint work with Ralf Schindler.

For $P$ a poset of regular uncountable size $\kappa$
$\FA^+(P)$ states that for any $P$-name $\dot{T}$ for a stationary subset of $\kappa$,
there is $G$ filter such that 
\[
\bp{\alpha<\kappa:\, \exists p\in G\, p\Vdash \check{\alpha}\in\dot{T}}
\] 
is stationary in its supremum.

\begin{theorem}[Schindler, V.]
Assume $\kappa=\theta^+=2^\theta$ with $\theta>\omega_1$.

Let $\mathcal{L}^{2}=\mathcal{L}^1\cup \bp{\dot{R}}$ with
$\dot{R}$ a binary predicate symbol of type $(\dot{Q},\dot{Q})$.

Let $\Sigma^2_\lambda$ enlarge $\Sigma^{1}_\lambda$ with the axioms\footnote{These axioms states that for $\delta\in K$
$R(\delta,\eta)$ holds if and only if $\eta$ is the supremum of $X_\delta\cap\Ord$.}
\begin{enumerate}[(I)]
\setcounter{enumi}{5}
\item
\emph{}

\begin{enumerate}
\item For all $\delta<\omega_1$ and $\eta<\lambda$,
\[
\qp{\check{\delta}\in\dot{K} \wedge \dot{R}(\check{\delta},\check{\eta})}\rightarrow 
\bigwedge_{\gamma<\lambda}
\qp{(\dot{X}_\delta(\check{\gamma})\rightarrow \check{\gamma}<\check{\eta})\wedge
(\check{\gamma}<\check{\eta}\rightarrow \bigvee_{n<\omega} \check{\gamma}\in \dot{e}_{n,\delta})}
\]
\end{enumerate}

\end{enumerate}

Let $\bar{P}_\lambda$ be exactly the forcing defined by $P_\lambda$ with the additon that 
a $\bar{P}_\eta$-witness
is obtained by an $\eta$-precertificate $\nu$ as the unique extension of $\nu$ to
a  $\bar{\nu}:\mathcal{L}^2_\lambda\to \mathcal{X}^\lambda_{(\mathcal{J},r,f,F,K)}$ 
such that: 
\begin{itemize}
\item
$\Sigma^2_\lambda\subseteq\Sigma_{\bar{\nu}}$,
\item
Condition \ref{def:keydefP0-3}\ref{def:keydefP0-3-c} is satisfied, now with
with $\bar{P}_\eta$ replacing $P_\eta$.
\end{itemize}

Assume $\FA^+_{\aleph_1}(\bar{P}_\kappa)$. 
Then $\square_\kappa$ fails.
\end{theorem}

Note that the unique
change from $P_\lambda$ to $\bar{P}_\lambda$ is that in the new forcing it is possible to express by an atomic formula the 
statement $\sup(X_\delta)=\eta$ for any $\delta<\omega_1^V,\eta<\lambda$. On the other hand it is not clear whether
the inclusion map of $P_\lambda$ into $\bar{P}_\lambda$ is a complete embedding:
if $r\in \bar{P}_\lambda$, $D$ is dense in $P_\lambda$, 
and $s\in D$ is an extension of $r\restriction P_\eta=r\cap \SL{\eta}$, the unique extension of a
$P_\eta$-witness $\nu$ for $s$ to an assignment $\bar{\nu}$ with $\Sigma_{\bar{\nu}}\supseteq \Sigma^2_\eta$
may not satisfy all the atomic sentences for $s\cup r$ or (even if  it satisfied all these atomic sentences) it might 
not be a $\bar{P}_\eta$-certificate.

\begin{proof}
Under our assumptions on $\kappa$, Shelah has proved that
$\Diamond_{\kappa}$ is witnessed by the set of points of countable cofinality below $\kappa$
(see \cite{MR2567499}).

Now the proof is just a slight refinement of the proof of Thm. \ref{thm:mainthmPkappa}.
%We assume $2^{\aleph_2}=\aleph_3=\kappa$ and $2^{\aleph_1}=\aleph_2$. 

Assume towards a contradiction that $(C_\alpha:\alpha<\kappa)$ is a square sequence, 
i.e. is such that $\otp(C_\alpha)\leq\theta$ for all $\alpha$, and $c_\alpha=c_\beta\cap\alpha$ whenever
$\alpha$ is a limit point of $c_\beta$.
By a standard diagonalization argument we can find $\xi<\theta$ such that
$\Diamond_\kappa$ is witnessed by $S_\xi$, where
\[
S_\xi=\bp{\eta<\kappa:\, \cof(\eta)=\omega\text{ and }\otp(C_\eta)=\xi}.
\]

If one parses through that proof one realizes that the following holds:
\begin{claim} 
Given $p\in \bar{P}_\kappa$ and a
$\bar{P}_\kappa$-name 
$\dot{C}$ for a club subset of $\kappa$, there are
$\delta<\omega_1$ and $\lambda\in S_\xi$ such that
\[
p\cup\bp{\dot{K}(\check{\delta}), \dot{F}_0(\check{\delta})=\check{\lambda},\dot{R}(\check{\delta},\check{\lambda})}
\]
is in $\bar{P}_\kappa$ and forces $\check{\lambda}\in\dot{C}$.
\end{claim}

The proof of this claim is (modulo the variations spelled out below) 
exactly the one given in the proof of 
Thm. \ref{thm:mainthmPkappa} to infer that on top of any condition $p\in P_\kappa$ and for any $\dot{e}_{i,\omega_1}$
which $p$ forces to be stationary for $N_{\omega_1}$ and for any
$P_\kappa$-name 
$\dot{C}$ for a club subset of $\omega_1$ one can always
find $\delta<\omega_1$ and $\lambda$ such that
\[
p\cup\bp{\dot{K}(\check{\delta}), \dot{F}_0(\check{\delta})=\check{\lambda}}
\]
is in $P_\kappa$ and $p$ forces $\check{\delta}\in \dot{e}_{i,\omega_1}\cap\dot{C}$.

The key changes are the following:
\begin{itemize}
\item A $\bar{P}_\kappa$-name for a subset of $\kappa$ is still given by a subset of $H_\kappa$ (now we need 
$\kappa$-many antichains of $\bar{P}_\kappa$ to decide $\dot{C}$), hence we can find
$\lambda\in S_\xi$ such that $\dot{C}\cap Q_\lambda=A_\lambda$.
\item
Claim \ref{clm:fundclmrproof} is now reinforced adding the further request that 
\[
\bar{\nu}:\bar{k}(\mathcal{L}^{1}_\kappa)\to X_{(\bar{\mathcal{J}},r_H,f^*,F^*,K^*)}
\] 
is such that its unique extension to a model of $\bar{k}(\Sigma^2_\lambda)$
is such that
$\dot{R}(\check{\omega_1^V},\bar{k}(\check{\lambda}))\in\Sigma_{\bar{\nu}}$.

The proof that this stronger form of the claim holds is exactly the same, with the following further addition:
since $\lambda$ has countable cofinality in $V$,
$\bar{k}(\lambda)$ has countable cofinality in $\bar{M}$, hence $V[G]$ models that
$\bar{k}[\lambda]$ is cofinal in $\bar{k}(\lambda)$. This gives
that $\dot{R}(\omega_1^V,\bar{k}(\lambda))\in\Sigma_{\bar{\nu}}$ assuming the $\bar{\nu}$ given in 
Claim \ref{clm:fundclmrproof}
has been 
extended to $\bar{k}(\mathcal{L}^2_\lambda)$ in the unique possible way which makes it 
a model of $\bar{k}(\Sigma^2_\lambda)$. 
\item
The same definability argument given in the proof of the main theorem yields that 
\[
\bar{k}(p)\cup\bp{\dot{K}(\check{\omega}_1),\dot{F}_0(\check{\omega}_1)=\bar{k}(\check{\lambda}),
\dot{R}(\check{\omega}_1,\bar{k}(\check{\lambda}))}
\]
forces that $k(\dot{\lambda})\in k(\dot{C})$ holds in $\bar{M}$ for $\bar{k}(P_\kappa)$.
\end{itemize}

This gives that
\begin{claim}
$\dot{S}$ is a $\bar{P}_\kappa$-name for a stationary set, where
\[
\dot{S}=\bp{\ap{\check{\lambda},p}:\, \lambda\in S_\xi,\text{ and }\, \dot{K}(\check{\delta}),\,(\dot{F}_0(\check{\delta})=\check{\lambda}),\, \dot{R}(\check{\delta},\check{\lambda})\in p\text{ for some }\delta\in \omega_1^V}.
\]
\end{claim}

Now by $\FA^{+1}(\bar{P}_\kappa)$ we can find $H$ filter on $\bar{P}_\kappa$ such that
\[
S=\dot{S}_H=\bp{\lambda\in S_\xi: \,\text{exists }p\in H \text{ forcing }\check{\lambda}\in\dot{S}}
\]
is stationary in its supremum $\eta$.
Then $S\cap C_\eta$ is stationary. If $\alpha<\gamma\in S\cap C_\eta$ are limit points of $C_\eta$, we get that:
\[
C_\alpha=C_\eta\cap\alpha=C_\eta\cap\gamma\cap\alpha=C_\gamma\cap\alpha
\]
hence 
\[
\xi=\otp(C_\alpha)=\otp(C_\gamma\cap\alpha)<\otp(C_\gamma)=\xi,
\]
which is a contradiction.

\end{proof}

\bibliographystyle{plain}
	\bibliography{Biblio}

\end{document}